%% file: main.tex
\documentclass[11pt]{amsart}
\usepackage{setspace}
\usepackage[english]{babel}
\usepackage{bbold}
\usepackage{stmaryrd}
\usepackage{verbatim}
\usepackage[utf8x]{inputenc}
\usepackage{amsmath}
\numberwithin{equation}{subsection}
\usepackage{amssymb}
\usepackage{amsthm}
\usepackage{graphicx}
\usepackage{mathtools}
\usepackage[T1]{fontenc}
\usepackage{amsthm}
\usepackage{setspace}
\usepackage{amsfonts}
\usepackage{eucal}

\usepackage{mathrsfs}
\usepackage{textcomp}
\usepackage{pictexwd,dcpic}
\usepackage{tikz}
\usepackage{tqft}
\usepackage{pigpen}
\usepackage{xcolor}
\usepackage{pdfpages}
\usepackage{hyperref}

\newtheorem*{theorem-non}{Theorem}
\newtheorem*{theorem-nonf}{Th\'eor\`eme}
\newtheorem*{corollary-non}{Corollary}
\newtheorem*{corollary-nonf}{Corollaire}

\newtheorem{lmm}[subsection]{Lemma}
\newtheorem{thm}[subsection]{Theorem}
\newtheorem{prop}[subsection]{Proposition}
\newtheorem{cor}[subsection]{Corollary}

\theoremstyle{definition}
\newtheorem{defn}[subsection]{Definition}

\newtheorem{ex}[subsection]{Example}
\newtheorem*{ex-non}{Example}
\newtheorem{rmk}[subsection]{Remark}

\newtheorem*{rmk-non}{Remark}
\newtheorem*{rmk-nonf}{Remarque}

\newtheorem{notation}[subsection]{Notation}
\newtheorem{construction}[subsection]{Construction}
\title{}
\date{ }

\newcommand{\x}[0]{\times}

\newcommand{\Ac}[0]{\mathcal{A}}

\newcommand{\Cc}[0]{\mathcal{C}}
\newcommand{\Dc}[0]{\mathcal{D}}
\newcommand{\Ec}[0]{\mathcal{E}}
\newcommand{\Fc}[0]{\mathcal{F}}
\newcommand{\Gc}[0]{\mathcal{G}}

\newcommand{\Lc}[0]{\mathcal{L}}
\newcommand{\Mcc}[0]{\mathcal{M}}

\newcommand{\Oc}[0]{\mathcal{O}}

\newcommand{\Uc}[0]{\mathcal{U}}

\newcommand{\Vs}[0]{\mathscr{V}}
\newcommand{\Cb}[0]{\mathbb{C}}
\newcommand{\Db}[0]{\mathbb{D}}

\newcommand{\N}[0]{\mathbb{N}}

\newcommand{\V}[0]{\mathbb{V}}
\newcommand{\Z}[0]{\mathbb{Z}}

\newcommand{\Hu}[0]{\textup{H}}
\newcommand{\Vu}[0]{\textup{V}}

\newcommand{\Map}[0]{\textup{Map}}

\newcommand{\oo}{\infty}

\newcommand{\Hom}[0]{\textup{\underline{Hom}}}
\newcommand{\Fun}[0]{\textup{Fun}}
\newcommand{\sLambda}{\underline{\Lambda}}
\newcommand{\CAlg}{\textup{CAlg}}
\newcommand{\Ql}[1]{\mathbb{Q}_{\ell#1}}
\newcommand{\Det}[1]{\mathcal{D}_{\text{\'et}}(#1)}
\newcommand{\aff}[2]{\mathbb{A}^{#1}_{#2}}

\newcommand{\sch}[0]{\textup{\textbf{Sch}}_{S}}

\newcommand{\DM}[1]{\textup{\textbf{DM}}_{#1}}
\newcommand{\Ar}[1]{\textup{\textbf{Art}}_{#1}}

\newcommand{\ag}[1]{[\mathbb{A}^1_{#1}/\mathbb{G}_{\textup{m},#1}]}
\newcommand{\gm}[1]{\mathbb{G}_{\textup{m},#1}}

\newcommand{\bgm}[1]{\textup{B}\mathbb{G}_{\textup{m},#1}}
\newcommand{\muinf}[1]{\mu_{\infty,#1}}
\newcommand{\prs}[0]{\textup{\textbf{Pr}}^{\textup{L}}_{\textup{stb}}}
\newcommand{\prso}[0]{\textup{\textbf{Pr}}^{\textup{L},\otimes}_{\textup{stb}}}
\newcommand{\prsr}[0]{\textup{\textbf{Pr}}^{\textup{R}}_{\textup{stb}}}
\newcommand{\prsro}[0]{\textup{\textbf{Pr}}^{\textup{R},\otimes}_{\textup{stb}}}

\newcommand{\shv}[1]{\textup{\textbf{Shv}}(#1)}

\title{\parbox{\linewidth}{\centering \'Etale tame vanishing cycles over \texorpdfstring{$\ag{S}$}{ag}}}
\date{}
\author{Denis-Charles Cisinski, Massimo Pippi}
\email{denis-charles.cisinski@ur.de}
\address{Fakultät für Mathematik
Universität Regensburg
93040 Regensburg
Germany}

\email{massimo.pippi@ur.de}
\address{Fakultät für Mathematik
Universität Regensburg
93040 Regensburg
Germany}
\begin{document}

\begin{abstract}
  We develop a theory of tame vanishing cycles for schemes over $\ag{S}$ in the context of \'etale sheaves. We show some desired properties of this formalism, among which: a compatibility with tame vanishing cycles over a (strctly) henselian trait, a compatibility with the theory of tame vanishing cycles over $\aff{1}{S}$, a compatibility with tensor product and with duality. In the last section, we prove that monodromy-invariant vanishing cycles, introduced by the second named author, are the homotopy fixed points with respect to a canonical continuous action of $\mu_{\oo}$ of tame vanishing cycles over $\ag{S}$.
\end{abstract}
\maketitle
\tableofcontents

\section*{Introduction}
\input{sections/introduction}

\section{\'Etale sheaves on Artin stacks}\label{etale sheaves on Artin stacks}
\input{sections/etale_sheaves_on_Artin_stacks}

\section{Recollements}\label{section recollements}
\input{sections/recollements}

\section{Review of the theory of tame vanishing cycles}\label{review of the theory of tame vanishing cycles}
\input{sections/review_of_the_theory_of_tame_vanishing_cycles}

\section{Tame nearby and vanishing cycles over \texorpdfstring{$\ag{S}$}{ag}}\label{tame nearby and vanishing cycles over ag}
\input{sections/tame_vanishing_cycles_over_ag}

\section{Functorial behaviour of tame nearby and vanishing cycles over \texorpdfstring{$\ag{S}$}{ag}}\label{functorial behaviour}
\input{sections/functorial_behaviour}

\section{Comparison with vanishing cycles over a strictly henselian trait}\label{comparison with vanishing cycles over a strictly henselian trait}

\input{sections/comparison_with_usual_tame_vanishing_cycles}

\section{Tame nearby cycles over \texorpdfstring{$\aff{1}{S}$}{A1} and comparison with the étale version of Ayoub's tame nearby cycles}\label{tame nearby cycles over A1 and comparison with the etale version of Ayoub's tame nearby cycles}
\input{sections/tame_nearby_cycles_over_A1_and_comparison_with_Ayoub}

\section{Compatibility with tensor product and duality}\label{compatibility with tensor produc and duality}
\input{sections/compatibility_with_tensor_product_and_duality}

\section{Tame vanishing cycles over \texorpdfstring{$\ag{S}$}{ag} and monodromy invariant vanishing cycles}\label{comparison with monodromy invariant vanishing cycles}
\input{sections/comparison_with_monodromy_invariant_vanishing_cycles}
\appendix
\section{}\label{appendix A}
\input{sections/appendix_A}

%\bibliographystyle{plain}
%\bibliography{sections/references}

\end{document}

%% file: sections/introduction.tex
The theory of vanishing cycles for the germ of an holomorphic function $(\mathbb{C}^n,\underline{0})\rightarrow (\mathbb{C},0)$ was introduced by J.~Milnor in his celebrated book \cite{m68}. 
The theory was developed in the algebraic setup by A.~Grothendieck (\cite{sga7i}) and by P.~Deligne (\cite{sga7ii}). 
In the latter context, the role of the germ of an holomorphic function is played by a scheme over an henselian trait. 
The henselian trait here plays the role of a small disk centered at the origin. 
A variant of the theory of (tame) nearby cycles, developed by J.~Ayoub in his monography \cite{ay07a,ay07b} in the motivic context (more generally, in the context of homotopy stable $2$-functors), replaces the small little disk centered at the origin with $\aff{1}{S}$. 
Moreover, Ayoub shows that these two formalisms are compatible in a suitable sense in \cite{ay14}. 
Since $\aff{1}{S}$ is naturally equipped with an action of the group scheme $\gm{S}$, it is natural to seek for a theory of (tame) nearby/vanishing cycles for $\gm{S}$-\emph{equivariant functions} $X\rightarrow \aff{1}{S}$.
This is precisely what we aim to do in this note. 
We achieve this by using the language of Artin stacks: the datum of a $\gm{S}$-equivariant function $f:X\rightarrow \aff{1}{S}$ is equivalent to that of a morphism $[X/\gm{S}]\rightarrow \ag{S}$ of stacks. 
In this setup, we let $\ag{S}$ play the role of a little disk centered at the origin. 
Then, the role of the origin of this little disk is played by the closed substack $\bgm{S}$ of $\ag{S}$, while that of the punctured disk by $S\simeq [\aff{1}{S}-\{0\}/\gm{S}]$.
The following situation, which was the one of interest in \cite{p20}, might help to clarify the geometric intuition underlying our theory. 
Let $X$ be an $S$-scheme. To provide a morphism $X\rightarrow \ag{S}$ is equivalent to provide a pair $(\Lc,s)$, where $\Lc$ is a line bundle over $X$ and $s$ is a global section of $\Lc$. 
Then, the fiber of $X$ over $\bgm{S}$ is the zero locus of $s$ (i.e. the closed subscheme defined by the ideal sheaf generated by the morphism $\Lc^{\vee}\rightarrow \Oc_X$ induced by $s$), while its fiber over $S$ is the open complementary to this closed subscheme.
In the theory developed in \cite{sga7i,sga7ii}, a crucial role is played by (the choice of) a seperable closure of the fraction field of the base henselian ring. 
In the variant where the base is $\aff{1}{S}$, this role is played by the pro-\'etale cover $\widetilde{\gm{S}}\rightarrow \gm{S}$, where $\widetilde{\gm{S}}=\varprojlim_{n\in \N_S}(\gm{S}\xrightarrow{t\mapsto t^{n}}\gm{S})$, where $t$ is the parameter of $\gm{S}$.
These are the geometric analogues of an universal cover of the punctured disk in the topological setting.
Since in our context the punctured disk is "topologically trivial", this universal cover has to be introduced in an alternative way.
The reader might benefit by considering the following topological analogue of our setup: let $X$ be a complex variety endowed with an action of $\mathbb{C}^*$ and let $f:X\rightarrow \mathbb{C}$ be a $\mathbb{C}^*$-equivariant function.
We can consider the following diagram, where all squares are Cartesian:
\begin{equation}
    \begindc{\commdiag}[15]
      \obj(-40,15)[1]{$f^{-1}(0)$}
      \obj(0,15)[2]{$X$}
      \obj(40,15)[3]{$X^*$}
      \obj(80,15)[4]{$\widetilde{X}^*$}
      \obj(-40,-15)[5]{$\{0\}$}
      \obj(0,-15)[6]{$\mathbb{C}$}
      \obj(40,-15)[7]{$\mathbb{C}^*$}
      \obj(80,-15)[8]{$\mathbb{C}.$}
      \mor{1}{2}{$i$}
      \mor{3}{2}{$j$}
      \mor{4}{3}{$$}
      \mor{5}{6}{$i$}
      \mor{7}{6}{$j$}
      \mor{8}{7}{$exp$}
      \mor{1}{5}{$$}
      \mor{2}{6}{$f$}
      \mor{3}{7}{$f_{|X^*}$}
      \mor{4}{8}{$$}
    \enddc
\end{equation}

The two squares on the left form a $\mathbb{C}^*$-equivariant diagram in a natural way. 
The exponential map $exp:\mathbb{C}\rightarrow \mathbb{C}^*$ is a morphism of groups and we let it play the role of the universal cover of the punctured disk. 
Notice that we can regard the whole diagram as being $\mathbb{C}$-equivariant thanks to the exponential map. This is also how the \emph{monodromy action} arises in this topological setup. 
\begin{rmk}
The precise definition of $\mathbb{C}^*$-equivariant vanishing cycles and its relationship with the formalism introduced here will appear elsewhere.
\end{rmk}
Clearly, in the algebraic setup we do not dispose of the exponential map. Therefore, we approximate it with the tower of morphisms $\gm{S}\xrightarrow{t\mapsto t^n}\gm{S}$. In other words, if $f:X\rightarrow \aff{1}{S}$ is an equivariant map of $\gm{S}$-schemes, we look at the tower of equivariant diagrams
\begin{equation}
    \begindc{\commdiag}[15]
      \obj(-40,15)[1]{$X_0$}
      \obj(0,15)[2]{$X$}
      \obj(40,15)[3]{$U_X$}
      \obj(80,15)[4]{$U_X^{(n)}$}
      \obj(-40,-15)[5]{$S$}
      \obj(0,-15)[6]{$\aff{1}{S}$}
      \obj(40,-15)[7]{$\gm{S}$}
      \obj(80,-15)[8]{$\gm{S}.$}
      \mor{1}{2}{$i$}
      \mor{3}{2}{$j$}
      \mor{4}{3}{$$}
      \mor{5}{6}{$i$}
      \mor{7}{6}{$j$}
      \mor{8}{7}{$(-)^n$}
      \mor{1}{5}{$$}
      \mor{2}{6}{$f$}
      \mor{3}{7}{$f_{|U_X}$}
      \mor{4}{8}{$$}
    \enddc
\end{equation}

However, $\gm{S}$ acts via $\gm{S}\xrightarrow{(-)^n}\gm{S}$ on the two left squares of this diagram. Therefore, taking the limit over all $n$'s, we obtain a $\widetilde{\gm{S}}$-equivariant diagram. This is how we obtain the \emph{monodromy action} in this setting.

In the language of stacks, this is obtained by looking at the following diagrams
\begin{equation}
    \begindc{\commdiag}[15]
      \obj(-40,15)[1]{$\bgm{S}$}
      \obj(0,15)[2]{$\ag{S}$}
      \obj(40,15)[3]{$S$}
      \obj(-40,-15)[4]{$\bgm{S}$}
      \obj(0,-15)[5]{$\ag{S}$}
      \obj(40,-15)[6]{$S,$}
      \mor{1}{2}{$i$}
      \mor{3}{2}{$j$}[\atright,\solidarrow]
      \mor{1}{4}{$\Theta_0^{(n)}$}[\atright,\solidarrow]
      \mor{2}{5}{$\Theta^{(n)}$}
      \mor{3}{6}{$id_S$}
      \mor{4}{5}{$i$}[\atright,\solidarrow]
      \mor{6}{5}{$j$}
    \enddc
\end{equation}
where $\Theta^{(n)}$ (resp. $\Theta_0^{(n)}$), at the level of $T$-points, induces the functor $(\Lc,s)\mapsto (\Lc^{\otimes n},s^{\otimes n})$ (resp. $\Lc\mapsto \Lc^{\otimes n}$).

For an Artin stack $X$ endowed with a morphism $p:X\rightarrow \ag{S}$, we introduce three $\oo$-functors
\begin{equation}
    \Psi_{p,\eta}^t:\Det{U_X;\Lambda}\rightarrow \Det{X_0;\Lambda}^{\mu_{\oo}}, 
\end{equation}
\begin{equation}
    \Psi_p^t:\Det{X;\Lambda}\rightarrow \Det{\Vs_p;\Lambda},
\end{equation} 
\begin{equation}
    \Phi_p^t:\Det{X;\Lambda}\rightarrow \Det{X_0;\Lambda}^{\mu_{\oo}},
\end{equation}
that we call respectively: \emph{tame nearby cycles}, \emph{total tame nearby cycles} and \emph{tame vanishing cycles}. Here $\Det{X;\Lambda}$ (resp. $\Det{X_0;\Lambda}$) denotes the stable $\oo$-category of complexes of \'etale sheaves on $X$ (resp. $X_0$). The objects of the stable $\oo$-category $\Det{X_0;\Lambda}^{\mu_{\oo}}$ are, roughly, complexes of \'etale sheaves on $X_0$ endowed with a \emph{continuous} action of $\mu_{\oo}=\varprojlim \mu_n$. Finally, $\Det{\Vs_p;\Lambda}$ is a stable $\oo$-category whose objects are, roughly, $\mu_{\oo}$-equivariant morphisms $F\rightarrow G$, where $F$ is endowed with the trivial action.
\subsection*{Applications and related works} Our formalism is intimately related to that of \emph{monodromy-invariant vanishing cycles}, introduced by the second named author in \cite{p20t,p20}. The precise relationship between these two constructions is spelled out in \S \ref{comparison with monodromy invariant vanishing cycles}. In a forthcoming paper, we will introduce the formalism of tame nearby/vanishing cycles over $\ag{S}$ in the context of \'etale motives (as considered by Ayoub~\cite{ay14}, or equivalently, of h-motives, as developed by the first named author in collaboration with F.~D\'eglise~\cite{cd16}). In fact, the present paper can be interpreted as such a formalism in the case where the algebra of coefficients is torsion (the case of rational coefficients is much more technical: not only because it must take place in the contexct of motivic sheaves, but also because there are obstructions we must deal with that vanish with torsion coefficients). The idea of considering $\gm{S}$-equivariant nearby cycles goes back implicitly at least to Verdier~\cite{Verdier}, in his work on monodromic sheaves that appear through his specialization functor, defined via deformation to the normal cone, and the present article is part of an ongoing series of papers devoted to developing microlocal methods for $\ell$-adic sheaves as well as for motivic sheaves.

\subsection*{Outline of the paper.} The paper is organised as follows:
\begin{itemize}
    \item In section \ref{etale sheaves on Artin stacks} we quickly recollect the formalism of Grothendieck's 6 functors formalism for stable $\oo$-categories of \'etale complexes on (derived) Artin stacks.
    \item In section \ref{section recollements}, we quickly review the theory of recollements in the context of $\oo$-categories as developed by J.~Lurie in \cite{lu17} and we provide certain constructions and results that we were not able to locate in the existing literature. In particular, we provide a method to construct an adjunction between two recollements, starting from the datum of two adjunctions between the subcategories that define them.
    \item In section \ref{review of the theory of tame vanishing cycles} we reinterpret the classical theory of tame vanishing cycles over a strictly henselian trait in the language of (pro) Deligne-Mumford stacks. We decided to include this section to emphatize that the language of stacks is well suited to develop a theory of vanishing cycles.
    \item Section \ref{tame nearby and vanishing cycles over ag} is the core of this paper. Here we construct the three $\oo$-functors $\Psi_{p,\eta}^t$, $\Psi_p^t$ and $\Phi_p^t$ that we mentioned in the introduction. Moreover, in section \ref{functorial behaviour} we exhibit certain natural compatibility properties with the $*$-pullback, $*$-pushforward, $!$-pullback and $!$-pushforward.
    \item Relying on the re-interpretation of the theory of tame nearby cycles over a strictly henselian trait provided in section \ref{review of the theory of tame vanishing cycles}, in section \ref{comparison with vanishing cycles over a strictly henselian trait} we show how our formalism of tame nearby vanishing cycles over $\ag{S}$ generalises the classical theory.
    \item In section \ref{tame nearby cycles over A1 and comparison with the etale version of Ayoub's tame nearby cycles}, we show that if the morphism $X\rightarrow \aff{1}{S}$ factors through $\aff{1}{S}\rightarrow \ag{S}$, we re-find the \'etale version of the formalism of tame nearby cycles over $\aff{1}{S}$ developed by J.~Ayoub by composing our functor $\Psi_{p,\eta}^t$ with the forgetful functor $\Det{X_0;\Lambda}^{\mu_{\oo}}\rightarrow \Det{X_0;\Lambda}$.
    \item Building on the results of section \ref{tame nearby cycles over A1 and comparison with the etale version of Ayoub's tame nearby cycles} and on the results provided by J.~Ayoub for his formalism, in section \ref{compatibility with tensor produc and duality} we show that $\Psi_{p,\eta}^t$ is compatible with tensor products and with duality.
    \item In section \ref{comparison with monodromy invariant vanishing cycles}, we prove that monodromy invariant vanishing cycles introduced by the first named author are nothing but the homotopy fixed points of the $\mu_{\oo}$-action on the sheaf of tame vanishing cycles introduced in section \ref{tame nearby and vanishing cycles over ag}. In pariticular, we show how the first Chern class of the line bundle defining the morphism $X_0\rightarrow \bgm{S}$ plays an important role in the canonical action of $\mu_{\oo}$ on the sheaf of tame vanishing cycles.
\end{itemize}
\subsection*{Notation}
\begin{itemize}
\item $\mathcal{S}$ denotes the $\oo$-category of spaces;
\item for a simplicial set $X$, we denote by $X([n])$ the set of $n$-simplicies of $X$;
\item whenever we will have a diagram $Y\xrightarrow{f}X\xleftarrow{g}Z$ in a ($\oo$-)category with finite limits, we will denote the fiber product by $Y\x_{f,X,g}Z$ if we want to keep track of $f$ and $g$ in the notation;
\item we will denote the $\oo$-category of stable and presentable $\oo$-categories with right adjoints as $\prsr$. Similarly, we denote the $\oo$-category of stable and presentable $\oo$-categories with left adjoints as $\prs$;
\item unless otherwise specified, by $2$-category we will always mean "strict $(2,1)$-category";
\item all functors will be implicitely derived;
\item following \cite{lu09}, we denote the simplicial nerve by $N$.
\end{itemize}
\section*{Acknowledgements}
MP wishes to thank Bertrand To\"en and Gabriele Vezzosi for many conversations on the subject of this paper.

This project has received funding from the European Research Council (ERC) under the European
Union Horizon 2020 research and innovation programme (grant agreement No. 725010).

The authors are supported by the collaborative research center SFB 1085 \emph{Higher Invariants - Interactions between Arithmetic Geometry and Global
Analysis} funded by the Deutsche Forschungsgemeinschaft.

%% file: sections/etale_sheaves_on_Artin_stacks.tex
One of the major achievements of \cite{sga4i,sga4ii, sga4iii} is the construction of \emph{Grothendieck's six functors formalism} on the derived categories of \'etale sheaves of schemes. 
In this paper, it is crucial that this formalism can be extended to Artin stacks. Many papers have been dedicated to this matter, including \cite{lmb00,be03,o07,lo08,lz17}. 
We will briefly recall what is the derived category of an Artin stack and how the assignment $X\mapsto \Det{X,\Lambda}$ is endowed with a six functor formalism following \cite{lz17}. 
Here $\Lambda$ denotes a fixed ring of coefficients, that we suppose to be torsion and such that for all points $s$ of the base scheme $S$, $char(k(s))$ is invertible in $\Lambda$. 
Notice that, in particular, we will use the language of $\oo$-categories: $\Det{X;\Lambda}$ is a presentable, stable $\oo$-category. 
Recall that, for an Artin stack $X$ and a smooth atlas $X'\rightarrow X$ with $X'$ a scheme, $\Det{X;\Lambda}$ is equivalent to the limit of the following diagram
\begin{equation}
    \begindc{\commdiag}[15]
    \obj(-70,0)[1]{$\Det{X';\Lambda}$}
    \obj(0,0)[2]{$\Det{X'\x_XX';\Lambda}$}
    \obj(90,0)[3]{$\Det{X'\x_XX'\x_XX';\Lambda}$}
    \obj(160,0)[4]{$\dots$}
    \mor(-55,2)(-25,2){$$}
    \mor(-55,-2)(-25,-2){$$}
    \mor(25,4)(55,4){$$}
    \mor(25,0)(55,0){$$}
    \mor(25,-4)(55,-4){$$}
    \mor(125,6)(155,6){$$}
    \mor(125,2)(155,2){$$}
    \mor(125,-2)(155,-2){$$}
    \mor(125,-6)(155,-6){$$}
    \enddc
\end{equation}

This construction does not depend on the atlas $X'\rightarrow X$. 

Grothendieck's six functors formalism consists of the following data and formulas:
\begin{itemize}
    \item For every morphisms of Artin stacks $f:Y\rightarrow X$, there is an adjunction
    \begin{equation}
        f^*:\Det{X;\Lambda}\leftrightarrows \Det{Y;\Lambda}:f_*,
    \end{equation}
    where the left adjoint $f^*$ is called \emph{pullback} (or $*$-pullback) and the right adjoint $f_*$ is called \emph{pushforward} (or $*$-pushforward).
    \item For every separated morphism of finite type, there is an adjunction
    \begin{equation}
        f_!:\Det{Y;\Lambda}\leftrightarrows \Det{X;\Lambda}:f^!,
    \end{equation}
    where the left adjoint $f_!$ is called \emph{exceptional pushforward} (or $!$-pushforward) and the right adjoint $f^!$ is called \emph{exceptional pullback} (or $!$-pushforward).
    \item For every $X$, $\Det{X;\Lambda}$ is a closed symmetric monoidal $\oo$-category. The tensor product is denoted by $-\otimes_X -$ and the internal hom by $\Hom_X$.
    \item The assignments $f\mapsto f^*$, $f\mapsto f_*$, $f\mapsto f_!$ and $f\mapsto f^!$ are functorial (in the $\oo$-categorical sense) in $f$. 
    \item There is a natural transformation of $\oo$-functors
    \begin{equation}
        f_!\rightarrow f_*
    \end{equation}
    that is an equivalence whenever $f:X\rightarrow Y$ is representable by Deligne-Mumford stacks, there exists a finite surjection $Z\rightarrow X$ with $Z$ an algebraic space and $f$ is proper.
    \item For every smooth separated morphism $f:X\rightarrow Y$ of relative dimension $d$, there is a natural equivalence of $\oo$-functors
    \begin{equation}
        f^*\simeq f^!(-d)[-2d],
    \end{equation}
    where $(-d)$ denotes the \emph{Tate shift} while $[-2d]$ the \emph{cohomological shift}.
    \item (\textbf{K\"unneth Formula}) For every finite collection of morphisms $\{ f_i:X_i\rightarrow Y_i \}_{\{i=1,\dots n\} }$ of Artin stacks, given a Cartesian square
    \begin{equation}
        \begindc{\commdiag}[15]
          \obj(-35,15)[1]{$X$}
          \obj(35,15)[2]{$X_1\x \dots \x X_n$}
          \obj(-35,-15)[3]{$Y$}
          \obj(35,-15)[4]{$Y_1\x \dots \x Y_n$}
          \mor{1}{2}{$(p_1, \dots, p_n)$}
          \mor{1}{3}{$f$}
          \mor{2}{4}{$f_1 \x \dots \x f_n $}
          \mor{3}{4}{$(q_1, \dots, q_n)$}
        \enddc
    \end{equation}
    the following square commutes (in the $\oo$-categorical sense)
    \begin{equation}
            \begindc{\commdiag}[15]
              \obj(-60,15)[1]{$\Det{X_1;\Lambda}\x \dots \x \Det{X_n;\Lambda}$}
              \obj(60,15)[2]{$\Det{X;\Lambda}$}
              \obj(-60,-15)[3]{$\Det{Y_1;\Lambda}\x \dots \x \Det{Y_n;\Lambda}$}
              \obj(60,-15)[4]{$\Det{Y;\Lambda}.$}
              \mor{1}{2}{$p_1^*\otimes \dots \otimes p_n^*$}
              \mor{1}{3}{$f_{1!}\x \dots \x f_{n!}$}
              \mor{2}{4}{$f_!$}
              \mor{3}{4}{$q_1^*\otimes \dots \otimes q_n^*$}
            \enddc
    \end{equation}
    There are two particularly important cases that deserve a name on their own:
    \begin{enumerate}
        \item (\textbf{Base Change Formulas}) If $n=1$, then we are just considering a Cartesian square
        \begin{equation}
            \begindc{\commdiag}[15]
              \obj(-25,15)[1]{$X'$}
              \obj(25,15)[2]{$Y'$}
              \obj(-25,-15)[3]{$X$}
              \obj(25,-15)[4]{$Y$}
              \mor{1}{2}{$f'$}
              \mor{1}{3}{$g'$}
              \mor{2}{4}{$g$}
              \mor{3}{4}{$f$}
            \enddc
        \end{equation}
        and the K\"unneth Formula implies that we have an equivalence of $\oo$-functors
        \begin{equation}
            f'_!\circ g'^*\simeq g^*\circ f_!.
        \end{equation}
        Notice that, by the uniqueness (up to a contractible space of choices) of right adjoints, this formally implies
        \begin{equation}
            g'_*\circ f'^!\simeq f^!\circ g_*.
        \end{equation}
        In particular, if $f$ is proper, representable by Deligne-Mumford stacks and $X$ admits a finite surjection from an algebraic space, we obtain the \textbf{proper base change formula}
    \begin{equation}
        f'_*\circ g'^*\simeq g^*\circ f_*,
    \end{equation}
    while if $f$ is smooth we obtain the \textbf{smooth base change formula}
    \begin{equation}
        g'_*\circ f'^*\simeq f^*\circ g_*.
    \end{equation}
    \item (\textbf{Projection Formula}) For every morphism $f:X\rightarrow Y$, we can consider the pullback square
    \begin{equation}
        \begindc{\commdiag}[15]
              \obj(-25,15)[1]{$X$}
              \obj(25,15)[2]{$Y$}
              \obj(-25,-15)[3]{$X\x Y$}
              \obj(25,-15)[4]{$Y\x Y,$}
              \mor{1}{2}{$f$}
              \mor{1}{3}{$(id, f)$}
              \mor{2}{4}{$\Delta_Y$}
              \mor{3}{4}{$f\x id$}
            \enddc
    \end{equation}
    in which case the K\"unneth formula translates as the equivalence
    \begin{equation}
        f_!(-\otimes f^*(-))\simeq f_!(-)\otimes -.
    \end{equation}
    \end{enumerate}
    Notice that the Projection Formula immediately implies the following:
    \begin{equation}
        \Hom_Y\bigl (f_!(-),- \bigr )\simeq f_*\Hom_X\bigl ( -,f^!(-)\bigr ),
    \end{equation}
    \begin{equation}
        f^!\Hom_Y\bigl ( -,-\bigr )\simeq \Hom_X\bigl (f^*(-), f^!(-) \bigr ).
    \end{equation}
    \item (\textbf{Localization property}) If $i:Z\rightarrow X$ is a closed embedding and $j:U\rightarrow X$ is its open complementary, there is a fiber-cofiber sequence of $\oo$-functors
    \begin{equation}
        j_!\circ j^!\rightarrow id \rightarrow i_*\circ i^*.
    \end{equation}
    It is known (see \cite[Proposition 2.3.3]{cd19}) that this implies that
    \begin{equation}
        i_*\circ i^!\rightarrow id \rightarrow j_*\circ j^*
    \end{equation}
    is a fiber-cofiber sequence of $\oo$-functors as well.
    \item (\textbf{Absolute Purity}) If $i:Z\rightarrow X$ is a closed embedding of regular stacks of constant codimension $c$, there is an equivalence
    \begin{equation}
        \underline{\Lambda}_Z(-c)[-2c]\simeq i^!\underline{\Lambda}_X.
    \end{equation}
    \item (\textbf{Duality}) Let $f:X\rightarrow S$ be a morphism from an Artin stack to a regular scheme $S$. Put $K_X:=f^!(\underline{\Lambda}_S)$ and consider the $\oo$-functor
    \begin{equation}
        \Db_X:=\Hom_X(-,K_X):\Det{X;\Lambda}^{\textup{op}}\rightarrow \Det{X;\Lambda}.
    \end{equation}
    Then $K_X$ is a \emph{dualizing object} in $\Det{X;\Lambda}$. That is, for every $M\in \Det{X;\Lambda}$, the canonical morphism
    \begin{equation}
        M\rightarrow \Db_X(\Db_X(M))
    \end{equation}
    is an equivalence. Moreover, the \emph{duality functor} $\Db_X$ satisfies the following compatibilities with the six functors. Let $f:X\rightarrow Y$ and let $L,M\in \Det{X;\Lambda}$, $N\in \Det{Y;\Lambda}$:
    \begin{equation}
        \Db_X(L\otimes_X \Db_X(M))\simeq \Hom_X(L,M);
    \end{equation}
    \begin{equation}
        \Db_X(f^*(N))\simeq f^!\Db_Y(N);
    \end{equation}
    \begin{equation}
        f^*\Db_Y(N)\simeq \Db_X(f^!N);
    \end{equation}
    \begin{equation}
        \Db_Y(f_*M)\simeq f_!\Db_X(M);
    \end{equation}
    \begin{equation}
        f_!\Db_X(M)\simeq \Db_Y(f_*M).
    \end{equation}
\end{itemize}
The six functor formalism is expressed by the existence of a lax monoidal $\oo$-functor
\begin{equation}\label{six functor formalism with left adjoints}
    Corr(N(\Ar{S}))^{\otimes}_{F,all}\rightarrow \prso,
\end{equation}
which encodes (in an homotopy coherent way!) all the properties listed above. Here $F$ denotes the class of separated morphisms locally of finite type, while $Corr(\Ar{S})^{\otimes}_{F,all}$ is the symmetric monoidal $\oo$-category denoted by $N(\Ar{S})_{corr: \Ec_{F},all}^{\otimes}$ in \cite{lz17}.
Roughly, this $\oo$-functor sends an Artin stack $X$ to $\Det{X;\Lambda}$ and a correspondence $X\xleftarrow{f} Z \xrightarrow{g} Y$ to
\begin{equation}
    f_!\circ g^*: \Det{Y;\Lambda}\rightarrow \Det{X;\Lambda}.
\end{equation}
We refer to \emph{loc. cit.} for further details. 

Also notice that passing to right adjoints, i.e. through the equivalence $\prso \simeq (\prsro)^{op}$, we get an $\oo$-functor
\begin{equation}
    Corr(N(\Ar{S}))^{\otimes}_{all,F}\rightarrow \prsro
\end{equation}
which sends a correspondence $X\xleftarrow{f} Z \xrightarrow{g} Y$ to 
\begin{equation}
    g_*\circ f^!: \Det{X;\Lambda}\rightarrow \Det{Y;\Lambda}.
\end{equation}

%% file: sections/recollements.tex
In this section we will quickly review the theory of \emph{recollement} of $\oo$-categories. 
We will follow \cite[Appendix A.8]{lu17}. To the authors knowledge, the notion of recollement first appeared in \cite{sga4i,sga4ii,sga4iii} in the context of topoi. All results and constructions of this section are well known (even those that might not appear in the literature).
The example which motivates the theory of recollement is the following one: let $X$ be a topological space. 
Let $j:U\subseteq X$ be an open subset and $i:Z=X-U\subseteq X$ be its closed complementary. 
Let $\Fc \in \shv{X}$ be a sheaf (of sets, say) on $X$. 
It determines the following triplet: $(\Fc_{|Z}:=i^*\Fc, \Fc_{|U}:=j^*\Fc, \Fc_{|Z}\rightarrow i^*j_*\Fc_{|U})$. 
Such triplets can be organized in a category in an evident way, denoted $\bigl (\shv{Z},\shv{U},i^*j_*:\shv{U}\rightarrow \shv{Z}\bigr )$ and the assignment 
\begin{equation*}
          \Fc \mapsto (\Fc_{|Z}, \Fc_{|U}, \Fc_{|Z} \rightarrow i^*j_*\Fc_{|U})
\end{equation*}
determines a functor
\begin{equation*}
          \shv{X}\rightarrow \bigl (\shv{Z},\shv{U},i^*j_*:\shv{U}\rightarrow \shv{Z} \bigr ).
\end{equation*}
Moreover, this functor happens to be an equivalence. For example, we can recover the sheaf $\Fc$ (up to isomorphism) from $(\Fc_{|Z}, \Fc_{|U}, \Fc_{|Z}\rightarrow i^*j_*\Fc_{|U})$ by taking the limit of the diagram
\begin{equation*}
  \begindc{\commdiag}[18]
    \obj(15,15)[1]{$j_*\Fc_{|U}$}
    \obj(-15,-15)[2]{$i_*\Fc_{|Z}$}
    \obj(15,-15)[3]{$i_*i^*j_* \Fc_{|U}.$}
    \mor{1}{3}{$$}
    \mor{2}{3}{$$}
  \enddc
\end{equation*}
The key properties of the phenomenon described above are highlighted in the following 
\begin{defn}{\cite[Definition A.8.1]{lu17}}
Let $\Cc$ be an $\oo$-category with finite limits and let $\Cc_0,\Cc_1\subseteq \Cc$ be two full subcategories. We say that $\Cc$ \emph{is the recollement of} $\Cc_0$ and $\Cc_1$ if the following hypothesis are met: 
\begin{enumerate}
\item $\Cc_0$ and $\Cc_1$ are stable under equivalences;
\item the inclusions $i_0:\Cc_0\hookrightarrow \Cc$ and $i_1:\Cc_1\hookrightarrow \Cc$ admit left adjoints $L_0:\Cc\rightarrow \Cc_0$ and $L_1:\Cc\rightarrow \Cc_1$;
\item $L_0$ and $L_1$ are left exact;
\item $L_1\circ i_0$ is equivalent to the constant functor $\Cc_0\rightarrow \Cc_1$ determined by a final object of $\Cc_1$ (which exists as $\Cc$ has finite limits and $L_1$ is left exact);
\item $(L_0,L_1)$ detects equivalences, i.e. an arrow $\alpha:x\rightarrow y$ in $\Cc$ is an equivalence if and only if $L_0(\alpha)$ and $L_1(\alpha)$ are.
\end{enumerate}
\end{defn}
\begin{ex}
Let $X,U,Z,i,j$ be as above. Let $\shv{X}$ (resp. $\shv{U}$, resp. $\shv{Z}$) be the $\oo$-category of $\mathcal{S}$-valued sheaves on $X$ (resp. $U$, resp $Z$). The functors $i_*$ and $j_*$ are fully faithful and the full subcategories stable under equivalences determined by these functors are such that $\shv{X}$ is a recollement of those.
\end{ex}
\begin{rmk}
It is proved in \cite[Appendix A.8]{lu17} that the theory of recollements is equivalent to that of left exact fibrations (see \cite[A.8.12]{lu17}). Recall that a \emph{left exact fibration} is a morphism of simplicial sets $p:\Mcc \rightarrow \Delta^1$ such that
\begin{enumerate}
\item $p$ is a Cartesian fibration;
\item $\Mcc_0$ and $\Mcc_1$ admit finite limits;
\item the $\oo$-functor $\Mcc_1\rightarrow \Mcc_0$ determined by $p$ is left exact.
\end{enumerate}
\end{rmk}
\begin{notation}
According to the equivalence between the theory of recollements and that of left exact fibrations, it is clear that an adjunction $\Cc \leftrightarrows \Dc$ between $\oo$-categories with finite limits determines an $\oo$-category that is a recollement of $\Cc$ and $\Dc$. We will refer to such $\oo$-category as the \emph{recollement determined by the adjunction} $\Cc \leftrightarrows \Dc$.
\end{notation}
\begin{ex}
Notice that every $\oo$-category with finite limits $\Cc$ is the recollement of $\Cc$ and $\Delta^0$. Let $p:\Cc^{\triangleright}\rightarrow \Delta^1=(\Cc\rightarrow \Delta^0)\star \Delta^0$ (see \cite{lu09} for notation). Then $p$ is a left exact fibration:
\begin{enumerate}
\item $p^{-1}(0)=\Cc$ has finite limits by assumption. Clearly, $p^{-1}(1)=\Delta^0$ has finite limits;
\item to show that $p$ is a Cartesian fibration it suffices to show that there exists a Cartesian edge over $0\rightarrow 1$. Let $v$ denote the (unique) object of $\Cc^{\triangleright}$ over $1\in \Delta^1$ and let $f$ be a final object of $\Cc$. Then $f\rightarrow v$ is a Cartesian edge over $0\rightarrow 1$;
\item the $\oo$-functor $\Delta^0\rightarrow \Cc$ determined by $p$ is equivalent to $f:\Delta^0\rightarrow \Cc$. In particular, it is left exact.
\end{enumerate}
By the results in \cite[Appendix A.8]{lu17} we see that $\Cc \simeq \textup{Fun}_{\Delta^1}(\Delta^1,\Cc^{\triangleright})$ is the recollement of $\Cc\simeq \{s\in \textup{Fun}_{\Delta^1}(\Delta^1,\Cc^{\triangleright}) : s(1) \text{ is a final object of } p^{-1}(1)\}$ and $\Delta^{0}\simeq \{s\in \textup{Fun}_{\Delta^1}(\Delta^1,\Cc^{\triangleright}): s \text{ is a Cartesian edge}\}$.
\end{ex}
We will be interested in the following $\oo$-categorical analogue of \cite[Exercise 9.5.11]{sga4i}:
\begin{prop}\label{recollement category of functors}
Let $\Cc$ and $\Dc$ be $\oo$-categories. Assume that $\Dc$ is the recollement of $\Dc_0,\Dc_1\subseteq \Dc$. Then $\Fun(\Cc,\Dc)$ is the recollement of $\Fun(\Cc,\Dc_0)$ and $\Fun(\Cc,\Dc_1)$.
\end{prop}
\begin{proof}
$\Fun(\Cc,\Dc)$ admits finite limits by (the dual of) \cite[Corollary 5.1.2.3 (1)]{lu09}. $\Fun(\Cc,\Dc_0)$ is a full subcategory of $\Fun(\Cc,\Dc)$ as it is a limit of the following diagram
\begin{equation*}
  \begindc{\commdiag}[18]
  \obj(30,15)[1]{$\Fun(\Cc,\Dc)$}
  \obj(-30,-15)[2]{$\Pi_{\Cc([0])}\Dc_0$}
  \obj(30,-15)[3]{$\Pi_{\Cc([0])}\Dc$,}
  \mor{1}{3}{$$}
  \mor{2}{3}{$$}[\atright, \injectionarrow]
  \enddc
\end{equation*}
where the vertical arrow is induced by evaluations at all objects $c\in \Cc([0])$ and the bottom arrow is induced by $\Dc_0\hookrightarrow \Dc$ and is fully faithful. An analogous argument shows that $\Fun(\Cc,\Dc_1)$ is a full subcategory of $\Fun(\Cc,\Dc)$ as well. The fact that $\Fun(\Cc,\Dc_0)$ and $\Fun(\Cc,\Dc_1)$ are stable under equivalences is straightforward as $\Dc_0\subseteq \Dc$ and $\Dc_1\subseteq \Dc$ have this property. The inclusions $\Fun(\Cc,\Dc_0),\Fun(\Cc,\Dc_1)\subseteq \Fun(\Cc,\Dc)$ are induced by $i_0:\Dc_0\hookrightarrow \Dc$ and by $i_1:\Dc_1\hookrightarrow \Dc$ and admit $L_0\circ - :\Fun(\Cc,\Dc_0)\rightarrow \Fun(\Cc,\Dc)$ and $L_1\circ - :\Fun(\Cc,\Dc_1)\rightarrow \Fun(\Cc,\Dc)$ as left adjoints. These are left exact as by (the dual of) \cite[Corollary 5.1.2.3 (2)]{lu09}, limits in $\Fun(\Cc,\Dc)$ are computed pointwise. This also implies that the functor $\Fun(\Cc,\Dc_0)\rightarrow \Fun(\Cc,\Dc_1)$ is equivalent to the constant functor determined by a final object of $\Fun(\Cc,\Dc_1)$. Finally, $(L_0\circ -, L_1\circ -)$ detects equivalences as a morphism in $\Fun(\Cc,\Dc)$ is an equivalence if and only if it is pointwise true.
\end{proof}

\begin{cor}\label{functors to a recollement}
Let $\Cc,\Dc$ be $\oo$-categories. Assume that $\Dc$ is the recollement of $\Dc_0$ and $\Dc_1$. The datum of functors $f_0:\Cc \rightarrow \Dc_0$, $f_1:\Cc \rightarrow \Dc_1$ and of a morphism $f_0\rightarrow L_0\circ i_1\circ f_1$ in $\Fun(\Cc,\Dc_0)$ determines a functor $f\in \Fun(\Cc,\Dc)$ such that $f_0\simeq L_0\circ f$ and $f_1\simeq L_1\circ f$.
\end{cor}
\begin{proof}
By the previous proposition, $\Fun(\Cc,\Dc)$ is the recollement of $\Fun(\Cc,\Dc_0)$ and $\Fun(\Cc,\Dc_1)$. Let $\Mcc_0=\Fun(\Cc,\Dc_0)$, $\Mcc_1$ be the full subcategory of $\Fun(\Delta^1,\Fun(\Cc,\Dc))$ spanned by morphisms $\alpha:g\rightarrow g'$ such that $g\in \Fun(\Cc,\Dc_1)$ and $\alpha$ exhibits $g'$ as a $\Fun(\Cc,\Dc_0)$-localization of $g$. The functor "evaluation at 1" $\Mcc_1^{\textup{op}}\rightarrow \Mcc_0^{\textup{op}}$ determines a morphism $\Theta: [1]\rightarrow Set_{\Delta}$. Let $\Mcc:=\textup{N}_{\Theta}([1])^{\textup{op}}$ (see \cite[Definition 3.2.5.2]{lu09}). By \cite[Proposition A.8.11]{lu17}, there is a left exact fibration $\Mcc \rightarrow \Delta^1$ and is the one associated to the recollement of $\Fun(\Cc,\Dc_0)$ and $\Fun(\Cc,\Dc_1)$. Unravelling the definitions, $(f_0,f_1,f_0\rightarrow L_0\circ i_1\circ f_1)$ is the datum of an object in $\Fun_{\Delta^1}(\Delta^1,\Mcc)$, which by \cite[Proposition A.8.8]{lu17} corresponds to a functor $f:\Cc \rightarrow \Dc$ ($\Delta^1\xrightarrow{id}\Delta^1$ is the left exact fibration associated to the recollement of $\Delta^0$ and $\Delta^0$, which is equivalent to $\Delta^0$).
The last assertion is clear.
\end{proof}

Moreover, we will be interested in the following situation. 
\begin{construction}\label{specialization functor}
  Let $\Dc$ be the recollement of $\Dc_0$ and $\Dc_1$. As usual, we will denote by $i_0$ (resp. $i_1$) the fully faithful embedding $\Dc_0 \hookrightarrow \Dc$ (resp. $\Dc_1 \hookrightarrow \Dc$) and by $L_0$ (resp. $L_1$) its left adjoint. We will further assume that the $\oo$-functor $L_0\circ i_1:\Dc_1 \rightarrow \Dc_0$ admits a left adjoint $U:\Dc_0 \rightarrow \Dc_1$. 

Notice that the adjunction $id:\Dc_1 \leftrightarrows \Dc_1:id$ induces a left exact fibration $\Mcc \rightarrow \Delta^1$. The associated $\oo$-category is $\Fun(\Delta^1,\Dc_1)$, recollement of $\Dc_1$ and $\Dc_1$. The relevant $\oo$-functors are
\begin{equation*}
  \begindc{\commdiag}[18]
  \obj(-50,0)[1]{$\Dc_1$}
  \obj(0,0)[2]{$\Fun(\Delta^1,\Dc_1)$}
  \obj(50,0)[3]{$\Dc_1,$}
  \mor(-48,2)(-14,2){$j_0$}
  \mor(-14,-2)(-48,-2){$R_0$}
  \mor(14,2)(48,2){$R_1$}
  \mor(48,-2)(14,-2){$j_1$}
  \enddc
\end{equation*}
defined at the level of objects as follows ($e$ denotes a final object of $\Dc_1$):
\begin{equation*}
  j_0(d)=(d\rightarrow e), \hspace{0.5cm} R_0(d_0\rightarrow d_1)=d_0,
\end{equation*}
\begin{equation*}
  j_1(d)=(d\xrightarrow{id} d), \hspace{0.5cm} R_1(d_0\rightarrow d_1)=d_1.
\end{equation*}
Notice that $R_1 \circ j_0(d)=e$ and $R_0\circ j_1\simeq id$.
 
We produce a morphism $sp_{\Dc}:\Dc \rightarrow \Fun(\Delta^1,\Dc_1)$ as follows: consider
\begin{itemize}
  \item $U\circ L_0 : \Dc:\rightarrow \Dc_1$;
  \item $L_1:\Dc \rightarrow \Dc_1$;
  \item the natural transformation $U\circ L_0\rightarrow L_1$ corresponding, under the adjunction $(U,L_0\circ i_1)$, to $L_0\rightarrow L_0\circ i_1 \circ L_1\simeq L_0 \circ (id\xrightarrow{unit} i_1\circ L_1)$.
\end{itemize}
By Corollary \ref{functors to a recollement}, these data determine an $\oo$-functor
\begin{equation*}
  sp_{\Dc}:\Dc \rightarrow \Fun(\Delta^1,\Dc_1).
\end{equation*}
\end{construction}
\begin{rmk}
  We will be particularly interested in the composition of $sp_{\Dc}:\Dc \rightarrow \Fun(\Delta^1,\Dc_1)$ with the $\oo$-functor
  \begin{equation*}
    fiber: \Fun(\Delta^1,\Dc_1)\rightarrow \Dc_1,
  \end{equation*}
  right adjoint to the $\oo$-functor
  \begin{equation*}
   j_0: \Dc_1 \rightarrow \Fun(\Delta^1,\Dc_1).
  \end{equation*}
\end{rmk}

\subsection*{Adjunctions between recollements}
In this subsection we will study in some detail how to produce adjunctions between recollements. 
\begin{notation}
In this section, we will consider two $\oo$-categories $\Cc$ and $\Dc$ such that $\Cc$ is the recollement of two subcategories $\Cc_0$ and $\Cc_1$ and $\Dc$ is the recollement of two subcategories $\Dc_0$ and $\Dc_1$. Moreover, we will label $i_{k}:\Cc_k \rightarrow \Cc$ (resp. $j_k:\Dc_k \rightarrow \Dc$) the fully faithful embedding and $L_k:\Cc \rightarrow \Cc_k$ (resp. $R_k:\Dc \rightarrow \Dc_k$) its left adjoint ($k=0,1$). 

We will assume that we are given two adjunctions $F_0:\Cc_0\leftrightarrows \Dc_0:G_0$ and $F_1:\Cc_1\leftrightarrows \Dc_1:G_1$ and homotopies
\begin{equation}
  \begindc{\commdiag}[15]
    \obj(-60,15)[1]{$\Cc_0$}
    \obj(-20,15)[2]{$\Cc_1$}
    \obj(-60,-15)[3]{$\Dc_0$}
    \obj(-20,-15)[4]{$\Dc_1$}
    \mor{2}{1}{$L_0\circ i_1$}[\atright,\solidarrow]
    \mor{1}{3}{$F_0$}[\atright,\solidarrow]
    \mor{2}{4}{$F_1$}
    \mor{4}{3}{$R_0\circ j_1$}
    \obj(20,15)[5]{$\Cc_0$}
    \obj(60,15)[6]{$\Cc_1$}
    \obj(20,-15)[7]{$\Dc_0$}
    \obj(60,-15)[8]{$\Dc_1.$}
    \mor{6}{5}{$L_0\circ i_1$}[\atright,\solidarrow]
    \mor{7}{5}{$G_0$}
    \mor{8}{6}{$G_1$}[\atright,\solidarrow]
    \mor{8}{7}{$R_0\circ j_1$}
  \enddc
\end{equation}
\end{notation}

Our aim will be to construct an adjunction $F:\Cc \leftrightarrows \Dc:G$ out of these data.

We will start with an explicit description of a recollement. Consider the following variant of Construction $\ref{specialization functor}$:

\begin{construction}\label{bar specialization functor}
  As we have already observed, $\Fun(\Delta^1,\Cc_0)$ is the recollement of $\Cc_0$ and $\Cc_0$ (corresponding to the left exact fibration given by the adjunction $id:\Cc_0\leftrightarrows \Cc_0:id$). Then, by Corollary $\ref{functors to a recollement}$, $(L_0:\Cc \rightarrow \Cc_0, L_0\circ i_1 \circ L_1:\Cc \rightarrow \Cc_0, L_0\rightarrow L_0\circ i_1 \circ L_1)$ corresponds to an $\oo$-functor
  \begin{equation}
    \overline{sp}_{\Cc}:\Cc \rightarrow \Fun(\Delta^1,\Cc_0).
  \end{equation}
\end{construction}

\begin{prop}\label{explicit description recollement}
Let $(\{ \Cc \}_{k=0,1}, \{ i_k \}_{k=0,1}, \{ L_k \}_{k=0,1} )$ be as above. The square
\begin{equation}
  \begindc{\commdiag}[13]
    \obj(-40,20)[1]{$\Cc$}
    \obj(40,20)[2]{$\Cc_1$}
    \obj(-40,-20)[3]{$\Fun(\Delta^1,\Cc_0)$}
    \obj(40,-20)[4]{$\Cc_0,$}
    \mor{1}{2}{$L_1$}
    \mor{1}{3}{$\overline{sp}_{\Cc}$}
    \mor{2}{4}{$L_0\circ i_1$}
    \mor{3}{4}{$ev_{1}$}
  \enddc
\end{equation}
where $ev_1$ denotes the "evaluation at $1$" $\oo$-functor, is Cartesian.
\end{prop}
\begin{proof}
We shall prove that $\Fun(\Delta^1,\Cc_0)\x_{\Cc_0}\Cc_1$ is the recollement of $\Cc_0$ and $\Cc_1$.  First notice that $\Fun(\Delta^1,\Cc_0)\x_{\Cc_0}\Cc_1$ has finite limits.
 
Let $p:\Fun(\Delta^1,\Cc_0)\x_{\Cc_0}\Cc_1\rightarrow \Fun(\Delta^1,\Cc_0)$, $q:\Fun(\Delta^1,\Cc_0)\x_{\Cc_0}\Cc_1\rightarrow \Cc_1$ and $r\simeq ev_1\circ p\simeq L_0\circ i_1\circ q:\Fun(\Delta^1,\Cc_0)\x_{\Cc_0}\Cc_1\rightarrow \Cc_0$ denote the canonical morphisms. Let
\begin{itemize}
  \item $l_0:\Cc_0\rightarrow \Cc \rightarrow \Fun(\Delta^1,\Cc_0)\x_{\Cc_0}\Cc_1$;
  \item $l_1:\Cc_1\rightarrow \Cc \rightarrow \Fun(\Delta^1,\Cc_0)\x_{\Cc_0}\Cc_1$.
\end{itemize}
Recall that, for every couple of objects $x,y$ in $\Fun(\Delta^1,\Cc_0)\x_{\Cc_0}\Cc_1$, the canonical morphism
\begin{equation}
  \Map_{\Fun(\Delta^1,\Cc_0)\x_{\Cc_0}\Cc_1}(x,y)\rightarrow \Map_{\Fun(\Delta^1,\Cc_0)}(p(x),p(y))\x_{\Map_{\Cc_0}(r(x),r(y))}\Map_{\Cc_1}(q(x),q(y))
\end{equation}
is a weak homotopy-equivalence.
It is then easy to see that $l_0$ and $l_1$ are fully faithful. 
Indeed, if $x,y$ are objects in $\Cc_0$, then $q(x)\simeq q(y)\simeq 0$ (final object in $\Cc_1$). In particular,
\begin{equation}
  Map_{\Fun(\Delta^1,\Cc_0)}(p(x),p(y))\x_{\Map_{\Cc_0}(r(x),r(y))}\Map_{\Cc_1}(q(x),q(y))\simeq Map_{\Fun(\Delta^1,\Cc_0)}(p(x),p(y)).
\end{equation}
Moreover, for such $x,y$,
\begin{equation}
  \Map_{\Cc}(x,y)\simeq \Map_{\Fun(\Delta^1,\Cc_0)}(p(x),p(y)).
\end{equation}

If $x,y \in \Cc_1$, it is easy to see that
\begin{equation}
  \Map_{\Fun(\Delta^1,\Cc_0)}(p(x),p(y))\simeq \Map_{\Cc_0}(r(x),r(y))
\end{equation}
(as $p(x)=L_0\circ i_1(x)\xrightarrow{id}L_0\circ i_1(x)$). Therefore,
\begin{equation}
  \Map_{\Cc}(x,y)\rightarrow Map_{\Fun(\Delta^1,\Cc_0)}(p(x),p(y))\x_{\Map_{\Cc_0}(r(x),r(y))}\Map_{\Cc_1}(q(x),q(y)) \simeq \Map_{\Cc_1}(q(x),q(y))
\end{equation}
is a weak homotopy-equivalence.

Moreover, the functors
\begin{equation}
  ev_0\circ p :\Fun(\Delta^1,\Cc_0)\x_{\Cc_0}\Cc_1\rightarrow \Cc_0, \hspace{0.5cm} q:\Fun(\Delta^1,\Cc_0)\x_{\Cc_0}\Cc_1\rightarrow \Cc_1
\end{equation}
are left adjoints to $l_0$ and $l_1$ respectively. It is also clear that these functors are left exact.

The functor $q\circ l_0$ is equivalent to $L_1\circ i_0$ and therefore it takes every object of $\Cc_0$ to the final object in $\Cc_1$.

We are left to show that $(ev_0\circ p,q)$ detect equivalences. 

An object in $\Fun(\Delta^1,\Cc_0)\x_{\Cc_0}\Cc_1$ is a triplet
\begin{equation}
  x=\bigl ( x_0\rightarrow x_0',x_1,\alpha_x:x_0'\simeq L_0\circ i_1(x_1) \bigr ),
\end{equation}
where $x_0\rightarrow x_0'$ is an arrow in $\Cc_0$ and $x_1$ is an object in $\Cc_1$.

A morphism between two objects $x,y$ in $\Fun(\Delta^1,\Cc_0)\x_{\Cc_0}\Cc_1$ is a triplet
\begin{equation}
\Bigl ( \sigma:\Delta^1\times \Delta^1 \rightarrow \Cc_0, g, \tau:\sigma_{|\{1\}\x \Delta^1}\simeq L_0\circ i_1(g) \Bigr ),
\end{equation}
where $\sigma_{|\Delta^1\x \{0\}}=(x_0\rightarrow x_0')$, $\sigma_{|\Delta^1\x \{1\}}=(y_0\rightarrow y_0')$ and $g:x_1\rightarrow y_1$. Such a morphism is an equivalence if and only if $\sigma_{|\{0\}\x\Delta^1}$ and $g$ are equivalences (as $\sigma_{|\{1\}\x\Delta^1}\simeq L_0\circ i_1(g)$). Using this description and the fact that $ev_0\circ p$ and $q$ send such triplet to $\sigma_{|\{0\}\x\Delta^1}$ and $g$ respectively, it is clear that $(ev_0\circ p,q)$ detect equivalences.

Then $\Cc \simeq \Fun(\Delta^1)\x_{\Cc_0}\Cc_1$ by \cite[Proposition A.8.14]{lu17}.
\end{proof}
\begin{cor}\label{mapping spaces in a recollement}
With the notation of the previous proposition, for any couple of objects $x,y$ in $\Cc$,
\begin{equation}
  \Map_{\Cc}(x,y)\simeq \Map_{\Fun(\Delta^1,\Cc_0)}(p(x),p(y))\x_{\Map_{\Cc_0}(r(x),r(y))}\Map_{\Cc_1}(q(x),q(y)).
\end{equation}
\end{cor}
\begin{notation}
We will denote by $F_0^{\Delta^1}:\Fun(\Delta^1,\Cc_0)\rightarrow \Fun(\Delta^1,\Dc_0)$ (resp. $G_0^{\Delta^1}:\Fun(\Delta^1,\Dc_0)\rightarrow \Fun(\Delta^1,\Cc_0)$) the functor corresponding to $(\Fun(\Delta^1,\Cc_0)\xrightarrow{ev_0} \Cc_0\xrightarrow{F_0}\Dc_0,\Fun(\Delta^1,\Cc_0)\xrightarrow{ev_1} \Cc_0\xrightarrow{F_0}\Dc_0, F_0\circ ev_0\rightarrow F_0\circ ev_1)$ (resp. $(\Fun(\Delta^1,\Dc_0)\xrightarrow{ev_0} \Dc_0\xrightarrow{G_0}\Cc_0,\Fun(\Delta^1,\Dc_0)\xrightarrow{ev_1} \Dc_0\xrightarrow{G_0}\Cc_0, G_0\circ ev_0\rightarrow G_0\circ ev_1)$ ).
\end{notation}

\begin{lmm}
  $F_0^{\Delta^1}:\Fun(\Delta^1,\Cc_0)\leftrightarrows \Fun(\Delta^1,\Dc_0):G_0^{\Delta^1}$ are adjoint functors.
\end{lmm}
\begin{proof}
This follows immediately from the definition and from the fact that $F_0$ is left adjoint to $G_0$.
\end{proof}

\begin{construction}
Let $F:\Cc \rightarrow \Dc$ be the functor corresponding, by Corollary $\ref{functors to a recollement}$, to $(\Cc \xrightarrow{F_0\circ L_0}\Dc_0, \Cc \xrightarrow{F_1\circ L_1}\Dc_1, F_0\circ L_0 \rightarrow R_0\circ j_1 \circ F_1\circ L_1\simeq F_0\circ L_0\circ i_1\circ L_1 )$.

Similarly, let $G:\Dc \rightarrow \Cc$ be the functor corresponding, by Corollary $\ref{functors to a recollement}$, to $(\Dc \xrightarrow{G_0\circ R_0}\Cc_0, \Dc \xrightarrow{G_1\circ R_1}\Cc_1, G_0\circ R_0 \rightarrow L_0\circ i_1 \circ G_1\circ R_1\simeq G_0\circ R_0\circ j_1 \circ R_1 )$.

Let $u:id_{\Cc}\rightarrow G\circ F$ denote the natural transformation defined as follows. By Proposition $\ref{recollement category of functors}$, $\Fun(\Cc,\Cc)$ is the recollement of $\Fun(\Cc,\Cc_0)$ and $\Fun(\Cc,\Cc_1)$. By Corollary $\ref{functors to a recollement}$, in order to define $u$ we need to exhibit two functors $u_0:\Delta^1\rightarrow \Fun(\Cc,\Cc_0)$, $u_1:\Delta^1\rightarrow \Fun(\Cc,\Cc_1)$ and a natural transformation $u_0\rightarrow L_0\circ i_1 \circ u_1$.
We define
\begin{itemize}
  \item $u_0:L_0\rightarrow L_0\circ G\circ F\simeq G_0\circ R_0\circ F\simeq G_0\circ F_0 \circ L_0$ induced by $id\rightarrow G_0\circ F_0$;
  \item $u_1:L_1\rightarrow L_1\circ G\circ F\simeq G_1\circ R_1\circ F\simeq G_1\circ F_1 \circ L_1$ induced by $id\rightarrow G_1\circ F_1$;
  \item $u_0\rightarrow L_0\circ i_1 \circ u_1$ as
  \begin{equation}
    \begindc{\commdiag}[13]
      \obj(-80,20)[1]{$L_0$}
      \obj(80,20)[2]{$G_0\circ F_0 \circ L_0$}
      \obj(-80,-20)[3]{$L_0\circ i_1\circ L_1$}
      \obj(80,-20)[4]{$L_0\circ i_1 \circ G_1\circ F_1\circ L_1\simeq G_0\circ R_0\circ j_1\circ F_1\circ L_1,$}
      \mor{1}{2}{$u_0$}
      \mor{1}{3}{$$}
      \mor{2}{4}{$$}
      \mor{3}{4}{$L_0\circ i_1\circ u_1$}
    \enddc
  \end{equation}
  where the left vertical arrow is induced by $id\rightarrow i_1\circ L_1$ and the right vertical arrow by $F_0\circ L_0\rightarrow R_0 \circ j_1\circ F_1\circ L_1$.
\end{itemize}
\end{construction}

\begin{prop}\label{adjunctions between recollements}
 $u:id_{\Cc}\rightarrow G\circ F$ is a unit transformation for $(F,G)$ (see \cite[Definition 5.2.2.7]{lu09}). In particular, $F$ is a left adjoint to $G$.
\end{prop}
\begin{proof}
Let $x$ (resp. $y$) be an object of $\Cc$ (resp. $\Dc$). We need to show that the morphism
\begin{equation}\label{unit transformation composition}
\Map_{\Dc}(F(x),y)\rightarrow \Map_{\Cc}(G\circ F(x),G(y))\xrightarrow{u(x)}\Map_{\Cc}(x,G(y))
\end{equation}
is a weak-equivalence.

By Proposition $\ref{explicit description recollement}$, we have two pullback squares
\begin{equation}
  \begindc{\commdiag}[13]
    \obj(-100,20)[1]{$\Cc$}
    \obj(-40,20)[2]{$\Cc_1$}
    \obj(-100,-20)[3]{$\Fun(\Delta^1,\Cc_0)$}
    \obj(-40,-20)[4]{$\Cc_0$}
    \mor{1}{2}{$q_{\Cc}\sim L_1$}
    \mor{1}{3}{$p_{\Cc}\sim \overline{sp}_{\Cc}$}[\atright,\solidarrow]
    \mor{2}{4}{$L_0\circ i_1$}
    \mor{3}{4}{$ev_1$}
    \mor{1}{4}{$r_{\Cc}$}
    \obj(40,20)[5]{$\Dc$}
    \obj(100,20)[6]{$\Dc_1$}
    \obj(40,-20)[7]{$\Fun(\Delta^1,\Dc_0)$}
    \obj(100,-20)[8]{$\Dc_0$}
    \mor{5}{6}{$q_{\Dc}\sim R_1$}
    \mor{5}{7}{$p_{\Dc}\sim \overline{sp}_{\Dc}$}[\atright,\solidarrow]
    \mor{6}{8}{$R_0\circ j_1$}
    \mor{7}{8}{$ev_1$}
    \mor{5}{8}{$r_{\Dc}$}
  \enddc
\end{equation}
and the map $(\ref{unit transformation composition})$, by Corollary $\ref{mapping spaces in a recollement}$, corresponds to
\begin{equation}
\begindc{\commdiag}[13]
  \obj(0,23)[1]{$\Map_{\Fun(\Delta^1,\Dc_0)}(p_{\Dc}\circ F(x),p_{\Dc}(y))\x_{\Map_{\Dc_0}(r_{\Dc}\circ F(x),r_{\Dc}(y))}\Map_{\Dc_1}(q_{\Dc}\circ F(x),q_{\Dc}(y))$}
  \obj(0,0)[2]{$\Map_{\Fun(\Delta^1,\Cc_0)}(p_{\Cc}\circ G\circ F(x),p_{\Cc}\circ G(y))\x_{\Map_{\Cc_0}(r_{\Cc}\circ G\circ F(x),r_{\Cc}\circ G(y))}\Map_{\Cc_1}(q_{\Cc}\circ G\circ F(x),q_{\Cc}\circ G(y))$}
   \obj(0,-23)[3]{$\Map_{\Fun(\Delta^1,\Cc_0)}(p_{\Cc}(x),p_{\Cc}\circ G(y))\x_{\Map_{\Cc_0}(r_{\Cc}(x),r_{\Cc}\circ G(y))}\Map_{\Cc_1}(q_{\Cc}(x),q_{\Cc}\circ G(y)).$}
  \mor{1}{2}{$$}
  \mor{2}{3}{$$}
  \enddc
\end{equation}
As we have equivalences
\begin{itemize}
  \item $p_{\Dc}\circ F\simeq F_0^{\Delta^1}\circ p_{\Cc}$, \hspace{2.25cm} $p_{\Cc}\circ G\simeq G_0^{\Delta^1}\circ p_{\Dc}$,
  \item $q_{\Dc}\circ F\simeq F_1\circ L_1$, \hspace{2.5cm} $q_{\Cc}\circ F\simeq G_1\circ R_1$,
  \item $r_{\Dc}\circ F\simeq R_0\circ j_1 \circ F_1\circ L_1$, \hspace{1.23cm} $r_{\Cc}\circ G\simeq L_0\circ i_1 \circ G_1\circ R_1$,
\end{itemize}
the compositions
\begin{equation}
\begindc{\commdiag}[13]
  \obj(0,23)[1]{$\Map_{\Fun(\Delta^1,\Dc_0)}(p_{\Dc}\circ F(x),p_{\Dc}(y))\simeq \Map_{\Fun(\Delta^1,\Dc_0)}(F_0^{\Delta^1}\circ p_{\Cc}(x),p_{\Dc}(y))$}
  \obj(0,0)[2]{$\Map_{\Fun(\Delta^1,\Cc_0)}(p_{\Cc}\circ G\circ F(x),p_{\Cc}\circ G(y))\simeq \Map_{\Fun(\Delta^1,\Cc_0)}(G_0^{\Delta^1}\circ F_0^{\Delta^1}\circ p_{\Cc}(x),G_0^{\Delta^1}\circ p_{\Dc}(y))$}
   \obj(0,-23)[3]{$\Map_{\Fun(\Delta^1,\Cc_0)}(p_{\Cc}(x),p_{\Cc}\circ G(y))\simeq \Map_{\Fun(\Delta^1,\Cc_0)}(p_{\Cc}(x),G_0^{\Delta^1}\circ p_{\Dc}(y))$}
  \mor{1}{2}{$$}
  \mor{2}{3}{$$}
  \enddc
\end{equation}
and
\begin{equation}
\begindc{\commdiag}[13]
  \obj(0,23)[1]{$\Map_{\Dc_1}(q_{\Dc}\circ F(x),q_{\Dc}(y))\simeq \Map_{\Dc_1}(F_1\circ L_1(x),q_{\Dc}(y))=\Map_{\Dc_1}(F_1\circ L_1(x),R_1(y))$}
  \obj(0,0)[2]{$\Map_{\Cc_1}(q_{\Cc}\circ G\circ F(x),q_{\Cc}\circ G(y))\simeq \Map_{\Cc_1}(G_1\circ F_1\circ L_1(x),G_1\circ R_1(y))$}
   \obj(0,-23)[3]{$\Map_{\Cc_1}(q_{\Cc}(x),q_{\Cc}\circ G(y))\simeq \Map_{\Cc_1}(L_1(x),G_1\circ R_1(y))$}
  \mor{1}{2}{$$}
  \mor{2}{3}{$$}
  \enddc
\end{equation}
are weak-equivalences, as $(F^{\Delta^1}_0,G_0^{\Delta^1})$ and $(F_1,G_1)$ are adjoint functors. Similarly, the composition
\begin{equation}
\begindc{\commdiag}[13]
  \obj(0,23)[1]{$\Map_{\Dc_0}(r_{\Dc}\circ F(x),r_{\Dc}(y))\simeq \Map_{\Dc_0}(F_0\circ L_0\circ i_1\circ L_1(x),R_0\circ j_1\circ R_1(y))$}
  \obj(0,0)[2]{$\Map_{\Cc_0}(r_{\Cc}\circ G\circ F(x),r_{\Cc}\circ G(y))\simeq \Map_{\Dc_0}(G_0\circ F_0\circ L_0\circ i_1\circ L_1(x),G_0\circ R_0\circ j_1\circ R_1(y))$}
   \obj(0,-23)[3]{$\Map_{\Cc_0}(r_{\Cc}(x),r_{\Cc}\circ G(y))\simeq \Map_{\Dc_0}(L_0\circ i_1\circ L_1(x),G_0\circ R_0\circ j_1\circ R_1(y))$}
  \mor{1}{2}{$$}
  \mor{2}{3}{$$}
  \enddc
\end{equation}
is a weak-equivalence as $(F_0,G_0)$ is an adjunction. 

Then $(\ref{unit transformation composition})$ is a weak-equivalence and the lemma follows from \cite[Proposition 5.2.2.8]{lu09}.
\end{proof}

%% file: sections/review_of_the_theory_of_tame_vanishing_cycles.tex
In this section we will review the theory of (tame) vanishing cycles as developed in \cite[Expos\'e II]{sga7i}, \cite[Expos\'e XIII]{sga7ii}. 
\subsection*{The classical definition of tame vanishing cycles}
Let $A$ be an henselian discrete valuation ring. Fix an uniformiser $\pi$. Then the residue field of $A$ is $k:=A/(\pi)$, while the fraction field of $A$ is $K:=A[\pi^{-1}]$. Let $p$ denote the characteristic of $k$. Fix a separable closure $k^s$ of $k$ and a separable closure $K^s$ of $K$. Let $K^u$ and $K^t$ denote the maximal unramified and maximal tamely ramified extensions of $K$ inside $K^s$. We have a short exact sequence of profinite groups (see \cite{se62})
\begin{equation}
  1\rightarrow I^t:=Gal(K^t/K^u) \rightarrow Gal(K^t/K) \rightarrow Gal(K^u/K)\simeq Gal(k^s/k)\rightarrow 1.
\end{equation}
$I^t$ is called the \emph{tame inertia group}. There is an isomorphism
\begin{equation}
  I^t\simeq \varprojlim_{(n,p)=1}\mu_n(K^s),
\end{equation}
where $\mu_n(K^s)$ denotes the group of $n^{th}$ roots of the unity in $K^s$.

We will use the following notation:
\begin{itemize}
  \item $S:=Spec(A)$,
  \item $\sigma:=Spec(k)$,
  \item $\sigma^s:=Spec(k^s)$,
  \item $\eta:=Spec(K)$,
  \item $\eta^u:=Spec(K^u)$,
  \item $\eta^t:=Spec(K^t)$.
\end{itemize}

For an $S$-scheme $p:X\rightarrow S$, we will consider the following diagram in which all squares are Cartesian:
\begin{equation}
  \begindc{\commdiag}[18]
    \obj(-60,10)[1u]{$X_{\sigma^s}$}
    \obj(-30,10)[2u]{$X_{\sigma}$}
    \obj(0,10)[3u]{$X$}
    \obj(30,10)[4u]{$X_{\eta}$}
    \obj(60,10)[5u]{$X_{\eta^u}$}
    \obj(90,10)[6u]{$X_{\eta^t}$}
    \obj(-60,-10)[1d]{$\sigma^s$}
    \obj(-30,-10)[2d]{$\sigma$}
    \obj(0,-10)[3d]{$S$}
    \obj(30,-10)[4d]{$\eta$}
    \obj(60,-10)[5d]{$\eta^u$}
    \obj(90,-10)[6d]{$\eta^t$}
    \mor{1u}{2u}{$$}
    \mor{2u}{3u}{$i$}
    \mor{4u}{3u}{$j$}
    \mor{5u}{4u}{$$}
    \mor{6u}{5u}{$$}
    \mor{1d}{2d}{$$}
    \mor{2d}{3d}{$i_0$}
    \mor{4d}{3d}{$j_0$}
    \mor{5d}{4d}{$$}
    \mor{6d}{5d}{$$}
    \mor{1u}{1d}{$$}
    \mor{2u}{2d}{$$}
    \mor{3u}{3d}{$$}
    \mor{4u}{4d}{$$}
    \mor{5u}{5d}{$$}
    \mor{6u}{6d}{$$}
    \cmor((-60,13)(-60,16)(-58,18)(-30,18)(-3,18)(-1,16)(-1,13)) \pdown(-30,21){$i^s$}
    \cmor((-60,-13)(-60,-16)(-58,-18)(-30,-18)(-3,-18)(-1,-16)(-1,-13)) \pup(-30,-22){$i_0^s$}
    \cmor((60,13)(60,16)(58,18)(30,18)(3,18)(1,16)(1,13)) \pdown(30,21){$j^u$}
    \cmor((60,-13)(60,-16)(58,-18)(30,-18)(3,-18)(1,-16)(1,-13)) \pup(30,-22){$j_0^u$}
    \cmor((90,18)(90,23)(88,25)(45,25)(2,25)(0,23)(0,18)) \pdown(45,29){$j^t$}
    \cmor((90,-18)(90,-23)(88,-25)(45,-25)(2,-25)(0,-23)(0,-18)) \pup(45,-29){$j_0^t$}
  \enddc
\end{equation}

Let $\Det{X_{\sigma^s};\Lambda}^{Gal(k^s/k)}$ denote the $\oo$-category of \'etale sheaves of $\Lambda$-modules on $X_{\sigma^s}$ endowed with a continuous action of $Gal(k^s/k)$ compatible with the action on $X_{\sigma^s}$.
Similarly, let $\Det{X_{\sigma^s};\Lambda}^{Gal(K^t/K)}$ denote the $\oo$-category of sheaves on $X_{\sigma^s}$ endowed with a continuous action of $Gal(K^t/K)$ compatible with the action on $X_{\sigma^s}$ induced by the morphism $Gal(K^t/K)\rightarrow Gal(k^s/k)$.
These two categories are related by an adjunction:
\begin{equation}
  triv: \Det{X_{\sigma^s};\Lambda}^{Gal(k^s/k)}\leftrightarrows \Det{X_{\sigma^s};\Lambda}^{Gal(K^t/K)}:(-)^{I^t}.
\end{equation}
Roughly, for $\Fc \in \Det{X_{\sigma^s};\Lambda}^{Gal(k^s/k)}$, $triv(\Fc)$ is the sheaf on $X_{\sigma^s}$ with the continuous $Gal(K^t/K)$-action induced by $Gal(K^t/K)\rightarrow Gal(k^s/k)$, while for $\Gc \in \Det{X_{\sigma^s};\Lambda}^{Gal(K^t/K)}$, $\Gc^{I^t}$ is the sheaf of $I^t$-homotopy fixed points of $\Gc$, with the induced $Gal(k^s/k)$-action.
\begin{notation}
  With a little abuse of notation, for $\Fc \in \Det{X_{\sigma^s};\Lambda}^{Gal(k^s/k)}$, we will write $\Fc$ instead of $triv(\Fc)$. 
  We hope that it will be clear from the context whether we are considering one object or the other.
\end{notation}
For $\Gc \in \Det{X_{\eta};\Lambda}$, the sheaf of \emph{tame nearby cycles with coefficients in $\Gc$} is
\begin{equation}
  \Psi^{cl,t}_{\eta,p}(\Gc):=(i^{s})^*\circ (j^t)_*(\Gc_{|X_{\eta^t}}) \in \Det{X_{\sigma^s};\Lambda}^{Gal(K^t/K)}.
\end{equation}

The $Gal(K^t/K)$-action on $\Psi^{cl,t}_{\eta,p}(\Gc)$ is the one induced by that on $\Gc_{|X_{\eta^t}} \in \Det{X_{\eta^t};\Lambda}^{Gal(K^t/K)}$.

For $\Fc \in \Det{X;\Lambda}$, the unit of the adjuction $((j^t)^*,(j^t)_*)$ induces a morphism
\begin{equation}
  sp: (i^s)^*(\Fc)\rightarrow \Psi^{cl,t}_{\eta,p}(j^*\Fc) 
\end{equation}
in the $\oo$-category $\Det{X_{\sigma^s};\Lambda}^{Gal(K^t/K)}$. This morphism is called the \emph{specialization morphism} and its cone, denoted by
\begin{equation}
  \Phi^{cl,t}_p(\Fc) \in \Det{X_{\sigma^s};\Lambda}^{Gal(K^t/K)},
\end{equation}
is called the sheaf of \emph{tame vanishing cycles with coefficients in $\Fc$}.

\subsection*{Tame vanishing cycles via Deligne-Mumford stacks.}
Our aim in this subsection will be to reinterpret the definition of nearby and vanishing cycles using the language of pro Deligne-Mumford stacks. Indeed, the $\oo$-categories $\Det{X_{\sigma^s};\Lambda}^{Gal(k^s/k)}$ and $\Det{X_{\sigma^s};\Lambda}^{Gal(K^t/K)}$ can (and will) be interpreted as the $\oo$-categories of \'etale sheaves of $\Lambda$-modules on some pro Deligne-Mumford stacks $[X_{\sigma^s}/Gal(k^s/k)]$ and $[X_{\sigma^s}/Gal(K^t/K)]$. The reason why we need to consider pro Deligne-Mumford stacks is that the groups $Gal(k^s/k)$ and $Gal(K^t/K)$ are profinite.

\subsection*{The pro DM stack \texorpdfstring{$[X_{\sigma^s}/Gal(k^s/k)]$}{dms}}
Write $k^s$ as the colimit $\varinjlim_ik_i$, where $k\subseteq k_i$ runs in the filtered set $\mathscr{I}$ of finite Galois extensions of $k$ inside $k^s$. 

For every finite Galois extension $k\subseteq k_i \subseteq k^s$, let $\sigma_i:=Spec(k_i)$ and let $S_i\rightarrow S$ be the corresponding finite \'etale covering of $S$. 
Put $S^s=\varprojlim_i S_i$ and set $X_{S_i}=X\x_S S_i$, $X_{S^s}=\varprojlim_i X_{S_i}$. Moreover, put
$X_{\sigma_i}:=X_{\sigma}\times_{X} X_{S_i}$. The group $Aut_S(S_i)\simeq Gal(k_i/k)$ acts naturally on $X_{S_i}$ and on $X_{\sigma_i}$.

\begin{rmk}
Notice that for each $i\in \mathscr{I}$, $X_{\sigma_i}$ is a nilthickeing of $X_{\sigma}\x_{Spec(\sigma)}Spec(\sigma_i)$.
\end{rmk}

For every tower of finite Galois extensions $k\subseteq k_i \subseteq k_j$ inside $k^s$, $Gal(k_j/k)$ acts on $X_{\sigma_i}$ by means of the canonical quotient morphism
\begin{equation}
  Gal(k_j/k)\rightarrow Gal(k_j/k)/Gal(k_j/k_i)\simeq Gal(k_i/k).
\end{equation}
In particular, we have a diagram of Deligne-Mumford stacks
\begin{equation}\label{diagram pro DM stack Xsigmas/Gal(sigma)}
    \mathscr{I}^{op} \rightarrow \DM{S}
\end{equation}
\begin{equation*}
    i\leq j \mapsto [X_{\sigma_j}/Gal(k_j/k)] \xrightarrow{\alpha_{ij}}[X_{\sigma_i}/Gal(k_i/k)].
\end{equation*}
Notice that this is the diagram obtained from the diagram
\begin{equation}
    \mathscr{I}^{op} \rightarrow \DM{S}
\end{equation}
\begin{equation*}
    i\leq j \mapsto [X_{S_j}/Gal(k_j/k)] \rightarrow [X_{S_i}/Gal(k_i/k)]
\end{equation*}
by pullbacks along $X_{\sigma}\rightarrow X$.

\begin{notation}
For every finite Galois extension $k\subseteq k_i$, we will use the following notation:
\begin{equation}
    \Det{X_{\sigma_i};\Lambda}^{Gal(k_i/k)}:=\Det{[X_{\sigma_i}/Gal(k_i/k)];\Lambda}.
\end{equation}
\end{notation}
For every tower of finite Galois extensions $k\subseteq k_i \subseteq k_j$, the adjunction
\begin{equation}
    \alpha_{ij}^*:\Det{[X_{\sigma_i}/Gal(k_i/k)];\Lambda}\leftrightarrows \Det{[X_{\sigma_j}/Gal(k_j/k)];\Lambda}: (\alpha_{ij})_*
\end{equation}
identifies with
\begin{equation*}
    v_{ij}^*:\Det{X_{\sigma_i};\Lambda}^{Gal(k_i/k)}\leftrightarrows \Det{X_{\sigma_j};\Lambda}^{Gal(k_j/k)}:(v_{ij})_*\circ (-)^{Gal(k_j/k_i)},
\end{equation*}
where $v_{ij}:X_{\sigma_j}\rightarrow X_{\sigma_i}$ is the morphism induced by $S_j\rightarrow S_i$.
We define
\begin{equation}
  [X_{\sigma^s}/Gal(k^s/k)]:="\varprojlim_i" [X_{\sigma_i}/Gal(k_i/k)],
\end{equation}
the pro-object in the $2$-category of DM stacks corresponding to the diagram (\ref{diagram pro DM stack Xsigmas/Gal(sigma)}).

\begin{defn}
Consider the $\oo$-functor
\begin{equation}\label{functoriality pushforward DM stacks}
    \DM{S}\rightarrow \prsr
\end{equation}
which expresses the functoriality of the assignments \begin{equation*}
    f:\mathfrak{Y}\rightarrow \mathfrak{X} \mapsto f_*:\Det{\mathfrak{Y};\Lambda}\rightarrow \Det{\mathfrak{X};\Lambda}.
\end{equation*}
We define the $\oo$-category of $Gal(k^s/k)$-equivariant sheaves on $X_{\sigma^s}$ as the limit of the diagram
\begin{equation*}
    (\ref{functoriality pushforward DM stacks})\circ (\ref{diagram pro DM stack Xsigmas/Gal(sigma)}):\mathscr{I}^{op}\rightarrow \DM{S} \rightarrow \prsr,
\end{equation*}
i.e. as 
\begin{equation}
  \Det{X_{\sigma^s};\Lambda}^{Gal(k^s/k)}=\Det{[X_{\sigma^s}/Gal(k^s/k)];\Lambda}:=
\end{equation}
\begin{equation*}  
  \varprojlim_i \Det{[X_{\sigma_i}/Gal(k_i/k)];\Lambda}=  \varprojlim_i \Det{X_{\sigma_i};\Lambda}^{Gal(k_i/k)} \in \prsr,
\end{equation*}

\end{defn}

\begin{rmk}
For every finite Galois extension $k\subseteq k_i \subseteq k^s$, let $\alpha_i:[X_{\sigma_i}/Gal(k_i/k)]\rightarrow X_{\sigma}$ denote the canonical morphism. It is a homeomorphism. The adjunction
\begin{equation}
  \alpha_i^*:\Det{X_{\sigma};\Lambda}\leftrightarrows \Det{[X_{\sigma_i}/Gal(k_i/k)];\Lambda}:(\alpha_{i})_*
\end{equation}
identifies with the adjunction
\begin{equation}
  v_i^*:\Det{X_{\sigma};\Lambda}\leftrightarrows \Det{X_{\sigma_i};\Lambda}^{Gal(k_i/k)}: (v_{i})_*\circ (-)^{Gal(k_i/k)},
\end{equation}
where $v_i:X_{\sigma_i}\rightarrow X_{\sigma}$. These functors are quasi-inverses. 

Taking the limit over all $i$'s\footnote{Technically, to take this limit we have to produce a diagram $\mathscr{I}^{op}\rightarrow \Fun(\Delta^1,\prsr)$, $i\mapsto (\phi_i)_*$, but we won't give the (easy) details of this construction here.}, we get the well known equivalence \cite[Expos\'e XIII]{sga7ii}

\begin{equation}
  v^*:\Det{X_{\sigma};\Lambda}\leftrightarrows \Det{X_{\sigma^s};\Lambda}^{Gal(k^s/k)}:v_*\circ (-)^{Gal(k^s/k)},
\end{equation}
where $v:X_{\sigma^s}\rightarrow X_{\sigma}$.

One can think of the homomorphism $[X_{\sigma^s}/Gal(k^s/k)]\simeq X_{\sigma}$ as the reason why there is an equivalence
\begin{equation}
  \Det{X_{\sigma^s};\Lambda}^{Gal(k^s/k)}\simeq \Det{X_{\sigma};\Lambda}.
\end{equation}
\end{rmk}

\subsection*{The pro DM stack \texorpdfstring{$[X_{\sigma^s}/Gal(K^t/K)]$}{dmk}}
Write $K^t$ as the colimit $\varinjlim_a K_a$, where $a$ runs in the filtered set $\mathscr{A}$ of finite Galois extensions $K\subseteq K_a\subseteq K^t$. For such an extension, denote by $K\subseteq K_a^u\subseteq K_a$ the maximal unramified extension of $K$ contained in $K_a$. We will denote by $k\subseteq k_a$ (resp. $S_a\rightarrow S$) the finite Galois extension (resp. finite \'etale covering) of $k$ (resp. $S$) corresponding to $K\subseteq K_a^u$ (see e.g. \cite{se62}). This provides a morphism of filtered sets
\begin{equation}
    \mathscr{A}\rightarrow \mathscr{I}.
\end{equation}

The canonical morphism
\begin{equation}
  Gal(K_a/K)\rightarrow Gal(K_a/K)/Gal(K_a^u/K)\simeq Gal(K_a^u/K)\simeq Gal(k_a/k)
\end{equation}
endows $X_{\sigma_a}$ with an action of $Gal(K_a/K)$.

Every tower of finite Galois extensions $K\subseteq K_a \subseteq K_b$ inside $K^t$ induces a tower of finite Galois extensions $K\subseteq K_a^u \subseteq K_b^u$ inside $K^u$ and therefore a commutative square
\begin{equation}\label{commutativity Gal(Kab)->Gal(kab)}
  \begindc{\commdiag}[18]
    \obj(-30,10)[1]{$Gal(K_b/K)$}
    \obj(30,10)[2]{$Gal(k_b/k)$}
    \obj(-30,-10)[3]{$Gal(K_a/K)$}
    \obj(30,-10)[4]{$Gal(k_a/k)$}
    \mor{1}{2}{$$}
    \mor{1}{3}{$$}
    \mor{2}{4}{$$}
    \mor{3}{4}{$$}
  \enddc
\end{equation}
which endows $X_{\sigma_a}$ with an action of $Gal(K_b/K)$.

In particular, we have a diagram of Deligne-Mumford stacks
\begin{equation}\label{diagram pro DM stack Xsigmas/Gal(eta)}
    \mathscr{A}^{op} \rightarrow \DM{S}.
\end{equation}
\begin{equation*}
    a\leq b \mapsto [X_{\sigma_b}/Gal(K_b/K)]\xrightarrow{\beta_{ab}}[X_{\sigma_a}/Gal(K_a/K)].
\end{equation*}

\begin{notation}
For every finite Galois extension $K\subseteq K_a \subseteq K^t$, we will use the following notation:
\begin{equation}
    \Det{X_{\sigma_a};\Lambda}^{Gal(K_a/K)}:=\Det{[X_{\sigma_a}/Gal(K_a/K)];\Lambda}.
\end{equation}
\end{notation}

For every tower of finite Galois extensions $K\subseteq K_a \subseteq K_b$, the adjunction
\begin{equation}
    \beta_{ab}^*:\Det{[X_{\sigma_a}/Gal(K_a/K)];\Lambda}\leftrightarrows \Det{[X_{\sigma_b}/Gal(K_b/K)];\Lambda}: (\beta_{ab})_*
\end{equation}
identifies with
\begin{equation*}
    u_{ab}^*:\Det{X_{\sigma_a};\Lambda}^{Gal(K_a/K)}\leftrightarrows \Det{X_{\sigma_b};\Lambda}^{Gal(K_b/K)}:(u_{ab})_*\circ (-)^{Gal(K_b/K_a)},
\end{equation*}
where $u_{ab}:X_{\sigma_b}\rightarrow X_{\sigma_a}$ is the morphism induced by $S_b\rightarrow S_a$.

We define the pro Deligne-Mumford stack
\begin{equation}
  [X_{\sigma^s}/Gal(K^t/K)]:="\varprojlim_a"[X_{\sigma_a}/Gal(K_a/K)].
\end{equation}
\begin{defn}
We define the $\oo$-category of $Gal(K^t/K)$-equivariant sheaves on $X_{\sigma^s}$ as the limit of the diagram
\begin{equation}
    (\ref{functoriality pushforward DM stacks})\circ(\ref{diagram pro DM stack Xsigmas/Gal(eta)}) : \mathscr{A}^{op}\rightarrow \prsr,
\end{equation}
i.e. as
\begin{equation}
  \Det{X_{\sigma^s};\Lambda}^{Gal(K^t/K)}=\Det{[X_{\sigma^s}/Gal(K^t/K)];\Lambda}:=
\end{equation}
\begin{equation*}  
 \varprojlim_a \Det{[X_{\sigma_a}/Gal(K_a/K)];\Lambda}=\varprojlim_a \Det{X_{\sigma_a};\Lambda}^{Gal(K_a/K)} \in \prsr.
\end{equation*}

\end{defn}

\subsection*{Tame vanishing cycles (reprise)}
For every finite Galois extension $K\subseteq K_a \subseteq K^t$, the morphism of Deligne-Mumford stacks
\begin{equation}
  \chi_a:[X_{\sigma_a}/Gal(K_a/K)]\rightarrow [X_{\sigma_a}/Gal(k_a/k)]
\end{equation}
induces an adjunction
\begin{equation}
  \chi_a^*:\Det{[X_{\sigma_a}/Gal(k_a/k)];\Lambda}\leftrightarrows \Det{[X_{\sigma_a}/Gal(K_a/K)];\Lambda}:(\chi_a)_*
\end{equation}
which identifies with
\begin{equation}
  triv:\Det{X_{\sigma_a};\Lambda}^{Gal(k_a/k)}\leftrightarrows \Det{X_{\sigma_a};\Lambda}^{Gal(K_a/K)}:(-)^{Gal(K_a/K_a^u)}.
\end{equation}

Notice that, by the commutativity of diagram \ref{commutativity Gal(Kab)->Gal(kab)}, the morphisms $\chi_a$ assemble in a morphism of filtered diagrams of DM stacks
\begin{equation}\label{chi A Fun(Delta1,DM)}
    \chi:\mathscr{A}^{op}\rightarrow \Fun(\Delta^1,\DM{S}).
\end{equation}
If we take the limit of the diagram
\begin{equation}
    (\ref{functoriality pushforward DM stacks})\circ (\ref{chi A Fun(Delta1,DM)}): \mathscr{A}^{op}\rightarrow \Fun(\Delta^1,\prsr),
\end{equation} 
we get the adjunction
\begin{equation}\label{adjunction etale sheaves vanishing topos}
  triv:\Det{X_{\sigma^s};\Lambda}^{Gal(k^s/k)}\leftrightarrows \Det{X_{\sigma^s};\Lambda}^{Gal(K^t/K)}:(-)^{I^t}.
\end{equation}

\begin{defn}
With the same notation as above, we define $\Det{\Vs^t_{p};\Lambda}$ to be the recollement determined by the adjunction \ref{adjunction etale sheaves vanishing topos}.
\end{defn}

We will denote by
\begin{equation}
  (i^s)^*:\Det{X;\Lambda}\rightarrow \Det{X_{\sigma^s};\Lambda}^{Gal(k^s/k)}
\end{equation}
the $\oo$-functor
\begin{equation}
  \Det{X;\Lambda}\xrightarrow{i^*}\Det{X_{\sigma};\Lambda}\xrightarrow{v^* \hspace{0.1cm}\simeq}\Det{X_{\sigma^s};\Lambda}^{Gal(k^s/k)}.
\end{equation}
Next we redefine the $\oo$-functor of tame nearby cycles. 
\begin{construction}
As a first step, we define an $\oo$-functor 
\begin{equation}\label{pushforward Xetat -> XSs}
    \bar{j}_*:\Det{X_{\eta^t};\Lambda}^{Gal(K^t/K)}\rightarrow \Det{X_{S^s},\Lambda}^{Gal(K^t/K)}.
\end{equation}

Let $\eta_a:=Spec(K_a)$ and let $X_{\eta_a}=X_{\eta}\times_{\eta}\eta_a$. Then $Gal(K_a/K)$ acts on $X_{\eta_a}$.

We have that 
\begin{equation}
  Aut_S(S_a)\simeq Gal(k_a/k).
\end{equation}
In particular, $Gal(K_a/K)$ acts on $X_{S_a}=X\times_SS_a$ by means of the morphism $Gal(K_a/K)\rightarrow Aut_S(S_a)\simeq Gal(k_a/k)$. This gives us a $Gal(K_a/K)$-equivariant map
\begin{equation}
     X_{\eta_a} \xrightarrow{j_a} X_{S_a}
\end{equation}
which defines a morphism of DM stacks
\begin{equation}
     [X_{\eta_a}/Gal(K_a/K)] \xrightarrow{\bar{j}_a} [X_{S_a}/Gal(K_a/K)].
\end{equation}
These morphisms assemble in a diagram
\begin{equation}
    \bar{j}:\mathscr{A}^{op}\rightarrow \Fun(\Delta^1,\DM{S}).
\end{equation}
The $\oo$-functor \ref{functoriality pushforward DM stacks} induces a diagram 
\begin{equation}
    \mathscr{A}^{op}\rightarrow \Fun(\Delta^1,\prsr)
\end{equation}
and we define \ref{pushforward Xetat -> XSs} as its limit.
\end{construction}
\begin{rmk}
Notice that, as $[X_{\eta_a}/Gal(K_a/K)]\simeq X_{\eta}$, the source of this $\oo$-functor is equivalent to $\Det{X_{\eta};\Lambda}$.
Also notice that if $a$ is associated to a finite, unramified Galois extension of $K$, then we also have that $[X_{S_a}/Gal(K_a/K)]\simeq X$. In particular, if we restrict the above diagram to the category $\mathscr{A}^u$ corresponding to unramified extensions, we get an $\oo$-functor
\begin{equation}
    j^u_*:\Det{X_{\eta^u};\Lambda}^{Gal(K^u/K)}\rightarrow \Det{X_{S^s};\Lambda}^{Gal(k^s/k)}
\end{equation}
that is canonically identified with 
\begin{equation}
    j_*:\Det{X_{\eta};\Lambda}\rightarrow \Det{X;\Lambda}.
\end{equation}
\end{rmk}
The next step will be to define the pullback
\begin{equation}\label{pullback bar(i)}
    \bar{i}^*:\Det{X_{S^s},\Lambda}^{Gal(K^t/K)}\rightarrow \Det{X_{\sigma^s},\Lambda}^{Gal(K^t/K)}.
\end{equation}
Here $X_{S^s}=\varprojlim_{a\in \mathscr{A}}X_{S_a}$. Notice that, for every $a\in \mathscr{A}$, there is a $Gal(K_a/K)$-equivariant map
\begin{equation}
    X_{\sigma_a} \xrightarrow{i_a} X_{S_a}
\end{equation}
which induces a diagram
\begin{equation}\label{diagram bar(i)}
    \bar{i}:\mathscr{A}^{op}\rightarrow \Fun(\Delta^1,\DM{S}).
\end{equation}
\begin{rmk}\label{remark about definition pullback compatible with continuous actions}
  At this point, one could consider the limit $\Det{X_{\sigma^s},\Lambda}^{Gal(K^t/K)}\rightarrow \Det{X_{S^s},\Lambda}^{Gal(K^t/K)}$ of the diagram
  \begin{equation*}
      (\ref{functoriality pushforward DM stacks})\circ(\ref{diagram bar(i)}): \mathscr{A}^{op}\rightarrow \Fun(\Delta^1,\prsr)
  \end{equation*}
  and define \ref{pullback bar(i)} as its left adjoint. However, with this definition it would not be clear that we have identities
  \begin{equation*}
      (-)^{Gal(K^t/K_a)}\circ \bar{i}^*\simeq i_a^*\circ (-)^{Gal(K^t/K_a)}.
  \end{equation*}
  Therefore, we will give another (equivalent) definition, for which the above equivalences will be tautological.
\end{rmk}

\begin{construction}\label{pullback via correspondences DM}
Consider the diagram $\bar{X}:\mathscr{A}^{op}\rightarrow \DM{S}$ given by $a\mapsto [X_{S_a}/Gal(K_a/K)]$. Also consider the closed embedding $X_{\sigma}\hookrightarrow X$. 
Then we can use the construction of Appendix \ref{appendix A} and get a diagram
\begin{equation}
    N(\mathscr{A}^{op})\rightarrow \Fun(\Delta^1,Corr(N(\DM{S}))_{F,all}).
\end{equation}
Indeed, since we put $X_{\sigma_i}=X_{S_i}\x_X X_{\sigma}$, we have that for each $a\leq b$
\begin{equation}
    [X_{\sigma_b}/Gal(K_b/K)]\simeq [X_{\sigma_a}/Gal(K_a/K)]\x_{[X_{S_a}/Gal(K_a/K)]}[X_{S_b}/Gal(K_b/K)].
\end{equation}
We can compose it with the $\oo$-functor \ref{six functor formalism with left adjoints} and get a diagram
\begin{equation}
    N(\mathscr{A}^{op})\rightarrow \Fun(\Delta^1,\prs).
\end{equation}
Finally, using the equivalences $\prs \simeq (\prsr)^{op}$ and $\Delta^1\simeq (\Delta^1)^{op}$ we get another diagram
\begin{equation}
    N(\mathscr{A}^{op})\rightarrow \Fun(\Delta^1,\prsr)
\end{equation}
and define \ref{pullback bar(i)} as its limit. 
\end{construction}
\begin{rmk}
Notice that it now follows from our definition that we have (natural) equivalences
\begin{equation*}
    (-)^{Gal(K^t/K_a)}\circ \bar{i}^*\simeq i_a^*\circ (-)^{Gal(K^t/K_a)},
\end{equation*}
as $(-)^{Gal(K^t/K_a)}: \Det{X_{S^s};\Lambda}^{Gal(K^t/K)}=\varprojlim \Det{X_{S_b}}^{Gal(K_b/K)}\rightarrow \Det{X_{S_a};\Lambda}^{Gal(K_a/K)}$ (resp. $(-)^{Gal(K^t/K_a)}: \Det{X_{\sigma^s};\Lambda}^{Gal(K^t/K)}=\varprojlim \Det{X_{\sigma_b}}^{Gal(K_b/K)}\rightarrow \Det{X_{\sigma_a};\Lambda}^{Gal(K_a/K)}$) is the projection onto the $a^{th}$ factor.
\end{rmk}
\begin{rmk} 
The construction can be performed in the language of \cite{gr17} as well.
Consider the $\oo$-functor $\ref{functoriality pushforward DM stacks}$. We will now use some notation and terminology borrowed from \cite{gr17}. As every morphism of DM stacks is representable by DM stacks, it follows from the results recalled in Section \ref{etale sheaves on Artin stacks} that for every proper morphism $f:\mathfrak{Y}\rightarrow \mathfrak{X}$, the $\oo$-functor $f_*:\Det{\mathfrak{Y};\Lambda}\rightarrow \Det{\mathfrak{X};\Lambda}$ admits a left adjoint. Moreover, $(\ref{functoriality pushforward DM stacks})$ satisfies the left Beck-Chevalley condition with respect to \emph{prop}, the class of proper morphisms. By \cite[Theorem 3.2.2, Chapter 7]{gr17}, this implies that we have a functor
\begin{equation}
    Corr(\DM{S})_{all,prop}^{prop}\rightarrow (\oo,2)\prsr.
\end{equation}
We shall restrict it to
\begin{equation}\label{functorial property *push !pull DM}
    Corr(\DM{S})_{all,prop}^{pceq}\rightarrow \prsr,
\end{equation}
where \emph{pceq}$\subseteq$\emph{prop} denotes the class of proper morphism of DM stacks such that the induced pushforward is an equivalence.
At the level of objects, this $\oo$-functor sends a DM stack $\mathfrak{X}$ to the $\oo$-category $\Det{\mathfrak{X};\Lambda}$. At the level of $1$-morphisms, it sends a correspondence 
\begin{equation*}
  \begindc{\commdiag}[12]
    \obj(-15,15)[1]{$\mathfrak{X}$}
    \obj(15,15)[2]{$\mathfrak{Y}$}
    \obj(-15,-15)[3]{$\mathfrak{Z}$}
    \mor{1}{2}{$f$}
    \mor{1}{3}{$g$}
  \enddc
\end{equation*}
to the $\oo$-functor $g_*\circ f^!:\Det{\mathfrak{Y};\Lambda}\rightarrow\Det{\mathfrak{Z};\Lambda}$.

We will now need to define a diagram
\begin{equation}
    \mathscr{A}^{op}\rightarrow \Fun(\Delta^1,Corr(\DM{S})_{all,prop}^{pceq}).
\end{equation}
In order to do so, compose the $\oo$-functors $(\DM{S})_{prop}\rightarrow (\DM{S})_{all}\rightarrow Corr(\DM{S})_{all,prop}^{prop}$ with the diagram
\begin{equation}
    \mathscr{A}^{op}\rightarrow \Fun(\Delta^1,\DM{S})
\end{equation}
induced by the morphisms $X_{\sigma}\x_{\sigma}\sigma_a\rightarrow X_{S_a}$.
This way, we get a diagram
\begin{equation}
    \mathscr{A}^{op}\rightarrow \Fun(\Delta^1,Corr(\DM{S})_{all,prop}^{prop}).
\end{equation}
For $a\leq b$ in $\mathscr{A}$, the associated morphism of correspondeces of DM stacks is
\begin{equation}\label{morphism of correspondences}
    \begindc{\commdiag}[15]
      \obj(-80,30)[0]{$\bar{X}_{\sigma_b}$}
      \obj(-30,15)[1]{$\bar{X}_{\sigma_a}\x_{\bar{X}_{S_a}}\bar{X}_{S_b}$}
      \obj(-30,-15)[3]{$\bar{X}_{S_b}$}
      \obj(30,15)[2]{$\bar{X}_{\sigma_a}$}
      \obj(30,-15)[4]{$\bar{X}_{S_a}$}
      \mor{0}{1}{$$}
      \mor{0}{3}{$i_b$}[\atright,\solidarrow]
      \mor{0}{2}{$\beta_{ab}$}
      \mor{1}{2}{$$}
      \mor{1}{3}{$$}
      \mor{3}{4}{$\beta_{ab}$}
      \mor{2}{4}{$i_a$}
    \enddc
\end{equation}

where $\bar{X}_{S_a}=[X_{S_a}/Gal(K_a/K)]$ (resp. \dots). In particular, as the morphisms $\bar{X}_{\sigma_b} \rightarrow \bar{X}_{\sigma_a}\x_{\bar{X}_{S_a}}\bar{X}_{S_b}$ are nilthickenings, we have an induced diagram
\begin{equation}
    \mathscr{A}^{op}\rightarrow \Fun(\Delta^1,Corr(\DM{S})_{all,prop}^{pceq}).
\end{equation}
By composing it with \ref{functorial property *push !pull DM}, we get a diagram
\begin{equation}
    \mathscr{A}^{op}\rightarrow \Fun(\Delta^1,\prsr).
\end{equation}
Then diagrams \ref{morphism of correspondences} take care of the functoriality (in $a,b$) of the natural equivalences
\begin{equation*}
    (i_b)_*\circ (\beta_{ab})^!\simeq (\beta_{ab})^!\circ (i_a)_* 
\end{equation*}
Finally, using the equivalences $\prsr \simeq (\prs)^{op}$ and $\Delta^1\simeq (\Delta^1)^{op}$, we get a diagram
\begin{equation}
    \mathscr{A}^{op}\rightarrow \Fun(\Delta^1,\prs).
\end{equation}
This expresses the functoriality (in $a,b$) of the equivalences
\begin{equation*}
    (-)^{Gal(K_b/K_a)}\circ i_a^*\simeq i_b^*\circ (-)^{Gal(K_b/K_a)}.
\end{equation*}
We can define \ref{pullback bar(i)} as the limit of this diagram.

Notice that, as both the inclusions $\prs \subseteq \widehat{Cat}_{\oo}$ and $\prsr \subseteq \widehat{Cat}_{\oo}$ preserve (small) limits, the source (resp. target) of this $\oo$-functor is indeed $\Det{X_{S^s};\Lambda}^{Gal(K^t/K)}$ (resp. $\Det{X_{\sigma^s};\Lambda}^{Gal(K^t/K)}$).
\end{rmk}

We are finally able to define the \emph{tame nearby cycles functor} $\Psi_{p,\eta}^{t}:\Det{X_{\eta};\Lambda}\rightarrow \Det{X_{\sigma^s};\Lambda}^{Gal(K^t/K)}$ as the composition
\begin{equation}\label{tame nearby cycles via DM stacks}
    \Det{X_{\eta};\Lambda}\simeq \Det{X_{\eta^t};\Lambda}^{Gal(K^t/K)}\xrightarrow{\bar{j}_*} \Det{X_{S^s},\Lambda}^{Gal(K^t/K)}\xrightarrow{\bar{i}^*}\Det{X_{\sigma^s};\Lambda}^{Gal(K^t/K)}.
\end{equation}
It follows immediately from the construction of $\bar{j}_*$ and $\bar{i}^*$ that 
\begin{equation}
    (-)^{I^t}\circ \Psi_{p,\eta}^t\simeq (i^{s})^*\circ(j^u)_*\circ(j^u)^*.
\end{equation}

In particular, there is a natural transformation
\begin{equation}
  (i^s)^*\rightarrow (-)^{I^t}\circ \Psi_{p,\eta}
\end{equation}
induced by the unit of the adjunction $((j^u)^*,(j^u)_*)$.
By Corollary \ref{functors to a recollement}, these data determine an $\oo$-functor
\begin{equation}
  \Psi_p^t:\Det{X;\Lambda}\rightarrow \Det{\Vs^t_p;\Lambda},
\end{equation}
the \emph{total tame vanishing cycles functor}.
We recover the \emph{tame vanishing cycles functor} 
\begin{equation}
  \Phi_p^t:\Det{X;\Lambda}\rightarrow \Det{X_{\sigma^s};\Lambda}^{Gal(K^s/K)}
\end{equation}
by applying Construction \ref{specialization functor}, i.e. as the composition
\begin{equation}
  \Det{X;\Lambda}\xrightarrow{\Psi_p^t} \Det{\Vs^t_p;\Lambda}\xrightarrow{sp}\Fun(\Delta^1,\Det{X_{\sigma^s};\Lambda}^{Gal(K^t/K)})\xrightarrow{cofiber}\Det{X_{\sigma^s};\Lambda}^{Gal(K^t/K)}.
\end{equation}
We conclude by verifying that the construction we described above agrees with the usual one.
\begin{prop}\label{comparison classical and stacky definition tame vanishing cycles}
  The $\oo$-functors $\Psi_{p,\eta}^{cl,t}$ (resp. $\Phi_{p}^{cl,t}$) and $\Psi_{p,\eta}$ (resp. $\Phi_p^t$) are equivalent.
\end{prop}
\begin{proof}
We obtain a canonical natural transformations $\Psi_{p,\eta}^{cl,t}\rightarrow \Psi_{p,\eta}^{t}$ (resp. $\Phi_{p}^{cl,t}\rightarrow \Phi_{p}^{t}$) by the canonical equivalences
\begin{equation}
    \Det{X_{\sigma_a};\Lambda}^{Gal(K_a/K)}\simeq \Det{[X_{\sigma_a}/Gal(K_a/K)];\Lambda} \hspace{0.5cm} (\text{resp. \dots})
\end{equation}
Let $\Fc \in \Det{X;\Lambda}$. Write $K^t=\varinjlim_aK_a$, where $a$ runs in the filtered set $\mathscr{A}$ of finite Galois extensions $K\subseteq K_a \subseteq K^t$. Consider the diagrams
\begin{equation}
  X_{\sigma_a}\xrightarrow{i_a}X_{S_a} \xleftarrow{j_a}X_{\eta_a}.
\end{equation} 
  Then $\Psi_{p,\eta}^{cl,t}(\Fc)\simeq \varinjlim_a i_a^*\circ(j_a)_*(\Fc_a)$, where $\Fc_a=(X_{S_a}\rightarrow X)^*\Fc$. This is an expression of the fact that $(\Psi_{p,\eta}^{cl,t}(\Fc))^{Gal(K^t/K_a)}\simeq i_a^*\circ(j_a)_*(\Fc_a)$ and that the action of $Gal(K^t/K)$ is continuous. Since $i_a^*\circ (j_a)_*$ identifies with $\bar{i}_a^*\circ(\bar{j}_a)_*$, to prove the first claim we are left to show that $\Psi_{p,\eta}^t(\Fc)\simeq \varinjlim_a \bar{i}_a^*\circ (\bar{j}_a)_*\circ (\Fc_a)$. By Lemma \ref{categorical lemma}, we only need to show that
  \begin{equation}
      (\Psi_{p,\eta}^{t}(\Gc))^{Gal(K^t/K_a)}\simeq \bar{i}_a^*\circ (\bar{j}_a)_*(\Fc_a),
  \end{equation}
  which follows immediately from the definitions of $\bar{i}^*$ and $\bar{j}_*$.
The second claim follows from
  \begin{equation}
    \Phi_{p}^{cl,t}(\Fc)\simeq \varinjlim_a\Bigl (cofib\bigl (i_a^*(\Fc_a) \rightarrow i_a^*\circ (j_a)_*\circ j_a^*(\Fc_a)\bigr )\Bigr )
  \end{equation}
  \begin{equation}
    \simeq \varinjlim_a\Bigl (cofib\bigr (\bar{i}_a^*(\Fc_a) \rightarrow \bar{i}_a^*\circ (\bar{j}_a)_*\circ \bar{j}_a^*(\Fc_a)\bigr )\Bigr )\simeq \Phi_p^t(\Fc).
  \end{equation}
\end{proof}

%% file: sections/tame_vanishing_cycles_over_ag.tex
Let $S$ denote a base scheme. In \cite{p20}, the second author introduced a formalism similar to that of inertia invariant vanishing cycles. In order to set up a complete theory, the language of stacks is needed. Indeed, we will define a sheaf of (tame) vanishing cycles for a scheme $X$ together with a global section $s$ of a line bundle $\Lc$. The datum of the triplet $(X,\Lc,s)$ is equivalent to that of a morphism $X\rightarrow \ag{S}$. In fact, such a morphism is the datum of a diagram
\begin{equation}
  X\leftarrow P \rightarrow \aff{1}{S},
\end{equation}
where $P\rightarrow X$ is a $\gm{S}$-torsor and $P \rightarrow \aff{1}{S}$ a $\gm{S}$-equivariant morphism.

The main idea is then to let $\ag{S}$ play the role of a disk. Then, $\bgm{S}$ will play the role of the center of the disk, while $S\simeq [\aff{1}{S}-\{ 0 \}/\gm{S}] \simeq \ag{S}-\bgm{S}$ that of the punctured disk. In the introduction of \cite{p20} there is a mental map with an analogy between the different disks in our, the topological and the geometric/arithmetic setting.

Let $X\rightarrow \ag{S}$ be an $X$-point of $\ag{S}$, corresponding to $(X,\Lc, s)$, where $\Lc$ is a line bundle on $X$ and $s \in \Hu^0(X,\Lc)$.
We will consider the diagram
\begin{equation}
  \begindc{\commdiag}[13]
    \obj(-50,15)[1]{$X_0$}
    \obj(0,15)[2]{$X$}
    \obj(50,15)[3]{$U_X$}
    \obj(-50,-15)[4]{$\bgm{S}$}
    \obj(0,-15)[5]{$\ag{S}$}
    \obj(50,-15)[6]{$S,$}
    \mor{1}{2}{$$}
    \mor{3}{2}{$$}
    \mor{1}{4}{$$}
    \mor{2}{5}{$$}
    \mor{3}{6}{$$}
    \mor{4}{5}{$$}
    \mor{6}{5}{$$}
  \enddc
\end{equation}
where both squares are Cartesian.

\begin{rmk}
Notice that $X_0\simeq \Vu(s)$ is the zero locus of $s$, while $U_X\simeq X-X_0\simeq X-\Vu(s)$ is its open complementary.

Moreover, let $\aff{1}{S}\rightarrow \ag{S}$ be the canonical atlas of $\ag{S}$ (i.e., the $\aff{1}{S}$-point of $\ag{S}$ corresponding to $(\aff{1}{S},\Oc_S[t],t)$). By pulling back the diagram above along this smooth morphism we get
\begin{equation}
\begindc{\commdiag}[13]
    \obj(-50,15)[1]{$P_{X_0}$}
    \obj(0,15)[2]{$P_X$}
    \obj(50,15)[3]{$P_{U_X}$}
    \obj(-50,-15)[4]{$S$}
    \obj(0,-15)[5]{$\aff{1}{S}$}
    \obj(50,-15)[6]{$\gm{S}.$}
    \mor{1}{2}{$$}
    \mor{3}{2}{$$}
    \mor{1}{4}{$$}
    \mor{2}{5}{$$}
    \mor{3}{6}{$$}
    \mor{4}{5}{$$}
    \mor{6}{5}{$$}
\enddc
\end{equation}

All morphisms in this diagram are $\gm{S}$-equivariant. Also, the canonical morphism $P_X\rightarrow X$ (resp. $P_{X_0}\rightarrow X_0$) is the $\gm{S}$-torsor $(\V(\Lc)-X)\rightarrow X$ (resp. $(\V(\Lc_{|X_0})-X_0)\rightarrow X_0$).
\end{rmk}

\begin{notation}
Let $\N_S$ denote the set of natural numbers that are invertible in $\Oc_S$. We will consider it as a filtered set, where $n\leq m$ if and only if $n$ divides $m$.
\end{notation}
\begin{construction}
Let $n\in \N_S$. Elevation to the $n^{\text{th}}$-power of the parameter of $\aff{1}{S}$ induces a morphism of Artin stacks
\begin{equation}
    \Theta^{(n)}:\ag{S}\rightarrow \ag{S}.
\end{equation}
More precisely, it is the morphism of quotient stacks induced by the commutative square
\begin{equation}
    \begindc{\commdiag}[15]
      \obj(-40,15)[1]{$\gm{S}\x \aff{1}{S}$}
      \obj(-10,15)[2]{$\aff{1}{S}$}
      \obj(-40,-15)[3]{$\gm{S}\x \aff{1}{S}$}
      \obj(-10,-15)[4]{$\aff{1}{S}$}
      \mor{1}{2}{$$}
      \mor{1}{3}{$$}
      \mor{2}{4}{$$}
      \mor{3}{4}{$$}
      \obj(10,15)[5]{$(u,x)$}
      \obj(60,15)[6]{$u\cdot x$}
      \obj(10,-15)[7]{$(u^n,x^n)$}
      \obj(60,-15)[8]{$u^n\cdot x^n=(u\cdot x)^n.$}
      \mor{5}{6}{$$}[\atright, \aplicationarrow]
      \mor{5}{7}{$$}[\atright, \aplicationarrow]
      \mor{6}{8}{$$}[\atright, \aplicationarrow]
      \mor{7}{8}{$$}[\atright, \aplicationarrow]
    \enddc
\end{equation}
As functor of points, for an $S$-scheme $X$, the induced functor
$\ag{S}(X)\rightarrow \ag{S}(X)$ sends an object $(\Lc,s)$ to $(\Lc^{\otimes n},s^{\otimes n})$.
\end{construction}
\begin{rmk}
These morphisms were first introduced (independently) in \cite{cad07} and \cite{agv08}.
\end{rmk}
\begin{rmk}
For every $n,m \in \N_S$, there are canonical equivalences of morphisms of Artin stacks
\begin{equation}
    \Theta^{(n)}\circ \Theta^{(m)}\simeq \Theta^{(nm)}\simeq \Theta^{(m)}\circ \Theta^{(n)}.
\end{equation}
\end{rmk}

\begin{defn}
Let $p:X\rightarrow \ag{S}$ be a morphism of Artin stacks. For every $n\in \N_S$ we define
\begin{equation}
    X^{(n)}:=X\x_{p,\ag{S},\Theta^{(n)}}\ag{S},
\end{equation}
\begin{equation}
    X_0^{(n)}:=X_0\x_X X^{(n)}.
\end{equation}
\end{defn}
For every $n,m\in N_S$ we dispose of a diagram of Artin stacks
\begin{equation}
    \begindc{\commdiag}[15]
      \obj(-60,15)[1]{$X_0^{(nm)}$}
      \obj(0,15)[2]{$X^{(nm)}$}
      \obj(60,15)[3]{$U_X$}
      \obj(-60,-15)[4]{$X_0^{(n)}$}
      \obj(0,-15)[5]{$X^{(n)}$}
      \obj(60,-15)[6]{$U_X.$}
      \mor{1}{2}{$i_X^{(nm)}$}
      \mor{3}{2}{$j_X^{(nm)}$}[\atright,\solidarrow]
      \mor{1}{4}{$u_{X_0,(n)}^{(nm)}$}[\atright,\solidarrow]
      \mor{2}{5}{$u_{X,(n)}^{(nm)}$}
      \mor{3}{6}{$id_{U_X}$}
      \mor{4}{5}{$i_X^{(n)}$}[\atright,\solidarrow]
      \mor{6}{5}{$j_X^{(n)}$}
    \enddc
\end{equation}
where all squares are Cartesian.

It follows immediately from the functoriality of fibre products that we dispose of diagrams of Artin stacks
\begin{equation}
    X_0^{(\bullet)}, X^{(\bullet)}: \N_S^{\textup{op}}\rightarrow \Ar{S}.
\end{equation}
Moreover, if we consider $U_X$ as a constant diagram $\N_S^{\textup{op}}\rightarrow \Ar{S}$, we have morphisms
\begin{equation}
    i_X^{(\bullet)}:X_0^{(\bullet)}\rightarrow X^{(\bullet)} \leftarrow U_X:j_X^{(\bullet)}.
\end{equation}
\begin{notation}
  Let $n\in \N_S$ and let $\Lambda$ be a ring of coefficients as in section $\S \ref{etale sheaves on Artin stacks}$. Let $\mu_n$ denote the group $\Oc_S[t]/(t^n-1)$. We will adopt the following notation:
  \begin{equation}
      \Det{X;\Lambda}^{\mu_n}:=\Det{X^{(n)};\Lambda},
  \end{equation}
  \begin{equation}
      \Det{X_0;\Lambda}^{\mu_n}:=\Det{X_0^{(n)};\Lambda}.
  \end{equation}
\end{notation}
The reason why we introduced the above notation is that $X^{(n)}$ (resp. $X_0^{(n)}$) plays the role of $[X/\mu_n]$ (resp. $[X_0/\mu_n]$)\footnote{Notice that, however, $X_0^{(n)}$ is a nilthickening of $[X_0/\mu_n]$}. Therefore, we are justified to think of complexes on $X^{(n)}$ (resp. $X_0^{(n)}$) as complexes on $X$ (resp. $X_0$) endowed with an action of $\mu_n$.
From this perspective, the adjunctions $\bigl ((u_{X_0,(n)}^{(nm)})^*,(u_{X_0,(n)}^{(nm)})_* \bigr )$, $\bigl ((u_{X,(n)}^{(nm)})^*,(u_{X,(n)}^{(nm)})_* \bigr )$ identify with 
\begin{equation}
    \textup{triv}:\Det{X_0;\Lambda}^{\mu_n}\leftrightarrows \Det{X_0;\Lambda}^{\mu_{nm}}: (-)^{\mu_m},
\end{equation}
\begin{equation}
    \textup{triv}:\Det{X;\Lambda}^{\mu_n}\leftrightarrows \Det{X;\Lambda}^{\mu_{nm}}: (-)^{\mu_m}.
\end{equation}

Let
\begin{equation}\label{functorial properties pushforward Art}
    \Ar{S}\rightarrow \prsr
\end{equation}
\begin{equation*}
    (f:\mathfrak{Y}\rightarrow \mathfrak{X})\mapsto \bigl ( f_*:\Det{\mathfrak{Y};\Lambda}\rightarrow \Det{\mathfrak{X};\Lambda} \bigr )
\end{equation*}
and consider the diagram of presentable stable $\oo$-categories
\begin{equation}
    (\ref{functorial properties pushforward Art})\circ X_0^{(\bullet)}: \N_S^{\textup{op}}\rightarrow \prsr.
\end{equation}
\begin{defn}
We define the $\oo$-category $\Det{X_0;\Lambda}^{\mu_{\oo}}$ as the limit of the diagram above:
\begin{equation}
    \Det{X_0;\Lambda}^{\mu_{\oo}}:= \varprojlim_{\N_S} \Det{X_0^{(\bullet)};\Lambda} \in \prsr.
\end{equation}
Similarly, we define the $\oo$-category $\Det{X;\Lambda}^{\muinf{}}$ as the limit of the diagram
\begin{equation}
    (\ref{functorial properties pushforward Art})\circ X^{(\bullet)}: \N_S^{\textup{op}}\rightarrow \prsr.
\end{equation}
Objects in these presentable stable $\oo$-category should be thought as complexes on $X_0$ (resp. $X$) endowed with a \emph{continuous} action of $\mu_{\oo}:=\varprojlim \mu_n$.
\end{defn}
\begin{rmk}
Notice that $\Det{X_0;\Lambda}^{\mu_{\oo}}$ is also a limit in $\widehat{\mathcal{C}at_{\oo}}$, the big $\oo$-category of $\oo$-categories. See \cite[Theorem 5.5.3.18]{lu09} and \cite[Theorem 1.1.4.4]{lu17}.
\end{rmk}
For every $n\in \N_S$ there is a short exact sequence of profinite groups
\begin{equation}
    1\rightarrow n\cdot \mu_{\oo}\rightarrow \mu_{\oo}\rightarrow \mu_n \rightarrow 1.
\end{equation}
The canonical projection $\Det{X_0;\Lambda}^{\mu_{\oo}}\rightarrow \Det{X_0;\Lambda}^{\mu_n}$ is the right adjoint in the adjunction
\begin{equation}
    \textup{triv}: \Det{X_0;\Lambda}^{\mu_n}\leftrightarrows \Det{X_0;\Lambda}^{\mu_{\oo}}: (-)^{n\cdot \mu_{\oo}}
\end{equation}
that corresponds to taking fixed points with respect to the action of the subgroup $n\cdot \mu_{\oo}\subseteq \mu_{\oo}$.

In particular, we dispose of the adjunction
\begin{equation}\label{triv muinf adjunction}
    \textup{triv}:\Det{X_0;\Lambda}\leftrightarrows \Det{X_0;\Lambda}^{\mu_{\oo}}:(-)^{\mu_{\oo}}.
\end{equation}
\begin{defn}
We define $\Det{\Vs_p^t;\Lambda}$ as the recollement determined by the adjunction \ref{triv muinf adjunction}.
\end{defn}

\begin{lmm}
$\Det{\Vs^t_p;\Lambda}$ is a presentable stable $\oo$-category.
\end{lmm}
\begin{proof}
$\Det{X_0;\Lambda}$ and $\Det{X_0;\Lambda}^{\mu_{\oo}}$ are presentable and by \cite[Proposition 5.4.7.7]{lu09} $(-)^{\mu_{\oo}}$ is an accessible functor. By \cite[Corollary 5.4.7.17]{lu09}, $\Det{\Vs_p^t;\Lambda}$ is accessible. As it also admits small colimits, $\Det{\Vs_p^t;\Lambda}$ is presentable.

It is stable by \cite[Proposition A.8.17]{lu17}.

Alternatively, by Proposition \ref{explicit description recollement} one sees that $\Det{\Vs_p^t;\Lambda}$ is a fiber product in $\prsr$, which is a complete (big) $\oo$-category (\cite[Theorem 5.5.3.18]{lu09}).
\end{proof}

\subsection*{Tame nearby cycles over \texorpdfstring{$\ag{S}$}{ag}}
We will define a tame nearby cycles functor in this $\gm{}$-equivariant setting following the lines of \ref{tame nearby cycles via DM stacks}.
\begin{construction}
Consider the diagram
\begin{equation}
    j_X^{(\bullet)}:\N_S^{op}\rightarrow \Fun(\Delta^1,\Ar{S}).
\end{equation}
We define the $\oo$-functor
\begin{equation}
    (\bar{j}_X)_*:\Det{U_X;\Lambda}\rightarrow \Det{X;\Lambda}^{\muinf{}}
\end{equation}
as the limit of the diagram
\begin{equation}
    (\ref{functorial properties pushforward Art})\circ j_X^{(\bullet)}:\N_S^{op}\rightarrow \Fun(\Delta^1,\prsr).
\end{equation}
\end{construction}

The next step will be constructing the $\oo$-functor
\begin{equation}\label{pullback bar(i)X}
    \bar{i}_X^*:\Det{X;\Lambda}^{\muinf{}}\rightarrow \Det{X_0;\Lambda}^{\muinf{}}.
\end{equation}
For the same reason we explained in Remark \ref{remark about definition pullback compatible with continuous actions}, we will not define it as the left adjoint to the limit of the diagram
\begin{equation}
    (\ref{functorial properties pushforward Art})\circ i_X^{(\bullet)}:\N_S^{op}\rightarrow \Fun(\Delta^1,\prsr),
\end{equation}
but rather we will follow the same path of Construction \ref{pullback via correspondences DM}.

\begin{construction}\label{pullback via correspondences Art}
Consider the diagram $X^{(\bullet)}:\N_S^{op}\rightarrow \Ar{S}$. Also consider the closed embedding $X_0 \hookrightarrow X$. 
Then we can use the construction of Appendix \ref{appendix A} and get a diagram
\begin{equation}
    \N_S^{op }\rightarrow \Fun(\Delta^1,Corr(N(\Ar{S}))_{F,all}).
\end{equation}
We can compose it with the $\oo$-functor \ref{six functor formalism with left adjoints} and get a diagram
\begin{equation}
    \N_S^{op}\rightarrow \Fun(\Delta^1,\prs).
\end{equation}
Finally, using the equivalences $\prs \simeq (\prsr)^{op}$ and $\Delta^1\simeq (\Delta^1)^{op}$ we get another diagram
\begin{equation}
    \N_S^{op}\rightarrow \Fun(\Delta^1,\prsr)
\end{equation}
that expresses the functoriality of the equivalences
\begin{equation*}
    (-)^{m\cdot \mu_{\oo}}\circ (i_X^{(nm)})^*\simeq (i_X^{(n)})^*\circ (-)^{m\cdot \mu_{\oo}}.
\end{equation*}
We define \ref{pullback bar(i)X} as its limit.
\end{construction}
\begin{rmk}
One can construct the $\oo$-functor \ref{pullback bar(i)X} also using the formalism of \cite{gr17}. Consider the $\oo$-functor \ref{functorial properties pushforward Art}. Moreover, let \emph{pDM} denote the class of proper morphisms representable by DM stacks in $\Ar{S}$. By the results reminded in Section \ref{etale sheaves on Artin stacks}, every morphism in \emph{pDM} satisfies the left Beck-Chevalley condition. Then, by \cite[Theorem 3.2.2, Chapter 7]{gr17} we have the functor
\begin{equation}
    Corr(\Ar{S})_{all,pDM}^{pDM}\rightarrow (\oo,2)\prsr.
\end{equation}
Similarly to what we did in Construction \ref{pullback via correspondences DM}, we shall restrict it to
\begin{equation}\label{functorial property *push !pull Ar}
    Corr(\Ar{S})_{all,pDM}^{pceq}\rightarrow \prsr,
\end{equation}
where \emph{pceq}$\subseteq$\emph{pDM} denotes the class of proper morphism representable by DM stacks such that the induced pushforward is an equivalence.
Just as in Construction \ref{pullback via correspondences DM}, at the level of objects this $\oo$-functor sends an Artin stack $\mathfrak{X}$ to the $\oo$-category $\Det{\mathfrak{X};\Lambda}$ and it sends a correspondence 
\begin{equation*}
  \begindc{\commdiag}[12]
    \obj(-15,15)[1]{$\mathfrak{X}$}
    \obj(15,15)[2]{$\mathfrak{Y}$}
    \obj(-15,-15)[3]{$\mathfrak{Z}$}
    \mor{1}{2}{$f$}
    \mor{1}{3}{$g$}
  \enddc
\end{equation*}
to the $\oo$-functor $g_*\circ f^!:\Det{\mathfrak{Y};\Lambda}\rightarrow\Det{\mathfrak{Z};\Lambda}$.

Consider the composition $\oo$-functors $(\Ar{S})_{pDM}\rightarrow (\Ar{S})_{all}\rightarrow Corr(\Ar{S})_{all,pDM}^{pDM}$ with diagram ${i'}_X^{(\bullet)}:{X'}_0^{(\bullet)}\rightarrow X^{(\bullet)}$. where ${X'}_0^{(n)}=X_0\x_{p_0,\bgm{S},\Theta_0^{(n)}}\bgm{S}$. This way, we get a diagram
\begin{equation}
    \N_S^{op}\rightarrow \Fun(\Delta^1,Corr(\Ar{S})_{all,pDM}^{pDM}).
\end{equation}
For $n, m$ in $\N_S$, the associated morphism of correspondeces of Artin stacks is
\begin{equation}\label{morphism of correspondences Artin}
    \begindc{\commdiag}[15]
      \obj(-80,30)[0]{${X'}_0^{(nm)}$}
      \obj(-30,15)[1]{${X'}_0^{(n)}\x_{X^{(n)}}X^{(nm)}$}
      \obj(-30,-15)[3]{$X^{(nm)}$}
      \obj(30,15)[2]{${X'}_0^{(n)}$}
      \obj(30,-15)[4]{$X^{(n)}$}
      \mor{0}{1}{$$}
      \mor{0}{3}{${i'}_X^{(nm)}$}[\atright,\solidarrow]
      \mor{0}{2}{$u_{X_0,(n)}^{(nm)}$}
      \mor{1}{2}{$$}
      \mor{1}{3}{$$}
      \mor{3}{4}{$u_{X,(n)}^{(nm)}$}
      \mor{2}{4}{${i'}_X^{(n)}$}
    \enddc
\end{equation}

The morphisms ${X'}_0^{(nm)} \rightarrow {X'}_0^{(n)}\x_{X^{(n)}}X^{(nm)}$ are nilthickenings and therefore they belong to \emph{pDM}. Thus, we have an induced diagram
\begin{equation}
    \N_S^{op}\rightarrow \Fun(\Delta^1,Corr(\Ar{S})_{all,pDM}^{pceq}).
\end{equation}
By composing it with \ref{functorial property *push !pull Ar}, we get a diagram
\begin{equation}
    \N_S^{op}\rightarrow \Fun(\Delta^1,\prsr),
\end{equation}
which provides us with the functoriality of the equivalences
\begin{equation*}
    ({i'}_X^{(nm)})_*\circ (u_{X,(n)}^{(nm)})^!\simeq (u_{X_0,(n)}^{(nm)})^!\circ ({i'}_X^{(n)})_*. 
\end{equation*}
Passing to left adjoints, i.e. using the equivalences $\prsr \simeq (\prs)^{op}$ and $\Delta^1\simeq (\Delta^1)^{op}$ as in Construction \ref{pullback via correspondences DM}, we get a diagram
\begin{equation}
    \N_S^{op}\rightarrow \Fun(\Delta^1,\prs)
\end{equation}
that expresses the functoriality of the equivalences
\begin{equation*}
    (-)^{m\cdot \mu_{\oo}}\circ ({i'}_X^{(nm)})^*\simeq ({i'}_X^{(n)})^*\circ (-)^{m\cdot \mu_{\oo}}.
\end{equation*}
We define \ref{pullback bar(i)X} as the limit of this diagram.

The preservation of (small) limits of both inclusions $\prs \subseteq \widehat{Cat}_{\oo}$ and $\prsr \subseteq \widehat{Cat}_{\oo}$ guarantees that the source (resp. target) of this $\oo$-functor is indeed $\Det{X;\Lambda}^{\mu_{\oo}}$ (resp. $\Det{X_0;\Lambda}^{\mu_{\oo}}$).
\end{rmk}
\begin{defn}
We define \emph{tame nearby cycles over} $\ag{S}$ as the $\oo$-functor
\begin{equation}
    \Psi_{\eta,p}^t:=\bar{i}_X^*\circ (\bar{j}_X)_*:\Det{U_X;\Lambda}\rightarrow \Det{X_0;\Lambda}^{\mu_{\oo}}.
\end{equation}
For $\Fc \in \Det{U_X;\Lambda}$, we will refer to $\Psi_{\eta,p}^t(\Fc)$ as the \emph{sheaf of tame nearby cycles with coefficients in} $\Fc$.
\end{defn}

\subsection*{Tame vanishing cycles over \texorpdfstring{$\ag{S}$}{ag}}
\begin{rmk}\label{muinf composed with tame nearby cycles}
There is an equivalence of $\oo$-functors
\begin{equation}
  (i_X^{(1)})^*\circ (j_X^{(1)})_*\simeq (-)^{\mu_{\oo}}\circ \Psi_{\eta,p}^t:\Det{U_X;\Lambda}\rightarrow \Det{X_0;\Lambda}.
\end{equation}
Indeed, by construction we have that 
\begin{equation}
    (-)^{\mu_{\oo}}\circ \bar{i}_X^{*}\simeq (i_X^{(1)})^*\circ (-)^{\mu_{\oo}}, \hspace{1cm}(-)^{\mu_{\oo}}\circ (\bar{j}_X)_{*}\simeq (j_X^{(1)})_*.
\end{equation}
\end{rmk}
The previous remark provides us with a natural transformation
\begin{equation}
  i_X^*=(i_X^{(1)})^*\rightarrow (-)^{\mu_{\oo}}\circ \Psi_{\eta,p}^t\circ (j_X^{(1)})^*\simeq (i_X^{(1)})^*\circ (j_X^{(1)})_*\circ (j_X^{(1)})^*:\Det{X;\Lambda}\rightarrow \Det{X_0;\Lambda},
\end{equation}
induced by the unit of the adjunction $((j_X^{(1)})^*,(j_X^{(1)})_*)$.
\begin{defn}
Define the $\oo$-functor of \emph{total tame vanishing cycles over $\ag{S}$}
\begin{equation}
\Psi_p^t:\Det{X;\Lambda}\rightarrow \Det{\Vs_p^t;\Lambda}
\end{equation}
to be the functor obtained by applying Corollary $\ref{functors to a recollement}$ to $(i_X^*,\Psi_{\eta,p}^t\circ (j_X^{(1)})^*,i_X^*\rightarrow (-)^{\mu_{\oo}}\circ \Psi_{\eta,p}^t\circ (j_X^{(1)})^*)$.

For $\Fc \in \Det{X;\Lambda}$, we refer to $\Psi_{p}^t(\Fc)$ as the \emph{sheaf of total tame vanishing cycles with coefficients in $\Fc$}.
\end{defn}
\begin{defn}
We define \emph{tame vanishing cycles over $\ag{S}$} as the $\oo$-functor
\begin{equation}
\Phi_p^t:\Det{X;\Lambda}\rightarrow \Det{X_0;\Lambda}^{\mu_{\oo}},
\end{equation}
obtained as the following composition
\begin{equation}
\Det{X;\Lambda}\xrightarrow{\Psi_p^t}\Det{\Vs_{p}^t;\Lambda}\xrightarrow{sp}\Fun(\Delta^1,\Det{X_0;\Lambda}^{\mu_{\oo}})\xrightarrow{cofib}\Det{X_0;\Lambda}^{\mu_{\oo}}.
\end{equation}
For $\Fc \in \Det{X;\Lambda}$, we refer to $\Phi_{p}^t(\Fc)$ as the \emph{sheaf of tame vanishing cycles with coefficients in $\Fc$}.
\end{defn}

%% file: sections/functorial_behaviour.tex
Let $f:Y\rightarrow X$ be a morphism over $\ag{S}$, i.e. assume that
\begin{equation}
  \begindc{\commdiag}[13]
    \obj(-30,15)[1]{$Y$}
    \obj(30,15)[2]{$X$}
    \obj(0,-15)[3]{$\ag{S}$}
    \mor{1}{2}{$f$}
    \mor{1}{3}{$q$}[\atright, \solidarrow]
    \mor{2}{3}{$p$}
  \enddc
\end{equation}
commutes. 
\begin{rmk}Assume that $X$ and $Y$ are $S$-schemes. Explicitly, this means that if $p$ is determined by $(X\rightarrow S, \Lc,s)$ and $q$ by $(Y\rightarrow S, \Mcc,t)$, then $f$ is an $S$-morphism and we are given an isomorphism $\alpha: \Mcc\simeq f^*\Lc$ such that $t$ corresponds to $f^*(s)$ under $\alpha$.
\end{rmk}
Then we have the following commutative diagrams
\begin{equation}
  \begindc{\commdiag}[18]
    \obj(-50,15)[1]{$Y_0^{(n)}$}
    \obj(0,15)[2]{$Y^{(n)}$}
    \obj(50,15)[3]{$U_Y$}
    \obj(-50,-15)[4]{$X_0^{(n)}$}
    \obj(0,-15)[5]{$X^{(n)}$}
    \obj(50,-15)[6]{$U_X.$}
    \mor{1}{2}{$i_Y^{(n)}$}
    \mor{3}{2}{$j_Y^{(n)}$}[\atright,\solidarrow]
    \mor{4}{5}{$i_X^{(n)}$}
    \mor{6}{5}{$j_X^{(n)}$}[\atright,\solidarrow]
    \mor{1}{4}{$f_0^{(n)}$}
    \mor{2}{5}{$f^{(n)}$}
    \mor{3}{6}{$f_U$}
    
  \enddc
\end{equation}
We will now define an adjunction 
\begin{equation}\label{star adjunction "vanishing topoi"}
  f_{\Vs^t}^*:\Det{\Vs^t_p;\Lambda}\leftrightarrows \Det{\Vs^t_q;\Lambda}: f_{\Vs^t,*}
\end{equation}
using the results we discussed in Section \ref{section recollements}.

Consider the adjunction 
\begin{equation}\label{star adjunction f_0}
  f_0^*=(f_0^{(1)})^*:\Det{X_0;\Lambda}\leftrightarrows \Det{Y_0;\Lambda}:(f_0^{(1)})_*=(f_0)_*
\end{equation}
induced by pullback/pushforward along $f_0=f_0^{(1)}:Y_0\rightarrow X_0$.

Clearly, the maps $f_0^{(n)}:Y_0^{(n)}\rightarrow X_0^{(n)}$ arrange in a morphism of diagrams of Artin stacks
\begin{equation}\label{diagram f_0^(bullet)}
  f_0^{(\bullet)}:Y_0^{(\bullet)}\rightarrow X_0^{(\bullet)} : \N_S^{op}\rightarrow \Ar{S}.
\end{equation}
In particular, we have adjunctions
\begin{equation}
  (f_0^{(n)})^*:\Det{X_0;\Lambda}^{\mu_n}\leftrightarrows \Det{Y_0;\Lambda}^{\mu_n}:(f_0^{(n)})_*.
\end{equation}

\begin{construction}\label{*pullback bar(f)_0} 

We will define an adjunction
\begin{equation}\label{star adjunction f_0tilde}
  (\bar{f}_0)^*:\Det{X_0;\Lambda}^{\mu_{\oo}}\leftrightarrows \Det{Y_0;\Lambda}^{\mu_{\oo}}:(\bar{f}_0)_*.
\end{equation}
In order to do so, consider the diagram
\begin{equation}
    X_0^{(\bullet)}:\N_S^{op}\rightarrow \Ar{S}.
\end{equation}
By considering the morphism $Y_0\rightarrow X_0$ and applying the construction in Appendix \ref{appendix A}, we get a diagram
\begin{equation}
    \N_S^{op}\x \Delta^1\rightarrow Corr(N(\Ar{S}))_{F,all}.
\end{equation}
If we then compose it with the $\oo$-functor \ref{six functor formalism with left adjoints}, we obtain a diagram
\begin{equation}
    \N_S^{op}\rightarrow \Fun(\Delta^1,\prs)
\end{equation}
which gives us homotopy coherent equivalences
\begin{equation}
    (-)^{\mu_m}\circ (f_0^{(nm)})^*\simeq (f_0^{n})^*\circ (-)^{\mu_m}.
\end{equation}
We define $(\bar{f}_0)^*:\Det{X_0;\Lambda}^{\mu_{\oo}}\rightarrow \Det{Y_0;\Lambda}^{\mu_{\oo}}$ as the limit of this diagram and $(\bar{f}_0)_*:\Det{Y_0;\Lambda}^{\mu_{\oo}}\rightarrow \Det{X_0;\Lambda}^{\mu_{\oo}}$ as its right adjoint.
\end{construction}

Then, the adjunction $(\ref{star adjunction "vanishing topoi"})$ is obtained by applying Proposition $\ref{adjunctions between recollements}$ to $(\ref{star adjunction f_0})$ and $(\ref{star adjunction f_0tilde})$.

Similarly, if $f$ is a separated morphism locally of finite type, one can define an adjunction
\begin{equation}
   f_{\Vs^t,!}:\Det{\Vs_q^t;\Lambda}\leftrightarrows \Det{\Vs_p^t;\Lambda}: f_{\Vs^t}^!.
\end{equation}
As above, one considers two adjunctions
\begin{equation}
  (f_0)_!=(f_0^{(1)})_!:\Det{Y_0;\Lambda}\leftrightarrows \Det{X_0;\Lambda}: (f_0^{(1)})^!=(f_0)^!,
\end{equation}
\begin{equation}
  (\bar{f}_0)_!:\Det{Y_0;\Lambda}^{\mu_{\oo}}\leftrightarrows \Det{X_0;\Lambda}^{\mu_{\oo}}: (\bar{f}_0)^!.
\end{equation}
\begin{construction}
We consider the same diagram
\begin{equation}
    \N_S\x \Delta^1\rightarrow Corr(N(\Ar{S}))_{F,all} 
\end{equation}
as above, i.e. the one obtained by applying the construction in Appendix \ref{appendix A} to $X_0^{(\bullet)}$ and $Y_0\rightarrow X_0$.
We composite it with \ref{six functor formalism with left adjoints} and then we pass to right adjoints, thus obtaining a diagram
\begin{equation}
    \N_S^{op}\rightarrow \Fun(\Delta^1,\prsr)
\end{equation}
which provides us with homotopy coherent equivalences
\begin{equation} 
    (-)^{\mu_m}\circ (f_0^{(nm)})^!\simeq (f_0^{n})^!\circ (-)^{\mu_m}.
\end{equation}
We define $(\bar{f}_0)^!:\Det{X_0;\Lambda}^{\mu_{\oo}}\leftrightarrows \Det{Y_0;\Lambda}^{\mu_{\oo}}$ as the limit of this diagram and $(\bar{f}_0)_!:\Det{Y_0;\Lambda}^{\mu_{\oo}}\leftrightarrows \Det{X_0;\Lambda}^{\mu_{\oo}}$ as its left adjoint.
\end{construction}

Then, the adjunction $(f_{\Vs^t,!}, f_{\Vs^t}^!)$ is defined by applying Proposition $\ref{adjunctions between recollements}$.

\subsection*{Tame vanishing cycles over \texorpdfstring{$\ag{S}$}{ag} and \texorpdfstring{$*$}{*}-pullbacks.}\label{section tame vanishing cycles over ag and *pullbacks}
The aim of the present subsection is to define a morphism of $\oo$-functors
\begin{equation}\label{vanishing cycles and star pullback}
  f_{\Vs^t}^*\circ \Psi^t_p\rightarrow \Psi_q^t\circ f^*:\Det{X;\Lambda}\rightarrow \Det{\Vs_q^t;\Lambda}.
\end{equation}
By composing both $f_{\Vs^t}^*\circ \Psi^t_p$ and $\Psi_q^t\circ f^*$ with $\Det{\Vs_q^t;\Lambda}\rightarrow \Det{Y_0;\Lambda}$ we find the $\oo$-functors $f_0^*\circ i_X^*$ and $i_Y^*\circ f^*$, which are clearly equivalent.

On the other hand, if we compose $f_{\Vs^t}^*\circ \Psi^t_p$ and $\Psi_q^t\circ f^*$ with $\Det{\Vs_q^t;\Lambda}\rightarrow \Det{Y_0;\Lambda}^{\mu_{\oo}}$ we obtain the $\oo$-functors $\bar{f}_0^*\circ \Psi_{\eta,p}^t\circ j_X^*$ and $\Psi_{\eta,q}^t\circ f^*\circ j_Y^*$ respectively. 

\begin{construction}
  By repeating Construction \ref{*pullback bar(f)_0} replacing diagram \ref{diagram f_0^(bullet)} with
  \begin{equation}
      f^{(\bullet)}:Y^{(\bullet)}\rightarrow X^{(\bullet)}
  \end{equation}
  we get the $\oo$-functor
  \begin{equation}
      \bar{f}^*:\Det{X;\Lambda}^{\mu_{\oo}}\rightarrow \Det{Y;\Lambda}^{\mu_{\oo}}
  \end{equation}
  and its right adjoint
  \begin{equation}
      \bar{f}_*:\Det{Y;\Lambda}^{\mu_{\oo}}\rightarrow \Det{X;\Lambda}^{\mu_{\oo}}.
  \end{equation}
\end{construction}
\begin{rmk}\label{concrete construction *pushforward bar(f)}
 
  One can verify that $\bar{f}_*:\Det{Y;\Lambda}^{\mu_{\oo}}\rightarrow \Det{X;\Lambda}^{\mu_{\oo}}$ can be obtained as the limit of the diagram
  \begin{equation}
      \N_S^{op}\xrightarrow{f^{(\bullet)}} \Fun(\Delta^1,\Ar{S})\xrightarrow{(\ref{functorial properties pushforward Art})} \Fun(\Delta^1,\prsr).
  \end{equation}
\end{rmk}
\begin{lmm}\label{main lemma subsection tame vanishing cycles and * pullback}
\begin{enumerate}
    \item There is an equivalence of $\oo$-functors $\bar{f}_0^*\circ \bar{i}_X^*\simeq \bar{i}_Y^*\circ \bar{f}^*$.
    \item There is a natural transformation $\alpha_f:\bar{f}^*\circ (\bar{j}_X)_*\rightarrow (\bar{j}_Y)_*\circ f_{U_Y}^*$.
    \item If $f$ is smooth, then $\alpha_f$ is an equivalence.
\end{enumerate}
\end{lmm}
\begin{proof}
\begin{enumerate}
    \item Notice that $\bar{f}_0^*\circ \bar{i}_X^*$ can be obtained by applying Construction \ref{*pullback bar(f)_0} to the diagram of Artin stacks
    \begin{equation}
        f_0^{(\bullet)}\circ i_X^{(\bullet)}:\N_S^{op} \rightarrow \Ar{S}.
    \end{equation}
    Similarly, $\bar{i}_Y^*\circ \bar{f}^*$ can be obtained by applying Construction \ref{*pullback bar(f)_0} to the diagram of Artin stacks
    \begin{equation}
        i_Y^{(\bullet)}\circ f^{(\bullet)}:\N_S^{op} \rightarrow \Ar{S}.
    \end{equation}
    As these diagrams are clearly equivalent, the claim follows.
    \item By adjunction, $\alpha_f$ corresponds to a natural transformation $(\bar{j}_X)_*\rightarrow \bar{f}_*\circ (\bar{j}_Y)_*\circ f_{U}^*$. By Remark \ref{concrete construction *pushforward bar(f)}, we see that the $\oo$-functor $\bar{f}_*\circ (\bar{j}_Y)_*$ can be obtained as the limit of the diagram
    \begin{equation}
        \N_S^{op}\xrightarrow{f^{(\bullet)}\circ j_Y^{(\bullet)}} \Fun(\Delta^1,\Ar{S})\xrightarrow{(\ref{functorial properties pushforward Art})} \Fun(\Delta^1,\prsr).
    \end{equation}
    As the diagram of Artin stacks $f^{(\bullet)}\circ j_Y^{(\bullet)}$ is equivalent to $j_X^{(\bullet)}\circ f_{U}$, we see that $\bar{f}_*\circ (\bar{j}_Y)_*$ is equivalent to the limit of the diagram
    \begin{equation}
        \N_S^{op}\xrightarrow{j_X^{(\bullet)}\circ f_{U}} \Fun(\Delta^1,\Ar{S})\xrightarrow{(\ref{functorial properties pushforward Art})} \Fun(\Delta^1,\prsr),
    \end{equation}
    i.e. to $(\bar{j}_X)_*\circ (f_{U})_*$. Then $(\bar{j}_X)_*\rightarrow \bar{f}_*\circ (\bar{j}_Y)_*\circ f_{U}^*$ is the morphism induced by the unit of the adjunction $(f_{U}^*,(f_{U})_*)$.
    \item It suffices to show that, for $\Fc \in \Det{U_Y;\Lambda}$, the morphism $\alpha_f(\Fc): \bar{f}^*\circ (\bar{j}_Y)_*(\Fc)\rightarrow (\bar{j}_X)_*\circ f_{U}^*(\Fc)$ is an equivalence. By construction,
    \begin{equation}
        \bar{f}^*\circ (\bar{j}_Y)_*(\Fc)\simeq \varinjlim_{n\in \N_S}(f^{(n)})^*\circ (j_Y^{(n)})_*(\Fc), \hspace{0.5cm} (\bar{j}_X)_*\circ f_{U_Y}^*(\Fc)\simeq \varinjlim_{n\in \N_S} (j_X^{(n)})_*\circ f_U^*(\Fc).
    \end{equation}
    Under these equivalences, $\alpha_f(\Fc)$ corresponds to the colimit of the base change morphisms $(f^{(n)})^*\circ (j_Y^{(n)})_*(\Fc)\rightarrow (j_X^{(n)})_*\circ f_U^*(\Fc)$, which are equivalences by smooth base change.
\end{enumerate}
\end{proof}
We define a natural transformation
\begin{equation}\label{tame nearby cycles and star pullback}
  \bar{f}_0^*\circ \Psi_{\eta,p}^t \circ j_X^*\rightarrow \Psi_{\eta,q}^t\circ j_Y^* \circ f^*
\end{equation}
as the morphism
\begin{equation}
    \bar{f}_0^*\circ \bar{i}_X^*\circ (\bar{j}_X)_*\circ j_X^*\simeq \bar{i}_Y^*\circ \bar{f}^*\circ (\bar{j}_X)_*\circ j_X^*\xrightarrow{\alpha_f}\bar{i}_Y^*\circ (\bar{j}_Y)_* \circ f_U^*\circ j_X^*\simeq \bar{i}_Y^*\circ (\bar{j}_Y)_* \circ j_X^*\circ f^*.
\end{equation}
\begin{lmm}
Applying $(-)^{\mu_{\oo}}$ to $(\ref{tame nearby cycles and star pullback})$ we obtain (up to homotopy) the morphism
\begin{equation}
  f_0^*\circ (i_X^{(1)})^*\circ (j_X^{(1)})_*\circ (j_X^{(1)})^* \rightarrow (i_Y^{(1)})^*\circ (j_Y^{(1)})_*\circ (j_Y^{(1)})^*\circ f^*.
\end{equation}
\end{lmm}
\begin{proof}
By construction, we have that
\begin{equation}
(-)^{\mu_{\oo}}\circ \bar{f}_0^*\simeq f_0^*\circ (-)^{\mu_{\oo}}.
\end{equation}
Then, using also Remark $\ref{muinf composed with tame nearby cycles}$, we get that 
\begin{equation}
  \bigl ( \bar{f}_0^*\circ \Psi_{\eta,p}^t \circ j_X^*\rightarrow \Psi_{\eta,q}^t\circ j_Y^* \circ f^*\bigr )^{\mu_{\oo}}\simeq f_0^*\circ (i_X^{(1)})^*\circ (j_X^{(1)})_*\circ (j_X^{(1)})^* \rightarrow (i_Y^{(1)})^*\circ (j_Y^{(1)})_*\circ (j_Y^{(1)})^*\circ f^*,
\end{equation}
where the morphism on the right is induced by the base change morphism $f^*\circ (j_X^{(1)})_*\rightarrow (j_Y^{(1)})_*\circ f_U^*$.
\end{proof}
It is then clear that we have a square $\sigma:\Delta^1\x \Delta^1 \rightarrow \Fun(\Det{X;\Lambda},\Det{Y_0;\Lambda})$
\begin{equation}
  \begindc{\commdiag}[13]
    \obj(-80,15)[1]{$f_0^*\circ i_X^*$}
    \obj(-80,-15)[2]{$i_Y^*\circ f^*$}
    \obj(80,15)[3]{$ f_0^*\circ (i_X^{(1)})^*\circ (j_X^{(1)})_*\circ (j_X^{(1)})^*$}
    \obj(80,-15)[4]{$(i_Y^{(1)})^*\circ (j_Y^{(1)})_*\circ (j_Y^{(1)})^*\circ f^*.$}
    \mor{1}{2}{$\sim$}
    \mor{1}{3}{$$}
    \mor{2}{4}{$$}
    \mor{3}{4}{$$}
  \enddc
\end{equation}

and therefore the data $(f_0^*\circ i_X^*\simeq i_Y^*\circ f^*,  \tilde{f}_0^*\circ \Psi_{\eta,p}^t \circ j_X^*\rightarrow \Psi_{\eta,q}^t\circ j_Y^* \circ f^*,\sigma)$ define the morphism $(\ref{vanishing cycles and star pullback})$.
\begin{prop}
  If $f:Y\rightarrow X$ is a smooth morphism, then $(\ref{vanishing cycles and star pullback})$ is an equivalence.
\end{prop}
\begin{proof}
As the $\oo$-functors $\Det{\Vs_q^t;\Lambda}\rightarrow \Det{Y_0;\Lambda}$ and $\Det{\Vs_q^t;\Lambda}\rightarrow \Det{Y_0}^{\mu_{\oo}}$ detect equivalences, it suffices to show that $(\ref{tame nearby cycles and star pullback})$ is an equivalence. But this is a consequence of Lemma \ref{main lemma subsection tame vanishing cycles and * pullback}.
\end{proof}
\subsection*{Tame vanishing cycles over \texorpdfstring{$\ag{S}$}{ag} and \texorpdfstring{$*$}{*}-pushforward.}
The aim of the present subsection is to define a morphism of $\oo$-functors
\begin{equation}\label{vanishing cycles and star pushforward}
  \Psi^t_p\circ f_* \rightarrow (f_{\Vs^t})_*\circ \Psi_{q}^t :\Det{Y;\Lambda}\rightarrow \Det{\Vs_p^t;\Lambda}.
\end{equation}
Composing $\Psi^t_p\circ f_*$ and $(f_{\Vs^t})_*\circ \Psi_{q}^t$ with $\Det{\Vs_q^t;\Lambda}\rightarrow \Det{Y_0;\Lambda}$ we obtain the $\oo$-functors $i_X^*\circ f_*$ and $(f_0)_*\circ i_Y^*$ respectively. Then, we clearly have a morphism
\begin{equation}
i_X^*\circ f_*\rightarrow (f_0)_*\circ i_Y^*: \Det{Y;\Lambda} \rightarrow \Det{X_0;\Lambda}.
\end{equation}
If we compose them with $\Det{\Vs_q^t;\Lambda}\rightarrow \Det{Y_0;\Lambda}^{\mu_{\oo}}$ instead, we get the $\oo$-functors $\Psi_{\eta,p}^t\circ j_X^* \circ f_*$ and $(\tilde{f}_0)_*\circ \Psi_{\eta,q}^t \circ j_Y^*$ respectively. 
\begin{lmm}\label{main lemma section tame vanishing cycles and * pushforward}
\begin{enumerate}
    \item There is an equivalence of $\oo$-functors $(\bar{i}_X)_*\circ (\bar{f}_0)_*\simeq \bar{f}_*\circ (\bar{i}_Y)_*$.
    \item There is a natural transformation $\beta_f: \bar{i}_X^*\circ \bar{f}_*\rightarrow (\bar{f}_0)_*\circ \bar{i}_Y^*$.
    \item If $f$ is a proper morphism representable by DM stacks, then $\beta_f$ is an equivalence.
\end{enumerate}
\begin{proof}
\begin{enumerate}
    \item This follows from the observation that the diagrams $i_X^{(\bullet)}\circ f^{(\bullet)}:\N_S^{op}\rightarrow \Fun(\Delta^1,\Ar{S})$ and $f_0^{(\bullet)}\circ i_Y^{(\bullet)}:\N_S^{op}\rightarrow \Fun(\Delta^1,\Ar{S})$ are equivalent and that $(\bar{i}_X)_*\circ (\bar{f}_0)_*$ (resp. $\bar{f}_*\circ (\bar{i}_Y)_*$) can be obtained as the limit of $\N_S^{op}\xrightarrow{i_X^{(\bullet)}\circ f^{(\bullet)}}\Fun(\Delta^1,\Ar{S})\xrightarrow{\ref{functorial properties pushforward Art}}\Fun(\Delta^1,\prsr)$ (resp. $\N_S^{op}\xrightarrow{f_0^{(\bullet)}\circ i_Y^{(\bullet)}}\Fun(\Delta^1,\Ar{S})\xrightarrow{\ref{functorial properties pushforward Art}}\Fun(\Delta^1,\prsr)$).
    \item The natural transformation $\beta_f$ corresponds by adjunction to a morphism $ \bar{f}_*\rightarrow (\bar{i}_X)_*\circ (\bar{f}_0)_*\circ \bar{i}_Y^*\simeq \bar{f}_*\circ (\bar{i}_Y)_*\circ \bar{i}_Y^*$, which is induced by the unit of the adjunction $(\bar{i}_Y^*,(\bar{i}_Y)_*)$.
    \item It suffices to show that, for $\Fc \in \Det{Y;\Lambda}^{\mu_{\oo}}$, the morphism $\beta_f(\Fc): \bar{i}_X^*\circ \bar{f}_*(\Fc)\rightarrow (\bar{f}_0)_*\circ \bar{i}_Y^*(\Fc)$ is an equivalence. By construction,
    \begin{equation}
        \bar{i}_X^*\circ \bar{f}_*(\Fc)(\Fc)\simeq \varinjlim_{n\in \N_S}(i_X^{(n)})^*\circ (f^{(n)})_*(\Fc), \hspace{0.5cm} (\bar{f}_0)_*\circ \bar{i}_Y^*(\Fc)\simeq \varinjlim_{n\in \N_S} (f_0^{(n)})_*\circ (i_Y^{(n)})^(\Fc).
    \end{equation}
    Under these equivalences, $\beta_f(\Fc)$ corresponds to the colimit of the base change morphisms $(i_X^{(n)})^*\circ (f^{(n)})_*(\Fc)\rightarrow (f_0^{(n)})_*\circ (i_Y^{(n)})^(\Fc)$, which are equivalences by proper base change.
\end{enumerate}
\end{proof}
\end{lmm}
We define 
\begin{equation}
    \Psi_{\eta,p}^t\circ j_X^* \circ f_* \rightarrow (\bar{f}_0)_*\circ \Psi_{\eta,q}^t \circ j_Y^*
\end{equation}
as the natural transformation
\begin{equation}
    \bar{i}_X^*\circ (\bar{j}_X)_*\circ j_X^* \circ f_*\rightarrow \bar{i}_X^*\circ (\bar{j}_X)_* \circ (f_U)_*\circ j_Y^*\simeq \bar{i}_X^*\circ \bar{f}_*\circ (\bar{j}_Y)_*\circ j_Y^*\xrightarrow{\beta_f} (\bar{f}_0)_*\circ \bar{i}_Y^*\circ (\bar{j}_Y)_*\circ j_Y^*.
\end{equation}

\begin{lmm}
The square $\tau:\partial(\Delta^1\x \Delta^1)\rightarrow \Fun(\Det{Y;\Lambda},\Det{X_0;\Lambda})$
\begin{equation}
  \begindc{\commdiag}[13]
    \obj(-50,20)[1]{$(i_X^{(1)})^*\circ f_*$}
    \obj(50,20)[2]{$(f_0)_*\circ (i_Y^{(1)})^*$}
    \obj(-50,-20)[3]{$(-)^{\mu_{\oo}}\circ \Psi_{\eta,p}^t\circ j_X^*\circ f_*$}
    \obj(50,-20)[4]{$(-)^{\mu_{\oo}}\circ (\bar{f}_0)_*\circ \Psi^t_{\eta,q}\circ j_Y^*$}
    \mor{1}{2}{$$}
    \mor{1}{3}{$$}
    \mor{2}{4}{$$}
    \mor{3}{4}{$$}
  \enddc
\end{equation}
commutes.
\end{lmm}
\begin{proof}
The bottom row in the square is homotopic to the composition
\begin{equation}
    i_X^*\circ (j_X)_*\circ j_X^*\circ f_*\rightarrow i_X^*\circ (j_X)_*\circ (f_U)_*\circ j_Y^*\simeq i_X^*\circ f_* \circ (j_Y)_*\circ j_Y^* \rightarrow (f_0)_*\circ i_Y^*\circ (j_Y)_*\circ j_Y^*.
\end{equation}
Thus, one obtains a diagram $\partial(\Delta^1\x \Delta^1)\rightarrow \Fun(\Det{Y;\Lambda},\Det{X_0;\Lambda})$
\begin{equation}
  \begindc{\commdiag}[13]
    \obj(-50,20)[1]{$(i_X^{(1)})^*\circ f_*$}
    \obj(50,20)[2]{$(f_0)_*\circ (i_Y^{(1)})^*$}
    \obj(-50,-20)[3]{$i_X^*\circ f_* \circ (j_Y)_*\circ j_Y^*$}
    \obj(50,-20)[4]{$(f_0)_*\circ i_Y^*\circ (j_Y)_*\circ j_Y^*,$}
    \mor{1}{2}{$$}
    \mor{1}{3}{$$}
    \mor{2}{4}{$$}
    \mor{3}{4}{$$}
  \enddc
\end{equation}
where the vertical arrows are induced by the unit of the adjunction $(j_Y^*,(j_Y)_*)$ and the horizontal arrows by the base change morphism $(i_X^{(1)})^*\circ f_*\rightarrow (f_0)_*\circ (i_Y^{(1)})^*$. It is then clear that the two compositions are homotpic, whence the claim.
\end{proof}
Therefore, the data $((i_X^{(1})^*\circ f_*\rightarrow (f_0)_*\circ (i_Y^{(1)})^*, \Psi_{\eta,p}^t\circ j_X^*\circ f_*\rightarrow (\tilde{f}_0)_*\circ \Psi^t_{\eta,q}\circ j_Y^*,\tau)$ define the morphism $(\ref{vanishing cycles and star pushforward})$.
\begin{prop}
If $f:Y\rightarrow X$ is a proper morphism representable by DM stacks, then $(\ref{vanishing cycles and star pushforward})$ is an equivalence.
\end{prop}
\begin{proof}
As $\Det{\Vs^t_p;\Lambda}\rightarrow \Det{X_0;\Lambda}$ and $\Det{\Vs^t_p;\Lambda}\rightarrow \Det{X_0}^{\mu_{\oo}}$ detect equivalences, it suffices to show that $(i_X^{(1})^*\circ f_*\rightarrow (f_0)_*\circ (i_Y^{(1)})^*$ and $\Psi_{\eta,p}^t\circ j_X^*\circ f_*\rightarrow (\tilde{f}_0)_*\circ \Psi^t_{\eta,q}\circ j_Y^*$ are equivalences, which follows from proper base change. For the former this is immediate, while for the latter it suffices invoke point (3) in Lemma \ref{main lemma section tame vanishing cycles and * pushforward} and the proper base change equivalence $(f_U)_*\circ j_Y^*\simeq j_X^*\circ f_*$.
\end{proof}

\subsection*{Vanishing cycles over \texorpdfstring{$\ag{S}$}{ag} and \texorpdfstring{$!$}{!}-pushforward}
In this subsection we will define a natural transformation
\begin{equation}\label{vanishing cycles and shriek pushforward}
  (f_{\Vs^t})_!\circ \Psi_{q}^t\rightarrow \Psi_{p}^t\circ f_!
\end{equation}
for every separated morphism $f:Y\rightarrow X$ locally of finite type. Composing $(f_{\Vs^t})_!\circ \Psi_{q}^t$ and $\Psi_{p}^t\circ f_!$ with the localization $\oo$-functor $\Det{\Vs^t_p;\Lambda}\rightarrow \Det{X_0;\Lambda}$ we get $(f_0^{(1)})_!\circ (i_Y^{(1)})^*$ and $(i_X^{(1)})^*\circ f_!$, that are equivalent (\cite{lz17}). On the other hand, if we compose them with the localization $\oo$-functor $\Det{\Vs^t_p;\Lambda}\rightarrow \Det{X_0;\Lambda}^{\mu_{\oo}}$, we find $(\tilde{f}_0)_!\circ \Psi_{q,\eta}^t\circ j_Y^*$ and $\Psi_{p,\eta}^t\circ j_X^*\circ f_!$ respectively. 

\begin{lmm}
\begin{enumerate}
    \item There is an equivalence of $\oo$-functors $(\bar{f}_0)_!\circ \bar{i}_Y^*\simeq \bar{i}_X^*\circ \bar{f}_!$.
    \item There is a natural transformation $\gamma_f:\bar{f}_!\circ (\bar{j}_Y)_*\rightarrow (\bar{j}_X)_*\circ (f_U)_!$.
\end{enumerate}
\end{lmm}
\begin{proof}
\begin{enumerate}
    \item Passing to right adjoints, we will show that there is an equivalence of $\oo$-functors $(\bar{i}_Y)_*\circ \bar{f}_0^!\simeq \bar{f}^!\circ (\bar{i}_X)_*$. The first $\oo$-functor can be obtained as the limit of the diagram
    \begin{equation}
        \N_S^{op}\rightarrow \Fun(\Delta^1,Corr(N(\Ar{S}))_{F,all})\rightarrow \Fun(\Delta^1,\prsr),
    \end{equation}
    where the first arrow is the diagram
    \begin{equation}
        \begindc{\commdiag}[15]
          \obj(-20,0)[1]{$n$}
          \obj(20,15)[2]{$Y_0^{(n)}$}
          \obj(50,15)[3]{$X_0^{(n)}$}
          \obj(20,-15)[4]{$Y^{(n)}.$}
          \obj(0,0)[0]{$\mapsto$}
          \mor{2}{3}{$f_0^{(n)}$}
          \mor{2}{4}{$i_Y^{(n)}$}
        \enddc
    \end{equation}
    However, this diagram of correspondences is equivalent to the composition of the diagrams
    \begin{equation}
        \begindc{\commdiag}[15]
          \obj(-80,0)[1a]{$n$}
          \obj(-40,15)[2a]{$X_0^{(n)}$}
          \obj(-10,15)[3a]{$X_0^{(n)}$}
          \obj(-40,-15)[4a]{$X^{(n)},$}
          \obj(-60,0)[0a]{$\mapsto$}
          \mor{2a}{3a}{$id$}
          \mor{2a}{4a}{$i_X^{(n)}$}
          \obj(40,0)[1b]{$n$}
          \obj(80,15)[2b]{$Y^{(n)}$}
          \obj(110,15)[3b]{$X^{(n)}$}
          \obj(80,-15)[4b]{$Y^{(n)},$}
          \obj(60,0)[0b]{$\mapsto$}
          \mor{2b}{3b}{$f^{(n)}$}
          \mor{2b}{4b}{$id$}
        \enddc
    \end{equation}
    whence the equivalence.
    \item By adjunction, we must provide a natural transformation $(\bar{j}_Y)_*\rightarrow \bar{f}^!\circ (\bar{j}_X)_*\circ (f_U)_*$. Using an analogue argument of the one in (1), we see that $\bar{f}^!\circ (\bar{j}_X)_*\simeq (\bar{j}_Y)_*\circ f_U^!$. We define the above-mentioned natural transformation as the morphism
    \begin{equation}
        (\bar{j}_Y)_*\rightarrow \bar{f}^!\circ (\bar{j}_X)_*\circ (f_U)_*\simeq (\bar{j}_Y)_*\circ f_U^!\circ (f_U)_!
    \end{equation}
    induced by the unit of the adjunction $((f_U)_!,f_U^!)$.
\end{enumerate}
\end{proof}
We define
\begin{equation}
    (\bar{f}_0)_!\circ \Psi_{\eta,q}^t\circ j_Y^* \rightarrow \Psi_{\eta,p}^t\circ j_X^*\circ f_!
\end{equation}
as the following composition:
\begin{equation}
    (\bar{f}_0)_!\circ \bar{i}_Y^*\circ (\bar{j}_Y)_*\circ j_Y^* \simeq \bar{i}_X^*\circ \bar{f}_! \circ (\bar{j}_Y)_* \circ j_Y^* \xrightarrow{\gamma_f} \bar{i}_X^*\circ (\bar{j}_Y)_* \circ (f_U)_! \circ j_Y^*\simeq  \bar{i}_X^*\circ (\bar{j}_Y)_* j_X^* \circ f_!.
\end{equation}

\begin{lmm}
The square $\tau:\partial(\Delta^1\x \Delta^1)\rightarrow \Fun(\Det{Y;\Lambda},\Det{X_0;\Lambda})$
\begin{equation}
  \begindc{\commdiag}[13]
    \obj(-50,20)[1]{$(f_0^{(1)})_!\circ (i_Y^{(1)})^*$}
    \obj(50,20)[2]{$(i_X^{(1)})^*\circ f_!$}
    \obj(-50,-20)[3]{$(-)^{\muinf{S}}\circ (\tilde{f}_0)_!\circ \Psi_{q,\eta}^t\circ j_Y^*$}
    \obj(50,-20)[4]{$(-)^{\muinf{S}}\circ \Psi_{p,\eta}^t\circ j_X^*\circ f_!$}
    \mor{1}{2}{$$}
    \mor{1}{3}{$$}
    \mor{2}{4}{$$}
    \mor{3}{4}{$$}
  \enddc
\end{equation}
commutes.
\end{lmm}
\begin{proof}
The bottom row in the square is homotopic to the composition
\begin{equation}
    (f_0)_!\circ i_Y^* \circ (j_Y)_* \circ j_Y^*\simeq i_X^*\circ f_! \circ (j_Y)_* \circ j_Y^* \rightarrow i_X^*\circ (j_X)_* \circ (f_U)_! \circ j_Y^* \simeq i_X^*\circ (j_X)_* \circ j_X^* \circ f_!.
\end{equation}
Thus, one obtains a diagram $\partial(\Delta^1\x \Delta^1)\rightarrow \Fun(\Det{Y;\Lambda},\Det{X_0;\Lambda})$
\begin{equation}
  \begindc{\commdiag}[13]
    \obj(-50,20)[1]{$(f_0)_!\circ i_Y^*$}
    \obj(50,20)[2]{$i_X^*\circ f_!$}
    \obj(-50,-20)[3]{$(f_0)_!\circ i_Y^* \circ (j_Y)_* \circ j_Y^*$}
    \obj(50,-20)[4]{$i_X^*\circ (j_X)_* \circ j_X^* \circ f_!,$}
    \mor{1}{2}{$$}
    \mor{1}{3}{$$}
    \mor{2}{4}{$$}
    \mor{3}{4}{$$}
  \enddc
\end{equation}
where the vertical arrows are induced by the unit of the adjunctions $(j_Y^*,(j_Y)_*)$ and $(j_X^*,(j_X)_*)$, while the horizontal arrows by the base change morphism. It is then clear that the two compositions are homotpic, whence the claim.
\end{proof}
The data $((f_0^{(1)})_!\circ (i_Y^{(1)})^*\simeq(i_X^{(1)})^*\circ f_!, (\bar{f}_0)_!\circ \Psi_{q,\eta}^t\circ j_Y^*\rightarrow \Psi_{p,\eta}^t\circ j_X^*\circ f_!, \tau)$ determine the morphism $(\ref{vanishing cycles and shriek pushforward})$.
\subsection*{Vanishing cycles over \texorpdfstring{$\ag{S}$}{ag} and \texorpdfstring{$!$}{!}-pullback}\label{section tame vanishing cycles over ag and !pullbacks}

In this subsection we will define a morphism
\begin{equation}\label{tame vanishing cycles and ! pullback}
  \Psi_{q}^t\circ f^!\rightarrow f_{\Vs^t}^!\circ \Psi_{p}^t.
\end{equation}

By composing both $\Psi_{q}^t\circ f^!$ and $f_{\Vs^t}^!\circ \Psi_{p}^t$ with the localization $\Det{\Vs^t_q;\Lambda}\rightarrow \Det{Y_0;\Lambda}$ we find the $\oo$-functors $(i_{Y}^{(1)})^*\circ f^!$ and $(f_{0})^!\circ (i_{X}^{(1)})^*$ respectively. We consider the morphism $(i_{Y}^{(1)})^*\circ f^!\rightarrow (f_{0})^!\circ (i_{X}^{(1)})^*$ adjoint to
\begin{equation}
f^!\rightarrow (i_{Y}^{(1)})_*\circ (f_{0})^!\circ (i_{X}^{(1)})^*\simeq f^!\circ (i_{X}^{(1)})_*\circ (i_{X}^{(1)})^*
\end{equation}
induced by the unit of the adjunction $((i_{X}^{(1)})^*,(i_{X}^{(1)})_*)$.

On the other hand, if we compose $\Psi_{q}^t\circ f^!$ and $f_{\Vs^t}^!\circ \Psi_{p}^t$ with the localization $\Det{\Vs^t_q;\Lambda}\rightarrow \Det{Y_0;\Lambda}^{\mu_{\oo}}$ we find the $\oo$-functors $\Psi_{q,\eta}^t\circ j_Y^* \circ f^!$ and $(\tilde{f}_0)^!\circ \Psi_{p,\eta}^t\circ j_X^*$ respectively. 
Notice that there is a morphism of $\oo$-functors $\bar{i}_Y^*\circ f^!\rightarrow \bar{f}_0^!\circ \bar{i}_X^*$, adjunct to the natural transfomation
\begin{equation}
    f^!\rightarrow (\bar{i}_Y)_*\circ \bar{f}_0^!\circ \bar{i}_X^*\simeq f^!\circ (\bar{i}_X)_*\circ \bar{i}_X^*
\end{equation}
induced by the unit of the adjunction $(\bar{i}_X^*,(\bar{i}_X)_*)$.

We define $\Psi_{q,\eta}^t\circ j_Y^* \circ f^!\rightarrow (\tilde{f}_0)^!\circ \Psi_{p,\eta}^t\circ j_X^*$ as the composition
\begin{equation}
     \bar{i}_Y^*\circ (\bar{j}_Y)_*\circ j_Y^* \circ f^! \rightarrow \bar{i}_Y^*\circ (\bar{j}_Y)_*\circ f_U^!\circ j_X^*\simeq \bar{i}_Y^*\circ \bar{f}^! \circ (\bar{j}_X)_*\circ j_X^*\rightarrow  \bar{f}_0^!\circ \bar{i}_X^* \circ (\bar{j}_X)_*\circ j_X^*.
\end{equation}

\begin{lmm}
The square $\tau:\partial(\Delta^1\x \Delta^1)\rightarrow \Fun(\Det{Y;\Lambda},\Det{X_0;\Lambda})$
\begin{equation}
  \begindc{\commdiag}[13]
    \obj(-50,20)[1]{$(i_{Y}^{(1)})^*\circ f^!$}
    \obj(50,20)[2]{$(f_{0})^!\circ (i_{X}^{(1)})^*$}
    \obj(-50,-20)[3]{$(-)^{\muinf{S}}\circ \Psi_{q,\eta}^t\circ j_Y^*\circ f^!$}
    \obj(50,-20)[4]{$(-)^{\muinf{S}}\circ(\bar{f}_0)^!\circ \Psi_{p,\eta}^t\circ j_X^*$}
    \mor{1}{2}{$$}
    \mor{1}{3}{$$}
    \mor{2}{4}{$$}
    \mor{3}{4}{$$}
  \enddc
\end{equation}
commutes.
\end{lmm}
\begin{proof}
This follows from the functoriality of the exchange morphisms $Ex^{*,!}$.
\end{proof}
We obtain \ref{tame vanishing cycles and ! pullback} by applying Corollary \ref{functors to a recollement} to the datum $\bigl ( (i_{Y}^{(1)})^*\circ f^!\rightarrow (f_{0})^!\circ (i_{X}^{(1)})^*, \Psi_{q,\eta}^t\circ j_Y^* \circ f^!\rightarrow (\tilde{f}_0)^!\circ \Psi_{p,\eta}^t\circ j_X^*, \tau \bigr )$.

%% file: sections/comparison_with_usual_tame_vanishing_cycles.tex
Assume that $S$ is a strictly henselian trait and keep the notation of section \S \ref{review of the theory of tame vanishing cycles}. Let $p:X\rightarrow S$ a morphism of finite type. The uniformizer $\pi$ induces a morphism $\pi:S\rightarrow \aff{1}{S}$. Let $f:X\rightarrow \ag{S}$ denote the following composition:
\begin{equation}
    X\xrightarrow{p} S \xrightarrow{\pi} \aff{1}{S} \rightarrow \ag{S}.
\end{equation}
It is a natural to compare the theory of tame vanishing cycles of $f$ with that of $p$ (defined in \cite[Exposé XIII]{sga7ii}). 
First observe that these two $\oo$-functors are indeed comparable. Notice that $X_0\simeq X_{\sigma}$.
\begin{lmm}
There exists a commutative square in $\prsr$
\begin{equation}
    \begindc{\commdiag}[13]
      \obj(-40,20)[1]{$\Det{X_0;\Lambda}$}
      \obj(40,20)[2]{$\Det{X_0;\Lambda}^{\mu_{\oo}}$}
      \obj(-40,-20)[3]{$\Det{X_0;\Lambda}$}
      \obj(40,-20)[4]{$\Det{X_0;\Lambda}^{I^t}.$}
      \mor{2}{1}{$(-)^{\mu_{\oo}}$}[\atright,\solidarrow]
      \mor{1}{3}{$id$}
      \mor{2}{4}{$\simeq$}
      \mor{4}{3}{$(-)^{I^t}$}
    \enddc
\end{equation}
In particular, there is an equivalence $\Det{\Vs^t_p;\Lambda}\simeq \Det{\Vs^t_f;\Lambda}$.
\end{lmm}
\begin{proof}
  First notice that, since we assume that $S$ is strictly henselian, $Gal(k^s/k)\simeq 0$. In particular, $\Det{\Vs^t_p;\Lambda}$ is the recollement of $\Det{X_0;\Lambda}$ and $\Det{X_0;\Lambda}^{\mu_{\oo}}$ and the second statement follows immediately from the first.
  
  We need to show that for every $n\in \N_S$ there is a commutative square
\begin{equation}
    \begindc{\commdiag}[13]
      \obj(-60,20)[1]{$\Det{X_0;\Lambda}$}
      \obj(60,20)[2]{$\Det{X_0;\Lambda}^{\mu_{n}}$}
      \obj(-60,-20)[3]{$\Det{X_0;\Lambda}$}
      \obj(60,-20)[4]{$\Det{X_0;\Lambda}^{Gal(K[\pi^{\frac{1}{n}}]/K)}.$}
      \mor{2}{1}{$(-)^{\mu_{n}}$}[\atright,\solidarrow]
      \mor{1}{3}{$id$}
      \mor{2}{4}{$\simeq$}
      \mor{4}{3}{$(-)^{Gal(K[\pi^{\frac{1}{n}}]/K)}$}
    \enddc
\end{equation}
Moreover, these squares have to fit in a diagram $\N_S\rightarrow \Fun(\Delta^1\x \Delta^1,\prsr)$.

Since $S$ is strictly henselian, the group scheme $\mu_n$ is isomorphic to $Gal(K[\pi^{\frac{1}{n}}]/K)\x S$, where the scheme on the right denotes the group scheme $\amalg_{Gal(K[\pi^{\frac{1}{n}}]/K)}S$ with the group structure induced by $Gal(K[\pi^{\frac{1}{n}}]/K)$. In particular, there are equivalences of DM stacks
\begin{equation}
    X_0^{(n)}\simeq [X/Gal(K[\pi^{\frac{1}{n}}]/K)]\x_X X_0.
\end{equation}
Moreover, they are compatible in the obvious sense, i.e. there is an equivalence of diagrams of DM stacks
\begin{equation}
    X_0^{(\bullet)}\simeq [X/Gal(K[\pi^{\frac{1}{\bullet}}]/K)]\x_X X_0:\N_S\rightarrow \DM{S}.
\end{equation}
The claim follows immediately.
\end{proof}

\begin{thm}
There is an equivalence of $\oo$-functors
\begin{equation}
  \Psi_p^{cl,t} \simeq \Psi_f^t:\Det{X;\Lambda}\rightarrow \Det{\Vs^t_p;\Lambda}\simeq \Det{\Vs^t_f;\Lambda}.
\end{equation}
\end{thm}
\begin{proof}
As the compositions of $\Psi_p^{cl,t}$ and $\Psi_f^t$ with $\Det{\Vs^t_p;\Lambda}\simeq \Det{\Vs^t_f;\Lambda}\rightarrow \Det{X_0;\Lambda}$ clearly agree, it suffices to show that there is an equivalence $\Psi_{\eta,p}^{cl,t}\circ j_X^*\simeq \Psi_{\eta,f}^t\circ j_X^*$, compatible with the morphisms $i_X^*\rightarrow \Psi_{\eta,p},^{cl,t}\circ j_X^*$ and $i_X^*\rightarrow \Psi_{\eta,f}^t\circ j_X^*$.
  This relies on the fact that, as we said above, $S$ is strictly henselian and for each $n\in \N_S$ there is an isomorphism of group $S$-schemes
  \begin{equation}
    \mu_{n}\simeq Gal(K[\pi^{\frac{1}{n}}]/K)\x S.
  \end{equation}
  Therefore we get equivalences of diagrams of DM stacks $\N_S \x (\cdot \rightarrow \cdot \leftarrow \cdot)\rightarrow \DM{S}$
  \begin{equation}
      \begindc{\commdiag}[13]
        \obj(-100,12)[1]{$X_0^{(\bullet)}$}
        \obj(0,12)[2]{$X^{(\bullet)}$}
        \obj(100,12)[3]{$X_{\eta}$}
        \obj(-100,-12)[4]{$[X/Gal(K[\pi^{\frac{1}{\bullet}}]/K)]\x_X X_0$}
        \obj(0,-12)[5]{$[X/Gal(K[\pi^{\frac{1}{\bullet}}]/K)]$}
        \obj(100,-12)[6]{$X_{\eta}$}
        \mor{1}{2}{$i_X^{(\bullet)}$}
        \mor{3}{2}{$j_X^{(\bullet)}$}[\atright, \solidarrow]
        \mor{1}{4}{$\simeq$}
        \mor{2}{5}{$\simeq$}
        \mor{3}{6}{$\simeq$}
        \mor{4}{5}{$i_X^{(\bullet)}$}
        \mor{6}{5}{$j_X^{(\bullet)}$}[\atright, \solidarrow]
      \enddc
  \end{equation}
Then the theorem follows from the comparison between the classical definition of tame nearby cycles and that using DM stacks carried out in Section \S \ref{review of the theory of tame vanishing cycles} (Proposition \ref{comparison classical and stacky definition tame vanishing cycles}).
\end{proof}

%% file: sections/tame_nearby_cycles_over_A1_and_comparison_with_Ayoub.tex
In \cite{ay07a,ay07b} J.~Ayoub introduced a formalism of tame nearby cycles over $\aff{1}{S}$ in the motivic context, which he proved to be compatible with tame nearby cycles under \'etale/$\ell$-adic realization in \cite{ay14}. 
Obviously, its construction also makes sense in the \'etale setting. 
Another formalism of tame nearby cycles over $\aff{1}{S}$ is provided by the formalism introduced here: for every $f:X\rightarrow \aff{1}{S}$, we can apply our formalism to the composition $p:X\xrightarrow{f} \aff{1}{S}\rightarrow \ag{S}$. 
In this section we will prove that these two notions of tame nearby cycles over $\aff{1}{S}$ agree in a suitable sense.
In order to do so, we shall recall Ayoub's construction, following \cite{ay14}. Let $(\Ec_S,\N_S)$ denote the following diagram of $S$-schemes:
\begin{equation}
    \Ec_S:\N_S\rightarrow \sch
\end{equation}
\begin{equation*}
    (n\rightarrow nm) \mapsto \gm{S}\xrightarrow{(-)^{m}} \gm{S}.
\end{equation*}
\begin{notation}
For an $S$-scheme $Z$, let $(Z,\N_S)$ denote the constant diagram indexed by $\N_S$ with constant value $Z$.
\end{notation}
\begin{rmk}
Notice that there is a natural transformation $(\Ec_S,\N_S)\rightarrow (\gm{S},\N_S)$ which corresponds to ($n\in \N_S$)
\begin{equation}
    \Ec_S(n)=\gm{S}\xrightarrow{(-)^n}\gm{S}.
\end{equation}
\end{rmk}
For $f:X\rightarrow S$, in \emph{loc. cit.} Ayoub considers the following diagram (of $\N_S$-diagrams):
\begin{equation}
    \begindc{\commdiag}[15]
      \obj(-60,15)[1]{$(X_0,\N_S)$}
      \obj(-20,15)[2]{$(X,\N_S)$}
      \obj(20,15)[3]{$(U_X,\N_S)$}
      \obj(60,15)[4]{$(\Ec_f,\N_S)$}
      \obj(-60,-15)[5]{$(S,\N_S)$}
      \obj(-20,-15)[6]{$(\aff{1}{S},\N_S)$}
      \obj(20,-15)[7]{$(\gm{S},\N_S)$}
      \obj(60,-15)[8]{$(\Ec_S,\N_S).$}
      \mor{1}{2}{$i_X$}
      \mor{3}{2}{$j_X$}[\atright,\solidarrow]
      \mor{4}{3}{$v_X$}[\atright,\solidarrow]
      \mor{5}{6}{$i$}
      \mor{7}{6}{$j$}[\atright,\solidarrow]
      \mor{8}{7}{$v$}[\atright,\solidarrow]
      \mor{1}{5}{$f_0$}
      \mor{2}{6}{$f$}
      \mor{3}{7}{$f_U$}
      \mor{4}{8}{$f_{\Ec}$}
    \enddc
\end{equation}
All squares are Cartesian. Then Ayoub defines the following $\oo$-functor (see \cite[Formula (100), pag. 69]{ay14}):
\begin{equation}
    \Psi^{Ay}_{f,\eta}:=(p_0)_{\#}\circ i_X^*\circ j_{X,*}\circ v_{X,*}\circ v_{X}^*\circ p_U^*:\Det{U_X;\Lambda}\rightarrow \Det{X_0;\Lambda}.
\end{equation}
Here, $p_0:(X_0,\N_S)\rightarrow (X_0,\Delta^0)$ (resp. $p_U:(U_X,\N_S)\rightarrow (U_X,\Delta^0)$ ) denotes the obvious morphism of diagrams of schemes. Notice that the diagram $(\Ec_f,\N_S)$ is $n\mapsto U_X\x_{\gm{S}}\gm{S}$, where we consider the pullback along the morphism $\gm{S}\xrightarrow{(-)^n}\gm{S}$.

Let $\text{Forget}:\Det{X_0;\Lambda}^{\mu_{\oo}}\rightarrow \Det{X_0,\Lambda}$ be the left adjoint to the $\oo$-functor $\Det{X_0;\Lambda}\rightarrow \Det{X_0;\Lambda}^{\mu_{\oo}}$ defined by the pushforwards along the canonical morphisms $X_0\rightarrow X_0^{(n)}$. Then we have the following
\begin{thm}\label{compatibility with Ayoub}
There is an equivalence of $\oo$-functors
\begin{equation}
    \Psi^{Ay}_{f,\eta}\simeq \text{Forget}\circ \Psi_{p,\eta}^t:\Det{U_X;\Lambda}\rightarrow \Det{X_0;\Lambda}.
\end{equation}
\end{thm}
\begin{proof}

As $(X_0,\N_S)$, $(X,\N_S)$ and $(U_X,\N_S)$ are constant diagrams of $S$-schemes, there are equivalences of $\oo$-categories
\begin{equation}
    \Det{(X_0,\N_S);\Lambda}\simeq \Det{X_0;\Lambda}, \hspace{0.5cm} \Det{(X,\N_S);\Lambda}\simeq \Det{X;\Lambda}, \hspace{0.5cm} \Det{(U_X,\N_S);\Lambda}\simeq \Det{U_X;\Lambda}.
\end{equation}
This means that, for any $\Fc \in \Det{U_X;\Lambda}$, by definition we have that 
\begin{equation}
    \Psi^{Ay}_{f,\eta}(\Fc) \simeq \varinjlim_{n\in \N_S} i_X^*\circ j_{X,*}^{(n)}\circ (v_X^{(n)})^*(\Fc),
\end{equation}
where $i_X:X_0\rightarrow X$, $j_X^{(n)}:U_X\x_{\gm{S}}\gm{S}=\Ec_f(n)\rightarrow X$ and $v_X^{(n)}:\gm{S}=\Ec_f(n)\rightarrow \gm{S}$. Notice that these morphisms naturally fit in the following diagram
\begin{equation}
    \begindc{\commdiag}[15]
      \obj(-100,15)[1]{$X_0$}
      \obj(-50,15)[2]{$X$}
      \obj(30,15)[3]{$U_X$}
      \obj(100,15)[4]{$\Ec_f(n)$}
      \obj(-100,-15)[5]{$X_0^{(n)}$}
      \obj(-50,-15)[6]{$X^{(n)}$}
      \obj(30,-15)[7]{$U_X^{(n)}=X^{(n)}\x_{\ag{S}}S$}
      \obj(100,-15)[8]{$U_X,$}
      \mor{1}{2}{$i_X$}
      \mor{3}{2}{$j_X$}
      \mor{4}{3}{$v_X^{(n)}$}
      \cmor((100,18)(100,22)(97,22)(25,22)(-47,22)(-50,22)(-50,18)) \pdown(25,27){$j_X^{(n)}$}
      \mor{5}{6}{$i_X^{(n)}$}
      \mor{7}{6}{$j_X$}
      \mor{8}{7}{$v_X^{(n)}$}
      \cmor((100,-18)(100,-22)(97,-22)(25,-22)(-47,-22)(-50,-22)(-50,-18)) \pup(25,-27){$j_X^{(n)}$}
      \mor{1}{5}{$q_0$}
      \mor{2}{6}{$q$}
      \mor{3}{7}{$q_U$}
      \mor{4}{8}{$q_U^{(n)}$}
    \enddc
\end{equation}
where all squares are Cartesian and all vertical morphisms are smooth. Since $\text{Forget}$ is a left adjoint,
\begin{equation}
    \text{Forget}\circ \Psi_{p,\eta}^t(\Fc)\simeq \varinjlim_{n\in \N_S}\text{Forget}\bigl ( (i_X^{(n)})^*\circ (j_X^{(n)})_*\circ (v_X^{(n)})^*(\Fc) \bigr ).
\end{equation}
By definition, $\text{Forget}:\Det{X_0;\Lambda}^{\mu_n}\rightarrow \Det{X_0;\Lambda}$ corresponds to $q_0^*$.
Moreover, $X_0\rightarrow X_0^{(n)}\x_{X^{(n)}}X$ is an homeomorphism.
Then, by the smooth base change theorem, we get that
\begin{equation}
    q_0^*\circ (i_X^{(n)})^*\circ (j_X^{(n)})_*\circ (v_X^{(n)})^*(\Fc)\simeq i_X^*\circ j_{X,*}^{(n)}\circ (v_X^{(n)})^*(\Fc).
\end{equation}
In particular, this implies that
\begin{equation}
    \text{Forget}\circ \Psi_{p,\eta}^t(\Fc)\simeq \Psi_{f,\eta}^{Ay}(\Fc).
\end{equation}
Naturality in $\Fc$ follows immediately from the fact that all the passages we have used are canonical. 
\end{proof}

%% file: sections/compatibility_with_tensor_product_and_duality.tex
Notice that since the forgetful functor $\CAlg(\prsr)\rightarrow \prsr$ is a right adjoint, and for any derived stack $Z$ the $\oo$-category $\Det{Z;\Lambda}$ is symmetric monoidal, for every $Z\rightarrow \ag{S}$, the $\oo$-categories $\Det{Z;\Lambda}^{\mu_{\oo}}$ and $\Det{Z_0;\Lambda}^{\mu_{\oo}}$ are symmetric monoidal (see \cite[Corollary 3.2.2.5]{lu17} for the existence of limits in $\CAlg(\prsr)$ and for the fact that they are preserved under the forgetful functor). 
More precisely, these $\oo$-categories are \emph{closed} symmetric monoidal. 
In fact, we can equivalently define them in the $\oo$-category $\CAlg(\prs)$.
Let $p:X\rightarrow \ag{S}$ be an Artin stack over $\ag{S}$.
In this section we will prove some compatibilities of $\Psi_{p,\eta}^t$ with external tensor products and with duality, using the compatibility of tame nearby cycles over $\ag{S}$ with the \'etale version of Ayoub's tame nearby cycles, for which such compatibilities are already established (see \cite{ay07a,ay07b,ay14,jy21}).

\subsection*{K\"unneth.}
In this subsection we shall define the K\"unneth morphism in our context and prove its compatibility with tame nearby cycles in characteristic zero.

\begin{notation}
In this section we will consider two morphisms $p:X\rightarrow \ag{S}$, $q:Y\rightarrow \ag{S}$ of Artin stacks. Let $r:Z:=X\times_{\ag{S}}Y\rightarrow \ag{S}$.
\end{notation}
\begin{defn}
We define the \emph{$\mu_{\oo}$-equvariant external tensor product} $-\boxtimes^{\mu_{\oo}} - $ as the composition
\begin{equation}
    \Det{X_0;\Lambda}^{\mu_{\oo}}\x \Det{Y_0;\Lambda}^{\mu_{\oo}}\xrightarrow{(\bar{pr}_{X_0})^*\times (\bar{pr}_{Y_0})^*} \Det{Z_0;\Lambda}^{\mu_{\oo}}\x \Det{Z_0;\Lambda}^{\mu_{\oo}} \xrightarrow{-\otimes -} \Det{Z_0;\Lambda}^{\mu_{\oo}},
\end{equation}
where $(\bar{pr}_{X_0})^*$ (resp. $(\bar{pr}_{Y_0})^*$) is the $\oo$-functor of Construction \ref{*pullback bar(f)_0} and $-\otimes -$ is the tensor product on $\Det{Z_0;\Lambda}^{\mu_{\oo}}$.
\end{defn}

We are now ready to define the \emph{K\"unneth morphism}. Let $F\in \Det{U_X;\Lambda}$ and $G\in \Det{U_Y;\Lambda}$. We have canonical morphisms
\begin{equation}
    (\bar{pr}_{X_0})^*\Psi_{\eta,p}^t(F)\rightarrow \Psi_{\eta,r}^t((pr_{U_X})^*F), \hspace{0.5cm} (\bar{pr}_{Y_0})^*\Psi_{\eta,q}^t(G)\rightarrow \Psi_{\eta,r}^t((pr_{U_Y})^*G)
\end{equation}
as constructed in Section \ref{section tame vanishing cycles over ag and *pullbacks}. Therefore, we can consider the morphism in $\Det{Z_0;\Lambda}^{\mu_{\oo}}$
\begin{equation}
    \Psi_{\eta,p}^t(F)\boxtimes^{\mu_{\oo}} \Psi_{\eta,q}^t(G) \rightarrow \Psi_{\eta,r}^t((pr_{U_X})^*F)\otimes \Psi_{\eta,r}^t((pr_{U_Y})^*G).
\end{equation}
Being a composition of lax monoidal $\oo$-functors, $\Psi_{\eta,r}^t$ is lax monoidal. Therefore, there is a canonical morphism
\begin{equation}
    \Psi_{\eta,r}^t((pr_{U_X})^*F)\otimes \Psi_{\eta,r}^t((pr_{U_Y})^*G)\rightarrow \Psi_{\eta,r}^t(F\boxtimes G),
\end{equation}
where $F\boxtimes G$ denotes the external tensor product
\begin{equation}
    \Det{U_X;\Lambda}\x \Det{U_Y;\Lambda}\rightarrow \Det{U_Z;\Lambda}.
\end{equation}
\begin{defn}
Let $F\in \Det{U_X;\Lambda}$ and $G\in \Det{U_Y;\Lambda}$. We define the \emph{K\"unneth morphism} (for $F$ and $G$) as the composition of the two morphisms above:
\begin{equation}
    \textup{K\"u}: \Psi_{\eta,p}^t(F)\boxtimes^{\mu_{\oo}} \Psi_{\eta,q}^t(G)\rightarrow \Psi_{\eta,r}^t(F\boxtimes G).
\end{equation}
\end{defn}

\begin{thm}\label{compatibility with tensor product}
Assume that $S$ is the spectrum of an algebraically closed field of characteristic $0$. Let $F\in \Det{U_X;\Lambda}$ and let $G\in \Det{U_Y;\Lambda}$ be constructible objects. The K\"unneth morphism 
\begin{equation}
    \textup{K\"u}:\Psi_{p,\eta}^t(F)\boxtimes \Psi_{q,\eta}^t(G)\rightarrow \Psi_{p,\eta}^t(F\boxtimes G )
\end{equation}
is an equivalence.
\end{thm}
\begin{proof}

Choose any smooth surjective morphism $X'\rightarrow X$, with $X'$ a scheme. The compatibility of external tensor product with pullbacks and that of tame nearby vanishing cycles with pullbacks along smooth morphisms tell us that it suffices to consider the K\"unneth morphism for $X'$. Therefore, we might assume that $X$ is a scheme.
\newline
Since the property of being an equivalence is local on $X_0$ with respect to the Zariski topology and the K\"unneth morphism is compatible with pullbacks along smooth morphisms (in particular, along open embddings), we might assume that $p:X\rightarrow \ag{S}$ factors through $f:X\rightarrow \aff{1}{S}$. Also, forgetting the continuous action of $\mu_{\oo}$ is conservative and compatible with external products and therefore we might ignore it. In this case, we are considering Ayoub's tame nearby cycles over $\aff{1}{S}$ (\'etale version) by Theorem \ref{compatibility with Ayoub}. The theorem follows by \cite[Th\'eor\`eme 3.5.17]{ay07b} and by the results of \cite{ay14}.
\end{proof}

\subsection*{Duality.} In this subsection we will investigate how tame nearby cycles over $\ag{S}$ behave with respect to duality. Let $p:X\rightarrow \ag{S}$ be a Artin stack over $\ag{S}$. Let $p_U:X\times_{\ag{S}}S=U_X\rightarrow S$.

Assume that $p$ is separated of finite type. Recall that the \emph{duality functor on} $\Det{U_X;\Lambda}$ is
\begin{equation}
    \Db_{U_X}:= \Hom_{U_X}\bigl (-,p_U^!(\sLambda_S)\bigr ):\Det{U_X;\Lambda}^{\textup{op}}\rightarrow \Det{U_X;\Lambda},
\end{equation}
where $\Hom_{U_X}$ denotes the internal hom in $\Det{U_X;\Lambda}$.

Similarly, we define the \emph{duality functor on} $\Det{X_0;\Lambda}^{\mu_{\oo}}$ as the $\oo$-functor
\begin{equation}
    \Db_{X_0}:= \Hom(-,\bar{p}_0^!(\sLambda_{\bgm{S}})):\Det{X_0;\Lambda}^{\mu_{\oo},\textup{op}}\rightarrow \Det{X_0;\Lambda}^{\mu_{\oo}}.
\end{equation}
Let $A\in \Det{U_X;\Lambda}$. There is a canonical morphism
\begin{equation}
    \Psi_{p,\eta}^t(\Db_{U_X}(A))\rightarrow \Db_{X_0}(\Psi_{p,\eta}^t(A)),
\end{equation}
defined by adjunction by the composition
\begin{equation}
    \Psi_{p,\eta}^t(\Db_{U_X}(A))\otimes \Psi_{p,\eta}^t(A)\rightarrow \Psi_{p,\eta}^t(\Db_{U_X}(A)\otimes A)\rightarrow \Psi_{p,\eta}^t(p_U^!(\sLambda_S))\rightarrow \bar{p}_0^!(\Psi_{id,\eta}^t(\sLambda_S))\simeq \bar{p}_0^!(\sLambda_{\bgm{S}}),
\end{equation}
where the first arrow is induced by the lax monoidal structure on $ \Psi_{p,\eta}^t$, the second arrow by the canonical map $\Db_{U_X}(A)\otimes A \rightarrow p_U^!(\sLambda_S)$ and the third arrow by the map constructed in Section \ref{section tame vanishing cycles over ag and !pullbacks}. The equivalence 
\begin{equation}
    \Psi_{id,\eta}^t(\sLambda_S)\simeq \sLambda_{\bgm{S}}
\end{equation}
will be proved in Section \ref{comparison with monodromy invariant vanishing cycles}.
\begin{thm}
Let $S$ be the spectrum of a field. Assume that $A$ is a constructible object in $\Det{U_X;\Lambda}$. Then
\begin{equation}
    \Psi_{p,\eta}^t\bigl ( \Db_{U_X}(A) \bigr )\simeq \Db_{X_0}\bigl ( \Psi_{p,\eta}^t(A)\bigr ).
\end{equation}
\end{thm}
\begin{proof}
The question is local on $X$, so that we consider the pullback along an open embedding $j:V\rightarrow X$. Indeed, the left hand side commutes with pullbacks along smooth morphisms. For what concerns the right hand side, it is easy to see (using the adjunctions $(\bar{j}_{0!}, \bar{j}_0^!), (\bar{j}_0^*,\bar{j}_{0*})$ and the equivalence $\bar{j}_0^*\simeq \bar{j}_0^!$) that
\begin{equation}
    \bar{j}_0^*\circ \Db_{X_0}\bigl ( \Psi_{p,\eta}^t(A)\bigr )\simeq \Db_{V_0}\bigl ( \Psi_{p\circ j,\eta}^t(A_{|V})\bigr ).
\end{equation}
Therefore, we can assume without loss of generality that $p:X\rightarrow \ag{S}$ factors through $\aff{1}{S}$.

Moreover, as $Forget: \Det{X_0;\Lambda}^{\mu_{\oo}}\rightarrow \Det{X_0;\Lambda}$ is a conservative functor, it suffices to show that the induced map
\begin{equation}
     Forget\circ \Psi_{p,\eta}^t(\Db_{U_X}(A))\rightarrow Forget \circ \Db_{X_0}(\Psi_{p,\eta}^t(A))
\end{equation}
is an equivalence. By Theorem \ref{compatibility with Ayoub}, the left hand side is the (\'etale version) of Ayoub's nearby cycles. As for the right hand side, we claim that 
\begin{equation}
    Forget \circ \Db_{X_0} \simeq \Db_{X_0} \circ Forget,
\end{equation}
where $\Db_{X_0}$ stands for $\Hom_{X_0}\bigl (-,p_0^!(\sLambda_{\bgm{S}})\bigr ):\Det{X_0;\Lambda}^{\textup{op}}\rightarrow \Det{X_0;\Lambda}$ on the right hand side of the formula. However, this follows immediately from the fact that $Forget$ is induced by the pullback along the smooth atlases $X_0\rightarrow X_0^{(n)}$ (which are smooth morphisms of relative dimension $0$) and from the fact that duality exchanges $*$-pullbacks with $!$-pullbacks and viceversa. Therefore, we are reduced to show that 
\begin{equation*}
     Forget\circ \Psi_{p,\eta}^t(\Db_{U_X}(A))\rightarrow \Db_{X_0}(Forget \circ \Psi_{p,\eta}^t(A))
\end{equation*}
is an equivalence, which follows immediately from Theorem \ref{compatibility with Ayoub} and \cite[Th\'eor\`eme 3.5.20]{ay07a,ay07b}, \cite{ay14} and \cite[Theorem 3.2.5]{jy21}
\end{proof}

%% file: sections/comparison_with_monodromy_invariant_vanishing_cycles.tex
In \cite{p20}, we introduced the so called \emph{monodromy invariant vanishing cycles} (in the $\ell$-adic setting). Let us briefly recall the definition. Let $X$ be an $S$ scheme. Fix a line bundle $\Lc$ on $X$ and a global section $s\in \Hu^0(X,\Lc)$. Consider the diagram
\begin{equation}
  \begindc{\commdiag}[18]
    \obj(-50,15)[1]{$X_0$}
    \obj(0,15)[2]{$X$}
    \obj(50,15)[3]{$U_X$}
    \obj(-50,-15)[4]{$X$}
    \obj(0,-15)[5]{$\V(\Lc)$}
    \obj(50,-15)[6]{$\Uc=\V(\Lc)-X$}
    \mor{1}{2}{$i$}
    \mor{3}{2}{$j$}
    \mor{1}{4}{$s_0$}
    \mor{2}{5}{$s$}
    \mor{3}{6}{$s_{\Uc}$}
    \mor{4}{5}{$i_0$}
    \mor{6}{5}{$j_0$}
  \enddc
\end{equation}
where $\V(\Lc)=Spec_X(Sym_{\Oc_X}(\Lc^{\vee}))$ is the total space of $\Lc$, $s:X\rightarrow \V(\Lc)$ is the global section determined by $s$ and both squares are Cartesian.
\begin{defn}
For an \'etale sheaf $\Fc \in \Det{\Uc;\Lambda}$, we define \emph{monodromy invariant vanishing cycles with coefficients in $\Fc$} as
\begin{equation}
  \Phi^{\textup{mi}}_{(X,s)}(\Fc):=cofiber \bigl ( i^*s^*j_{0*}\Fc \rightarrow i^*j_*s_{\Uc}^*\Fc \bigr ) \in \Det{X_0;\Lambda}.
\end{equation}
\end{defn}
In \cite{p20}, it is proved that \footnote{actually, we proved what follows in the $\ell$-adic setting, but the proofs work \emph{mutatis mutandis} with finite coefficients too.} 
\begin{equation}
  \Phi^{\textup{mi}}_{(X,s)}(\sLambda_X)\simeq cofiber \Bigl ( cofiber \bigl (c_1(\Lc_{|X_0}):\sLambda_{X_0}(-1)[-2]\rightarrow \sLambda_{X_0} \bigr) \rightarrow i^*j_*\sLambda_{U_X} \Bigr ),
\end{equation}
where the map $ cofiber \bigl (c_1(\Lc_{|X_0}):\sLambda_{X_0}(-1)[-2]\rightarrow \sLambda_{X_0} \bigr) \rightarrow i^*j_*\sLambda_{U_X} $ is induced by the commutative diagram
\begin{equation}
  \begindc{\commdiag}[18]
    \obj(-60,15)[1]{$\sLambda_{X_0}(-1)[-2]$}
    \obj(60,15)[2]{$\sLambda_{X_0}$}
    \obj(-60,-15)[3]{$i^*j_*\sLambda_{U_X}(-1)[-2]$}
    \obj(60,-15)[4]{$i^*j_*\sLambda_{U_X}$}
    \mor{1}{2}{$c_1(\Lc_{|X_0})$}
    \mor{2}{4}{$$}
    \mor{1}{3}{$$}
    \mor{3}{4}{$i^*j_*(c_1(\Lc_{|U_X}))\sim 0$}
  \enddc
\end{equation}
  
Notice that the squares in the diagram
\begin{equation}
  \begindc{\commdiag}[18]
    \obj(-50,15)[1]{$X_0$}
    \obj(0,15)[2]{$X$}
    \obj(50,15)[3]{$U_X$}
    \obj(-50,-15)[4]{$\bgm{S}$}
    \obj(0,-15)[5]{$\ag{S}$}
    \obj(50,-15)[6]{$S$}
    \mor{1}{2}{$i$}
    \mor{3}{2}{$j$}
    \mor{1}{4}{$p_0$}
    \mor{2}{5}{$p$}
    \mor{3}{6}{$p_{\Uc}$}
    \mor{4}{5}{$i_0$}
    \mor{6}{5}{$j_0$}
  \enddc
\end{equation}
are Cartesian as well. Here $p$ is the morphism associated to $(\Lc,s)$ and $p_0$ the morphism associated to $\Lc_{|X_0}$. 

We have already observed (Remark $\ref{muinf composed with tame nearby cycles}$) that 
\begin{equation}
  (\Psi^t_{\eta,p}(\sLambda_{X}))^{\mu_{\oo}}\simeq i^*j_*\sLambda_{U_X}.
\end{equation}
We will now show that 
\begin{equation}
  (triv(\sLambda_{X_0}))^{\mu_{\oo}}\simeq cofiber \bigl (c_1(\Lc_{|X_0}):\sLambda_{X_0}(-1)[-2]\rightarrow \sLambda_{X_0} \bigr ).
\end{equation}
\subsection*{The base}
In this subsection we will investigate the construction of tame vanishing cycles applied to the identity morphism on $\ag{S}$, i.e. the $\oo$-functor
\begin{equation}
  \Psi^t_{id}:\Det{\ag{S};\Lambda}\rightarrow \Det{\Vs_{id}^t;\Lambda}.
\end{equation}

For any $n,m \in \N_S$ there are adjunctions
\begin{equation}
  triv^{\oo}_n:\Det{\bgm{S};\Lambda}^{\mu_n}\leftrightarrows \Det{\bgm{S};\Lambda}:(-)^{n\cdot \mu_{\oo}},
\end{equation}
\begin{equation}
  triv^{nm}_{n}:\Det{\bgm{S};\Lambda}^{\mu_n}\leftrightarrows \Det{\bgm{S};\Lambda}^{\mu_{nm}}:(-)^{m\cdot \mu_{nm}}.
\end{equation}
These $\oo$-functors are compatible in the following sense:
\begin{equation}
  triv^{\oo}_{nm}\circ triv^{nm}_n\simeq  triv^{\oo}_n,
\end{equation}
\begin{equation}
  (-)^{m\cdot \mu_{nm}}\circ (-)^{nm\cdot \mu_{\oo}}\simeq (-)^{n\cdot \mu_{\oo}}.
\end{equation}

By definition, $\Psi_{id,\eta}^t(\sLambda_{\ag{S}})$ is the object of $\Det{\bgm{S};\Lambda}^{\mu_{\oo}}=\varprojlim_n\Det{\bgm{S};\Lambda}^{\mu_{n}}$ corresponding to $\{ (\Psi_{id,\eta}^t)^{(n)}(\sLambda_{\ag{S}}) \}$.
Using the above mentioned equivalences of $\oo$-functors, for any $n,m\in \N_S$ there is a morphism
\begin{equation}\label{eq:tame_identity1}
triv^{\oo}_n\bigl ( (\Psi_{id,\eta}^t(\sLambda_{\ag{S}}))^{n\cdot \mu_{\oo}} \bigr )\rightarrow triv^{\oo}_{nm}\bigl ( (\Psi_{id,\eta}^t(\sLambda_{\ag{S}}))^{nm\cdot \mu_{\oo}} \bigr )
\end{equation}
induced by the isomorphism
\begin{equation}
triv^{\oo}_n\bigl ( (\Psi_{id,\eta}^t(\sLambda_{\ag{S}}))^{n\cdot \mu_{\oo}} \bigr )\simeq triv^{\oo}_{nm}\Bigl ( triv^{nm}_n \bigl ( (\Psi_{id,\eta}^t(\sLambda_{\ag{S}}))^{nm\cdot \mu_{\oo}} \bigr )^{\mu_m} \Bigr )
\end{equation}
and the morphism
\begin{equation}
 triv^{\oo}_{nm}\Bigl ( triv^{nm}_n \bigl ( (\Psi_{id,\eta}^t(\sLambda_{\ag{S}}))^{nm\cdot \mu_{\oo}} \bigr )^{\mu_m} \Bigr )\rightarrow triv^{\oo}_{nm}\bigl ( (\Psi_{id,\eta}^t(\sLambda_{\ag{S}}))^{nm\cdot \mu_{\oo}} \bigr )
\end{equation}
induced by the counit of the adjunction $(triv^{nm}_n,(-)^{\mu_m})$. As these counits and the equivalences
\begin{equation}
  triv^{nmp}_{nm}\circ triv^{nm}_m\simeq  triv^{nmp}_m, \hspace{0.5cm}  (-)^{\mu_m}\circ (-)^{\mu_n}\simeq (-)^{\mu_{nm}}
\end{equation}
are natural, these morhisms \eqref{eq:tame_identity1} assemble in a diagram
\begin{equation}
  \N_S\rightarrow \Det{\bgm;\Lambda}^{\mu_{\oo}}.
\end{equation}
\begin{lmm}
The following equivalence holds in $\Det{\bgm{S};\Lambda}^{\mu_{\oo}}$:
\begin{equation}
  triv(\sLambda_{\bgm{S}})\simeq \varinjlim \Bigl (triv^{\oo}_n\bigl ( (\Psi_{id,\eta}^t(\sLambda_{S}))^{n\cdot \mu_{\oo}} \bigr )\rightarrow triv^{\oo}_{nm}\bigl ( (\Psi_{id,\eta}^t(\sLambda_{S}))^{nm\cdot \mu_{\oo}} \bigr )  \Bigr ).
\end{equation}
\end{lmm}
\begin{proof}
By definition, we have that $(\Psi_{id,\eta}^t(\sLambda_{S}))^{n\cdot \mu_{\oo}}=(\Psi_{id,\eta}^t)^{(n)}(\sLambda_{S})$. Consider the following diagram of Artin stacks
\begin{equation}
  \begindc{\commdiag}[20]
    \obj(-140,15)[1]{$\bgm{S}$}
    \obj(-80,15)[2]{$\bgm{S}\x_{i,\ag{S},\Theta^{(n)}}\ag{S}$}
    \obj(0,15)[3]{$\ag{S}$}
    \obj(60,15)[4]{$S$}
    \obj(-80,-15)[5]{$\bgm{S}$}
    \obj(0,-15)[6]{$\ag{S}$}
    \obj(60,-15)[7]{$S.$}
    \mor{1}{2}{$\beta$}
    \mor{2}{3}{$$}
    \mor{4}{3}{$j$}
    \mor{5}{6}{$i$}
    \mor{7}{6}{$j$}
    \mor{1}{5}{$\Theta_0^{(n)}$}
    \mor{2}{5}{$$}
    \mor{3}{6}{$\Theta^{(n)}$}
    \mor{4}{7}{$id_S$}
    \cmor((-140,19)(-140,21)(-138,23)(-40,23)(-2,23)(0,21)(0,19)) \pdown(-40,27){$i$}
  \enddc
\end{equation}
Both squares are Cartesian and $\beta$ is a closed embedding with empty open complementary.
\begin{equation}
  (\Psi_{id,\eta}^t)^{(n)}(\sLambda_S):=i^*j_*(\sLambda_S)
\end{equation}
Then the distinguished triangle
\begin{equation}
  i^!(\sLambda_{\ag{S}})\rightarrow \sLambda_{\bgm{S}}\rightarrow i^*j_*(\sLambda_S)
\end{equation}
and absolute purity for Artin stacks (\cite[Theorem 3.28]{k19})
\begin{equation}
  \sLambda_{\bgm{S}}(-1)[-2]\simeq i^!(\sLambda_{\ag{S}})
\end{equation}
imply that 
\begin{equation}
  i^*j_*(\sLambda_S)\simeq cofiber \bigl ( \sLambda_{\bgm{S}}(-1)[-2]\xrightarrow{c_1} \sLambda_{\bgm{S}} \bigr ).
\end{equation}
Here $c_1\in \Hu_{et}^2(\bgm{S},\sLambda_{\bgm{S}}(1))$ is the \emph{universal first Chern class}. Since $\Theta_0^{(n)}:\bgm{S}\rightarrow \bgm{S}$ corresponds to $n\cdot U$, where $U$ is the \emph{universal line bundle}, we see that
\begin{equation}
    (\Theta_0^{(n)})^*\circ (\Theta_0^{(n)})_*(c_1)\simeq n\cdot c_1:\sLambda_{\bgm{S}}(-1)[-2]\rightarrow \sLambda_{\bgm{S}}.
\end{equation}
Under these equivalences, the morphisms $triv^{nm}_n\bigl ( (\Psi_{id,\eta}^t(\sLambda_{\ag{S}}))^{n\cdot \mu_{\oo}} \bigr )\rightarrow \bigl ( (\Psi_{id,\eta}^t(\sLambda_{\ag{S}}))^{nm\cdot \mu_{\oo}} \bigr )$ corresponds to (\cite[Proposition 10.3]{ay14})
\begin{equation}
  \begindc{\commdiag}[20]
    \obj(-50,15)[1]{$cofiber \Bigl ( triv^{nm}_m(\sLambda_{\bgm{S}}(-1)[-2]) $}
    \obj(50,15)[2]{$triv^{nm}_m(\sLambda_{\bgm{S}}) \Bigr )$}
    \obj(-50,-15)[3]{$cofiber \Bigl ( \sLambda_{\bgm{S}}(-1)[-2]$}
    \obj(50,-15)[4]{$\sLambda_{\bgm{S}} \Bigr ).$}
    \mor{1}{2}{$m\cdot c_1$}
    \mor{1}{3}{$m$}
    \mor{2}{4}{$1$}
    \mor{3}{4}{$c_1$}
  \enddc
\end{equation}
It is then clear that, as $\Lambda$ is torsion,
\begin{equation}
  triv(\sLambda_{S})\simeq \varinjlim \Bigl (triv^{\oo}_n\bigl ( (\Psi_{id,\eta}^t(\sLambda_{\ag{S}}))^{n\cdot \mu_{\oo}} \bigr )\rightarrow triv^{\oo}_{nm}\bigl ( (\Psi_{id,\eta}^t(\sLambda_{\ag{S}}))^{nm\cdot \mu_{\oo}} \bigr )  \Bigr ).
\end{equation}
\end{proof}
We will now need the following categorical lemma.
\begin{lmm}\label{categorical lemma}
Let $\Cc_{\bullet}:\N_S^{\textup{op}}\rightarrow \textup{\textbf{Pr}}^{\textup{R}}_{\textup{st}}$ be a diagram in the (big) $\oo$-category of stable, presentable $\oo$-categories with right adjoints. For every $n,m\in \N_S$, there is an adjunction
\begin{equation}
i_{n,nm}:\Cc_n\leftrightarrows \Cc_{nm}:t_{n,nm}.
\end{equation}
They are compatible in the obvious sense. Let $\Cc:=\varprojlim_{n}\Cc_n$ be the limit of such diagram. For every $n\in \N_S$ we have an adjunction
\begin{equation}
i_{n}:\Cc_n\leftrightarrows \Cc:t_{n}.
\end{equation}
These are also compatible in an obvious way.

For any $F\in \Cc$, let $F_n:=t_n(F)$. As $t_{n,nm}(F_{nm})\simeq F_n$, the counits of the adjunctions $(i_{n,nm},t_{n,nm})$ induce morphisms
\begin{equation}
  i_n(F_n)\rightarrow i_{nm}(F_{nm})
\end{equation}
that define a diagram $\N_S \rightarrow \Cc$. The canonical morphism
\begin{equation}
  \varinjlim_ni_n(F_n)\rightarrow F
\end{equation}
is an equivalence.
\end{lmm}
\begin{proof}
Let $G\in \Cc$. Then
\begin{equation}
\Map_{\Cc}(\varinjlim_ni_n(F_n),G)\simeq \varprojlim_n \Map_{\Cc}(i_n(F_n),G)\simeq \varprojlim \Map_{\Cc_n}(F_n,G_n),
\end{equation}
where $G_n:=t_n(G)$. $\Map_{\Cc}(F,G)$ is the pullback of the following diagram
\begin{equation}
 \begindc{\commdiag}[20]
   \obj(30,15)[1]{$\Fun(\Delta^1,\Cc)$}
   \obj(-30,-15)[2]{$\Delta^0$}
   \obj(30,-15)[3]{$\Cc \x \Cc.$}
   \mor{1}{3}{$(ev_0,ev_1)$}
   \mor{2}{3}{$(F,G)$}
 \enddc
\end{equation}
As $\Fun(\Delta^1,\Cc)\simeq \varprojlim_n \Fun(\Delta^1,\Cc_n)$ and $\Cc \x \Cc \simeq \varprojlim_n \Cc_n\x\Cc_n$, this diagram is the limit of the diagrams
\begin{equation}
 \begindc{\commdiag}[20]
   \obj(30,15)[1]{$\Fun(\Delta^1,\Cc_n)$}
   \obj(-30,-15)[2]{$\Delta^0$}
   \obj(30,-15)[3]{$\Cc_n \x \Cc_n,$}
   \mor{1}{3}{$(ev_0,ev_1)$}
   \mor{2}{3}{$(F_n,G_n)$}
 \enddc
\end{equation}
whose pullbacks are $\Map_{\Cc_n}(F_n,G_n)$. Since limits commute with limits, we find that 
\begin{equation}
\Map_{\Cc}(F,G)\simeq \varprojlim_n \Map_{\Cc_n}(F_n,G_n).
\end{equation}
The claim follows.
\end{proof}
In particular, we get that
\begin{cor}
The following equivalences hold in $\Det{\bgm{S};\Lambda}^{\mu_{\oo}}$:
\begin{equation}
  triv(\sLambda_{\bgm{S}})\simeq \varinjlim_n (\Psi_{id,\eta}^t)^{(n)}(\sLambda_{S})\simeq \Psi_{id,\eta}^t(\sLambda_{S}).
\end{equation}
\end{cor}
\begin{prop}
With the same notation of the previous sections,
\begin{equation}
  (triv(\sLambda_{X_0}))^{\mu_{\oo}}\simeq cofiber \bigl (:\sLambda_{X_0}(-1)[-2]\xrightarrow{c_1(\Lc_{|X_0})} \sLambda_{X_0} \bigr ) \in \Det{X_0;\Lambda}.
\end{equation}
\end{prop}
\begin{proof}
As $\sLambda_{X_0}\simeq p_0^*(\sLambda_{\bgm{S}})$, we have that
\begin{equation}
  (triv(\sLambda_{X_0}))^{\mu_{\oo}}\simeq (\tilde{p}_0^*(\sLambda_{\bgm{S}}))^{\mu_{\oo}}\simeq p_0^*(\sLambda_{\bgm{S}})^{\mu_{\oo}}.
\end{equation}
The previous corollary and the fact that $(-)^{\mu_{\oo}}\circ \Psi_{id,\eta}^t\simeq (\Psi_{id,\eta}^t)^{(1)}\simeq i_0^*\circ (j_0)_*$ thus imply that
\begin{equation}
  (triv(\sLambda_{X_0}))^{\mu_{\oo}}\simeq p_0^*\bigl ( \sLambda_{\bgm{S}}(-1)[-2]\xrightarrow{c_1} \sLambda_{\bgm{S}} \bigr ),
\end{equation}
where $c_1 \in \Hu_{\text{\'et}}^2(\bgm{S},\sLambda_{\bgm{S}}(1))$ is the \emph{universal first Chern class}. The claim follows as 
\begin{equation}
  p_0^*(c_1)=c_1(\Lc_{|X_0})\in \Hu^2_{\text{\'et}}(X_0,\sLambda_{X_0}(1)).
\end{equation}

Indeed, we have the following commutative diagram:
\begin{equation}
  \begindc{\commdiag}[20]
    \obj(-40,15)[1]{$\Hu^{1}_{\text{\'et}}(X_0,\mathbb{G}_{\textup{m}})$}
    \obj(40,15)[2]{$\Hu^{2}_{\text{\'et}}(X_0,\sLambda_{}(1))$}
    \obj(-40,-15)[3]{$\Hu^{1}_{\text{\'et}}(\bgm{S},\mathbb{G}_{\textup{m}})$}
    \obj(40,-15)[4]{$\Hu^2_{\text{\'et}}(\bgm{S},\sLambda(1)).$}
    \mor{1}{2}{$\partial_{X_0}$}
    \mor{3}{1}{$p_0^*$}
    \mor{4}{2}{$p_0^*$}
    \mor{3}{4}{$\partial_{\bgm{S}}$}
  \enddc
\end{equation}
Moreover,
\begin{equation}
\Z \subseteq \Hu^{1}_{\text{\'et}}(\bgm{S},\mathbb{G}_{\textup{m}})
\end{equation}
The element $U \in \Hu^{1}_{\text{\'et}}(\bgm{S},\mathbb{G}_{\textup{m}})$ corresponding to $1$ is the \emph{universal $\gm{S}$-torsor} and is such that
\begin{equation}
  p_0^*(U)=[\Lc_{|X_0}], \hspace{0.5cm} \partial_{\bgm{S}}(U)=c_1.
\end{equation}

\end{proof}

\begin{thm}
There is an equivalence
\begin{equation}
  (\Phi_p^t(\sLambda_{X}))^{\mu_{\oo}}\simeq \Phi^{\textup{mi}}_{(X,s)}(\sLambda_{X}) \in \Det{X_0;\Lambda}.
\end{equation}
\end{thm} 
\begin{proof}
We have that
\begin{equation}
\Phi^{\textup{mi}}_{(X,s)}(\sLambda_X)\simeq cofiber \Bigl ( cofiber \bigl (c_1(\Lc_{|X_0}):\sLambda_{X_0}(-1)[-2]\rightarrow \sLambda_{X_0} \bigr) \xrightarrow{f} i^*j_*\sLambda_{U_X} \Bigr ),
\end{equation}
where $f$ is the morphism induced by the counit $\tilde{f}:\sLambda_{X_{0}}\rightarrow i^*j_*\sLambda_{U_X}$ ($\tilde{f}\circ c_1(\Lc_{|X_0})\sim 0$). 

On the other hand, the previous lemmata and propositions imply that
\begin{equation}
(\Phi_p^t(\sLambda_{X}))^{\mu_{\oo}}\simeq cofiber \Bigl ( cofiber \bigl (c_1(\Lc_{|X_0}):\sLambda_{X_0}(-1)[-2]\rightarrow \sLambda_{X_0} \bigr) \xrightarrow{sp^{\mu_{\oo}}} i^*j_*\sLambda_{U_X} \Bigr ).
\end{equation}
Therefore, we only need to show that the morphisms $f$ and $sp^{\mu_{\oo}}$ are homotopic. With the same notation as in Construction \ref{bar specialization functor}, we see that the composition
\begin{equation}
  \sLambda_{X_0}\rightarrow cofiber \bigl (c_1(\Lc_{|X_0}):\sLambda_{X_0}(-1)[-2]\rightarrow \sLambda_{X_0} \bigr) \xrightarrow{sp^{\mu_{\oo}}} i^*j_*\sLambda_{U_X}
\end{equation}
is homotopic to $\bar{sp}$, i.e. to $\tilde{f}$. The claim follows.
\end{proof}
\begin{rmk}
Using the usual limiting procedure, we find monodromy-invariant ($\Ql{}$-adic) vanishing cycles (see \cite[Definition 4.2.6]{p20}) as 
\begin{equation}
  \Phi^{\textup{mi}}_{(X,s)}(\Ql{})\simeq \Phi_{p}^t(\Ql{,X})^{\mu_{\oo}}.
\end{equation} 
\end{rmk}

%% file: sections/appendix_A.tex
Let $\Ac$ be a filtered set with a minimal element $e$. 
We consider $\Ac$ as an ordinary category by letting $Hom_{\Ac}(a,a')$ be the punctual set if $a\leq a'$ and the empty set otherwise.
Let $X:\Ac^{op}\rightarrow \Cc$ be a diagram in a $(2,1)$-category.
We assume that $\Cc$ has $2$-fibre products.
Suppose that there is a distinguished class of $1$-morphisms $F$ in $\Cc$, which is stable under composition, pullbacks and contains all identities, and that $X(a\leq a')\in F$ for every $a\leq a' \in \Ac$.
Moreover, suppose that we are given a $1$-morphism in $\Cc$
\begin{equation}
    i(e):X_0(e)\rightarrow X(e).
\end{equation}
We shall put $X_0(a):=X_0(e)\x_{X(e)}X(a)$ for all $a\in \Ac$.
It follows from the properties of $2$-fibre products that we get a diagram $X_0:\Ac^{op}\rightarrow \Cc$, such that $X_0(a\leq a')\in F$ for every $a\leq a' \in \Ac$.

Starting from these data, we will construct a diagram
\begin{equation}
    D:N(\Ac^{op})\x \Delta^1 \rightarrow \delta^*_{2,\{2\}}\bigl ( N(\Cc) \bigr )^{cart}_{F,all}.
\end{equation}
Here $N$ denotes the simplicial nerve of \cite{lu09} and $\Cc$ is regarded as a simplicial category by applying the nerve construction on all hom categories.
The simplicial set $\delta^*_{2,\{2\}}\bigl ( N(\Cc) \bigr )^{cart}_{F,all}$ is the one defined in \cite{lz17}.

We will define $D$ explicitly.

\begin{rmk}[$n$-simplexes in $\delta^*_{2,\{2\}}\bigl ( N(\Cc) \bigr )^{cart}_{F,all}$]
By definition, an $n$-simplex $\Delta^n\rightarrow \delta^*_{2,\{2\}}\bigl ( N(\Cc) \bigr )^{cart}_{F,all}$ is a morphism of simplicial sets
\begin{equation}
    \Delta^n\x (\Delta^n)^{op}\rightarrow N(\Cc)
\end{equation}
such that vertical morphisms lie in $F$ and all squares are Cartesian.
By definition of the simplicial nerve, this corresponds to simplicial functor
\begin{equation}
    S:\Cb[\Delta^n\x (\Delta^n)^{op}]\rightarrow \Cc
\end{equation}
such that $S((i,j)\leq (i',j))$ is in $F$ for all $i\leq i'$ and all $j$ and such that each diagram
\begin{equation}\label{pullback condition}
    \begindc{\commdiag}[15]
      \obj(-25,15)[1]{$S(i,j)$}
      \obj(25,15)[2]{$S(i,j')$}
      \obj(-25,-15)[3]{$S(i',j)$}
      \obj(25,-15)[4]{$S(i',j')$}
      \mor{1}{2}{$$}
      \mor{1}{3}{$$}
      \mor{2}{4}{$$}
      \mor{3}{4}{$$}
    \enddc
\end{equation}
is a $2$-fibre product, where $i\leq i', j\geq j'$.
\end{rmk}
\begin{construction}
Let $n\geq 0$. We will now construct a function
\begin{equation}
    D_n: Hom_{Set_{\Delta}}(\Delta^n,N(\Ac^{op})\x \Delta^1)\rightarrow Hom_{Set_{\Delta}}(\Delta^n,\delta^*_{2,\{2\}}\bigl ( N(\Cc) \bigr )^{cart}_{F,all}).
\end{equation}
Let $(a,s)=(a_0\leq a_1 \leq \dots \leq a_n, s_0\leq \dots \leq s_n)$ be an $n$-simplex of $N(\Ac)\x \Delta^1$. 
We need to define a simplicial functor
\begin{equation}
    D_n(a,s):\Cb[\Delta^n\x (\Delta^n)^{op}]\rightarrow \Cc.
\end{equation}
Let $(i,j)$ be an object of $\Cb[\Delta^n\x (\Delta^n)^{op}]$. 
Then we put
\begin{equation}
    D_n(a,s)(i,j)=\begin{cases}
                   X(a_i) \hspace{0.5cm} \text{if } s_j=1;\\
                   X_0(a_i) \hspace{0.5cm} \text{if } s_j=0.
                \end{cases}
\end{equation}
For each $(i,j),(i',j')\in \Cb[\Delta^n\x (\Delta^n)^{op}]$, we need to define a morphism of simplicial sets
\begin{equation}
    \Map_{\Cb[\Delta^n\x (\Delta^n)^{op}]}\bigl ((i,j),(i',j')\bigr )=N(P_{(i,j),(i',j')})\rightarrow N\bigl (Hom_{\Cc}(D_n(a,s)(i,j),D_n(a,s)(i',j'))\bigr ),
\end{equation}
that is, we have to define a functor
\begin{equation}
    P_{(i,j),(i',j')}\rightarrow Hom_{\Cc}(D_n(a,s)(i,j),D_n(a,s)(i',j')).
\end{equation}
Here $P_{(i,j),(i',j')}$ stands for the partially ordered set defined in \cite[Example 1.1.5.9]{lu09}.
As a first approximation, notice that the morphism of diagrams $i:X_0\rightarrow X$ provides us with a diagram
\begin{equation}
    \begindc{\commdiag}[15]
       \obj(-80,30)[11]{$X_0(a_n)$}
       \obj(-40,30)[21]{$X_0(a_n)$}
       \obj(0,30)[31]{$\dots$}
       \obj(40,30)[41]{$X_0(a_n)$}
       \obj(80,30)[51]{$X(a_n)$}
       \obj(120,30)[61]{$X(a_n)$}
       \obj(160,30)[71]{$\dots$}
       \obj(200,30)[81]{$X(a_n)$}
       \mor{11}{21}{$id$}
       \mor{21}{31}{$id$}
       \mor{31}{41}{$id$}
       \mor{41}{51}{$i(a_n)$}
       \mor{51}{61}{$id$}
       \mor{61}{71}{$id$}
       \mor{71}{81}{$id$}
       
       \obj(-80,10)[12]{$X_0(a_{n-1})$}
       \obj(-40,10)[22]{$X_0(a_{n-1})$}
       \obj(0,10)[32]{$\dots$}
       \obj(40,10)[42]{$X_0(a_{n-1})$}
       \obj(80,10)[52]{$X(a_{n-1})$}
       \obj(120,10)[62]{$X(a_{n-})$}
       \obj(160,10)[72]{$\dots$}
       \obj(200,10)[82]{$X(a_{n-1})$}
       \mor{12}{22}{$id$}
       \mor{22}{32}{$id$}
       \mor{32}{42}{$id$}
       \mor{42}{52}{$i(a_{n-1})$}[\atright, \solidarrow]
       \mor{52}{62}{$id$}
       \mor{62}{72}{$id$}
       \mor{72}{82}{$id$}
       
       \obj(-80,-10)[13]{$\vdots$}
       \obj(-40,-10)[23]{$\vdots$}
       \obj(0,-10)[33]{$\vdots$}
       \obj(40,-10)[43]{$\vdots$}
       \obj(80,-10)[53]{$\vdots$}
       \obj(120,-10)[63]{$\vdots$}
       \obj(160,-10)[73]{$\vdots$}
       \obj(200,-10)[83]{$\vdots$}
       
       \obj(-80,-30)[14]{$X_0(a_{0})$}
       \obj(-40,-30)[24]{$X_0(a_{0})$}
       \obj(0,-30)[34]{$\dots$}
       \obj(40,-30)[44]{$X_0(a_{0})$}
       \obj(80,-30)[54]{$X(a_{0})$}
       \obj(120,-30)[64]{$X(a_{0})$}
       \obj(160,-30)[74]{$\dots$}
       \obj(200,-30)[84]{$X(a_{0}).$}
       \mor{14}{24}{$id$}
       \mor{24}{34}{$id$}
       \mor{34}{44}{$id$}
       \mor{44}{54}{$i(a_{0})$}
       \mor{54}{64}{$id$}
       \mor{64}{74}{$id$}
       \mor{74}{84}{$id$}
       
       \mor{11}{12}{$$}
       \mor{12}{13}{$$}
       \mor{13}{14}{$$}
       \mor{21}{22}{$$}
       \mor{22}{23}{$$}
       \mor{23}{24}{$$}
       \mor{41}{42}{$$}
       \mor{42}{43}{$$}
       \mor{43}{44}{$$}
       \mor{51}{52}{$$}
       \mor{52}{53}{$$}
       \mor{53}{54}{$$}
       \mor{61}{62}{$$}
       \mor{62}{63}{$$}
       \mor{63}{64}{$$}
       \mor{81}{82}{$$}
       \mor{82}{83}{$$}
       \mor{83}{84}{$$}
    \enddc
\end{equation}
The vertical maps are those defined by $X_0$ and $X$. The fact that condition (\ref{pullback condition}) holds is obvious from the definitions. The $2$-cells are not depicted. However, the $2$-cells providing isomorphisms
\begin{equation}
    X_0(a\leq a'')\simeq X_0(a'\leq a'')\circ X_0(a\leq a') \hspace{0.5cm} (\text{resp. } X(a\leq a'')\simeq X(a'\leq a'')\circ X(a\leq a') )
\end{equation}
are provided by the functor $X_0$ (resp. $X$). 
Moreover, we attach the identity $2$-cells in the squares that do not belong to the column where the $i(a_j)$ appear and $2$-cells provided by the definition of $2$-fibre products (see eg \cite[Tag 003O]{stproj}) to those that belong to that column. These are compatible in an obvious sense that we will not make explicit here (by the universal property of $2$-fibre products).

Each $A\in P_{(i,j),(i',j')}$ admits an unique "minimal path", that is a "path" $(i,j)\rightarrow (i',j')$ such that there are no elements of $A$ under this "path".
At the level of objects, we define $P_{(i,j),(i',j')}\rightarrow Hom_{\Cc}(D_n(a,s)(i,j),D_n(a,s)(i',j'))$ by sending $A$ to the composition of the images of the "$1$-steps" of this minimal path. If $A\subseteq A'$, then the minimal path of $A'$ lies below that of $A$ and we define the morphism
\begin{equation}
    D_n(a,s)(A\subseteq A'):D_n(a,s)(A)\rightarrow D_n(a)(A')
\end{equation}
by using the $2$-cells mentioned above.
The fact that 
\begin{equation}
    D_n(a,s)(A\subseteq A'')=D_n(a,s)(A'\subseteq A'')\circ D_n(a,s)(A\subseteq A')
\end{equation}
follows from the compatibility of the $2$-cells. The equality
\begin{equation}
    D_n(a,s)(A\subseteq A)=id
\end{equation}
is obvious.

This defines the functors $P_{(i,j),(i',j')}\rightarrow Hom_{\Cc}(D_n(a,s)(i,j),D_n(a,s)(i',j'))$, as desired.

Finally, notice that the square
\begin{equation}
    \begindc{\commdiag}[20]
      \obj(-70,10)[1]{$P_{(i,j),(i',j')}\x P_{(i',j'),(i'',j'')}$}
      \obj(70,10)[2]{$Hom_{\Cc}(D_n(a,s)(i,j),D_n(a,s)(i',j'))\x Hom_{\Cc}(D_n(a,s)(i',j'),D_n(a,s)(i'',j''))$}
      \obj(-70,-10)[3]{$P_{(i,j),(i'',j'')}$}
      \obj(70,-10)[4]{$Hom_{\Cc}(D_n(a,s)(i,j),D_n(a,s)(i'',j''))$}
      \mor{1}{2}{$$}
      \mor{1}{3}{$\cup$}
      \mor{2}{4}{$\circ_{\Cc}$}
      \mor{3}{4}{$$}
    \enddc
\end{equation}
is commutative on the nose: at the level of objects, this follows immediately from the observation that if $A\in P_{(i,j),(i',j')}$ and $A'\in P_{(i',j'),(i'',j'')}$, then the "minimal path" of $A\cup A'$ is the join of the minimal paths of $A$ and $A'$. At the level of morphisms, it follows once again from the compatibility of the $2$-cells in the diagram above. This provides us with the desired simplicial functor $D_n(a,s):\Cb[\Delta^n\x (\Delta^n)^{op}]\rightarrow \Cc$.
\end{construction}
\begin{lmm}
The functions $D_n: Hom_{Set_{\Delta}}(\Delta^n,N(\Ac^{op})\x \Delta^1)\rightarrow Hom_{Set_{\Delta}}(\Delta^n,\delta^*_{2,\{2\}}\bigl ( N(\Cc) \bigr )^{cart}_{F,all})$ assemble in a morphism of simplicial sets
\begin{equation}
    D:N(\Ac^{op})\x \Delta^1 \rightarrow \delta^*_{2,\{2\}}\bigl ( N(\Cc) \bigr )^{cart}_{F,all}.
\end{equation}
\end{lmm}
\begin{proof}
We need to show that the $D_n$'s are compatible with the face and degeneracy maps of $N(\Ac^{op})\x \Delta^1$ and $\delta^*_{2,\{2\}}\bigl ( N(\Cc) \bigr )^{cart}_{F,all}$. The description of face and degeneracy maps in $N(\Ac^{op})\x \Delta^1$ being well known, we will limit to describe those of the $\delta^*_{2,\{2\}}\bigl ( N(\Cc) \bigr )^{cart}_{F,all}$. For $S:\Cb[\Delta^n\x (\Delta^n)^op]\rightarrow N(\Cc)$, its image along the degeneracy map
\begin{equation}
    s_n^i: Hom_{Set_{\Delta}}(\Delta^n,\delta^*_{2,\{2\}}\bigl ( N(\Cc) \bigr )^{cart}_{F,all})\rightarrow Hom_{Set_{\Delta}}(\Delta^{n+1},\delta^*_{2,\{2\}}\bigl ( N(\Cc) \bigr )^{cart}_{F,all})
\end{equation}
corresponds to inserting a row
\begin{equation}
    \begindc{\commdiag}[20]
       \obj(-80,10)[1]{$S(i,n)$}
       \obj(-40,10)[2]{$S(i,n-1)$}
       \obj(0,10)[3]{$\dots$}
       \obj(40,10)[4]{$S(i,1)$}
       \obj(80,10)[5]{$S(i,0)$}
       \mor{1}{2}{$$}
       \mor{2}{3}{$$}
       \mor{3}{4}{$$}
       \mor{4}{5}{$$}
       
       \obj(-80,-10)[6]{$S(i,n)$}
       \obj(-40,-10)[7]{$S(i,n-1)$}
       \obj(0,-10)[8]{$\dots$}
       \obj(40,-10)[9]{$S(i,1)$}
       \obj(80,-10)[10]{$S(i,0),$}
       \mor{6}{7}{$$}
       \mor{7}{8}{$$}
       \mor{8}{9}{$$}
       \mor{9}{10}{$$}
       
       \mor{1}{6}{$id$}
       \mor{2}{7}{$id$}
       \mor{4}{9}{$id$}
       \mor{5}{10}{$id$}
    \enddc
\end{equation}
a column 
\begin{equation}
    \begindc{\commdiag}[20]
       \obj(-20,40)[1]{$S(0,i)$}
       \obj(-20,20)[2]{$S(1,i)$}
       \obj(-20,0)[3]{$\dots$}
       \obj(-20,-20)[4]{$S(n-1,i)$}
       \obj(-20,-40)[5]{$S(n-1,0)$}
       \mor{1}{2}{$$}
       \mor{2}{3}{$$}
       \mor{3}{4}{$$}
       \mor{4}{5}{$$}
       
       \obj(20,40)[6]{$S(0,i)$}
       \obj(20,20)[7]{$S(1,i)$}
       \obj(20,0)[8]{$\dots$}
       \obj(20,-20)[9]{$S(n-1,i)$}
       \obj(20,-40)[10]{$S(n,i),$}
       \mor{6}{7}{$$}
       \mor{7}{8}{$$}
       \mor{8}{9}{$$}
       \mor{9}{10}{$$}
       
       \mor{1}{6}{$id$}
       \mor{2}{7}{$id$}
       \mor{4}{9}{$id$}
       \mor{5}{10}{$id$}
    \enddc
\end{equation}
and the obvious $2$-cells. 
It is then clear from this description and that of the $D_n(a_0\leq \dots \leq a_i\leq a_i \leq \dots \leq a_n,s_0\leq \dots \leq s_i\leq s_i \leq \dots \leq s_n)$ that the $D_n$'s are compatible with the degeneracy maps.

On the other hand, the face maps
\begin{equation}
    d_n^i: Hom_{Set_{\Delta}}(\Delta^n,\delta^*_{2,\{2\}}\bigl ( N(\Cc) \bigr )^{cart}_{F,all})\rightarrow Hom_{Set_{\Delta}}(\Delta^{n-1},\delta^*_{2,\{2\}}\bigl ( N(\Cc) \bigr )^{cart}_{F,all})
\end{equation}
correspond to taking compositions $S(k+1,i)\circ S(k,i)$ and $S(i,j-1)\circ S(i,j)$ to get rid if the $i^{th}$ row and column in the diagram $S:\Delta^n\x(\Delta^n)^{op}\rightarrow \Cc$. Then it is also obvious that the $D_n$'s are compatible with the face maps as well.
\end{proof}
\begin{cor}
The above construction provides us with an $\oo$-functor
\begin{equation}
    N(\Ac^{op})\x \Delta^1 \rightarrow Corr( N(\Cc) )_{F,all}.
\end{equation}
\end{cor}
\begin{proof}
This follows immediately from \cite[Section 6.1]{lz17}.
\end{proof}